\documentclass[a4paper,11pt,twoside,reqno,onehalfspacing,final]{amsart}

\usepackage{tikz-cd}

\usepackage{fontspec}

\usepackage{mathtools}
\usepackage{xfrac}

\usepackage{unicode-math}

\let\llbracket\undefined
\let\rrbracket\undefined
\newcommand{\llbracket}{\symbol{"27E6}}
\newcommand{\rrbracket}{\symbol{"27E7}}
\newcommand{\lightning}{\symbol{"26A1}}

\usepackage{setspace}
\usepackage[skip=\smallskipamount]{parskip} 
\usepackage{geometry}
\usepackage{enumitem}
\setlist[enumerate,1]{label=\textbf{(\roman*)},ref=(\roman*)}

\usepackage{xcolor}
\usepackage[color]{showkeys}
\colorlet{refkey}{orange}
\colorlet{labelkey}{teal}
\usepackage[right,mathlines]{lineno}

\usepackage[backend=biber,bibencoding=utf8,style=numeric,
maxbibnames=99,minbibnames=99,
url=false,doi=true,eprint=true,isbn=false,
hyperref=auto,backref=false]{biblatex}
\usepackage[unicode=true,bookmarks=true,final=true,pdfstartview={FitH},
colorlinks=true,linktoc=page,linkcolor=red,citecolor=blue]{hyperref}

\reversemarginpar

\let\oldmarginpar\marginpar
\renewcommand\marginpar[1]{\-\oldmarginpar[\raggedleft\footnotesize #1]{\raggedright\footnotesize #1}}

\newtheorem{theorem}{Theorem}[section]
\newtheorem{lemma}[theorem]{Lemma}
\newtheorem{proposition}[theorem]{Proposition}
\newtheorem{corollary}[theorem]{Corollary}

\newtheorem{remark}[theorem]{Remark}
\newtheorem{example}[theorem]{Example}

\newtheoremstyle{intro}
  {12pt}
  {12pt}
  {\itshape}
  {}
  {\bfseries}
  {.}
  {.5em}
  {}

\theoremstyle{intro}
\newtheorem{introthm}{Theorem}

\usepackage{hsinchu}

\title[Duflo--Kontsevich isomorphisms for calculi of dg manifolds]{Duflo--Kontsevich-type isomorphisms\\ for Tamarkin--Tsygan calculi of dg manifolds}

\thanks{This research was partially supported by NSF grant DMS-2302447,
Simons Foundation grant MP-TSM-00002272,
and Taiwan's NSTC grants 112-2115-M-007-016-MY3 and 115-2628-M-007-007-MY4.
This work was also supported by the Research Institute for Mathematical Sciences (RIMS),
an International Joint Usage/Research Center located at Kyoto University.}

\author{Hsuan-Yi Liao}
\address{Department of Mathematics, National Tsing Hua University}
\email{hyliao@math.nthu.edu.tw}

\author{Mathieu Stiénon}
\address{Department of Mathematics, Pennsylvania State University}
\email{stienon@psu.edu}

\author{Ping Xu}
\address{Department of Mathematics, Pennsylvania State University}
\email{ping@math.psu.edu}

\begin{document}


\begin{abstract}
We establish a Duflo--Kontsevich-type theorem for differential graded (dg) manifolds. Specifically, we prove that the Hochschild--Kostant--Rosenberg maps twisted by the square root of the Todd class realize an isomorphism between the Tamarkin--Tsygan calculus and the Cartan calculus of the dg manifold. At the level of cohomology, this confirms the Kontsevich--Shoikhet conjecture formulated in \cite{arXiv:math/9812009}.
\end{abstract}

\maketitle

\tableofcontents


\section{Introduction}

The main goal of this paper is to establish Duflo--Kontsevich-type isomorphisms for Tamarkin--Tsygan calculi on dg manifolds.

Dg manifolds have recently become increasingly important. A dg manifold is defined as a $\ZZ$-graded manifold equipped with a homological vector field, i.e. a vector field $Q$ of degree $+1$ satisfying $\schouten{Q}{Q}=0$. Originally arising in physics as BRST operators for describing gauge symmetries, dg manifolds (also known as $Q$-manifolds) have since become prominent in mathematical physics, notably in the AKSZ formalism \cite{MR1432574,MR1230027}, and also appear naturally in geometry, Lie theory, and mathematical physics. Standard examples include $L_\infty$ algebras, foliations, complex manifolds, and derived intersections.

\paragraph{\sc $L_\infty$ algebras}
Given a finite-dimensional Lie algebra $\frakg$, we write $\frakg[1]$ to denote the dg manifold whose algebra of functions is $C^\infty(\frakg[1]):=\Lambda^\bullet\frakg^\vee$ and whose homological vector field is the Chevalley--Eilenberg differential $Q:=d_{\CE}$. This construction admits an “up to homotopy” version: given a $\ZZ$-graded finite-dimensional vector space $\frakg=\bigoplus_{i\in\ZZ}\frakg_i$, the graded manifold $\frakg[1]$ is a dg manifold, i.e. admits a homological vector field, if and only if the graded vector space $\frakg$ carries the structure of a curved $L_\infty$ algebra.

\paragraph{\sc Foliations and complex manifolds}
Given an integrable distribution $F\subseteq T_M^{\KK}:=T_M\otimes\KK$ ($\KK$ being either $\RR$ or $\CC$), we write $F[1]$ to denote the dg manifold whose algebra of functions is $C^\infty(F[1])=\Omega_F^\bullet:=\sections{\Lambda^\bullet F^\vee}$ and whose homological vector field is the leafwise de~Rham differential $Q:=d_{\DR}^F$. In particular, a complex manifold $X$ determines an integrable distribution $T^{0,1}_X\subset T_X^{\CC}$. Then $(T^{0,1}_X[1],\bar{\partial})$ is a dg manifold: the algebra of functions is $C^\infty(T^{0,1}_X[1])=\Omega^{0,\bullet}(X)$ and the homological vector field is the Dolbeault operator $\bar{\partial}$.

\paragraph{\sc Derived intersections}
Given a smooth section $s$ of a vector bundle $E\to M$, we write $E[-1]$ to denote the dg manifold whose algebra of functions is $C^\infty (E[-1])=\sections{\Lambda^{-\bullet}E\dual}$ and whose homological vector field is $Q=\iota_s$, the interior product with $s$. This dg manifold can be regarded as a smooth model for the (possibly singular) intersection of $s$ with the zero section of the vector bundle $E$, and is often called a “derived intersection” or a \emph{quasi-smooth derived manifold} \cite{MR4735657}.

The notion of dg manifolds provides a highly effective framework for studying the differential geometry of manifolds with singularities.
Dg manifolds of amplitude $[-n,-1]$ are equivalent to \emph{Lie $n$-algebroids} \cite{MR2768006}, which can be viewed as infinitesimal counterparts of higher groupoids.
Furthermore, dg manifolds of amplitude $[1,n]$ correspond to derived manifolds \cite{MR4735657,arXiv:2307.08179}.
Consequently, general dg manifolds with amplitude $[-m,n]$ are capable of encoding both stacky and derived singularities in differential geometry.
Notably, dg manifolds also serve as a powerful tool for analyzing singular foliations \cite{MR4164730}.

A calculus, as defined by Tamarkin and Tsygan \cite{MR1783778,MR2986860}, consists of a pair $(\aTheta^{\bullet},\Xi_\bullet)$ such that $\aTheta^{\bullet}$ is a Gerstenhaber algebra, $\Xi_\bullet$ is both an associative algebra module (with action $\iI$) over the algebra $\aTheta^{\bullet}$ and a Lie algebra module (with action $\iL$) over the Lie algebra $(\aTheta[1])^{\bullet}$, and $\Xi_\bullet$ is equipped with a degree $-1$ operator $d$ satisfying $d^2=0$ and identities analogous to those of the classical Cartan calculus. For a smooth manifold $M$, let $\aTheta^{\bullet} := \Tpoly{\bullet}(M) = \Gamma(\Lambda^\bullet T_M)$ be the space of \emph{polyvector fields}, and let $\Xi_\bullet := \Omega^{-\bullet}(M)$ be the space of \emph{differential forms} given the opposite grading. The pair $(\aTheta^{\bullet}, \Xi_\bullet)$ forms a calculus in the sense of Tamarkin--Tsygan, called the Cartan calculus of $M$ and denoted by $\calculus_C (M)$, where the operators $\iI_X$ and $\iL_X$ are the usual interior product and Lie derivative with respect to a polyvector field $X$, and $d$ is the de~Rham differential.

On the other hand, Tamarkin and Tsygan \cite{MR1039918,MR1783778,MR2986860,MR154906} observed that there exists an analogue of the classical Cartan calculus, called the Tamarkin--Tsygan calculus of $M$ and denoted by $\calculus_H (M)$, on the pair $(\Theta^\bullet,\Xi_\bullet)$ consisting of smooth Hochschild cohomology
$\Theta^\bullet=H^{\bullet}_{\smooth}(C^\infty(M),C^\infty(M))$
and smooth Hochschild homology
$\Xi_\bullet=H_{-\bullet}^{\smooth}(C^\infty(M),C^\infty(M))$.
Here, the smooth Hochschild cohomology is defined as the cohomology of the subcomplex $(\Dpoly{\bullet}(M),\hochschild)$ of the Hochschild cochain complex of $C^\infty(M)$ consisting of \emph{polydifferential operators}.
Meanwhile, the smooth Hochschild homology is defined as the homology of the chain complex of \emph{polyjets}, constructed analogously to the Hochschild chain complex but with the standard tensor product replaced by $\infty$-jets along the diagonal $\Cpoly{-\bullet}(M)= \cJ^\infty_\Delta(M^{\times \bullet})$.
The structural operations of this calculus include the cup product and the Gerstenhaber bracket \cite{MR161898}, the Connes--Rinehart operator $B$ \cite{MR1303779,MR154906}, and, for every Hochschild cochain $D$, the explicit contraction action $\iI_D$ and Lie derivative action $\iL_D$ on Hochschild chains (see~\cite{MR1039918,MR1783778,MR2986860,MR154906}).

At the level of cochains and chains, there exists a pair of natural quasi-isomorphisms $\hkr:\Tpoly{\bullet}(M)\to\Dpoly{\bullet}(M)$ and $\hhkr:\Cpoly{-\bullet}(M)\to\Omega^\bullet(M)$, called HKR maps, defined by
\begin{gather}
\textstyle\hkr(X_1\wedge\cdots\wedge X_n)
=\frac{1}{n!}\sum_{\sigma\in S_n}
\sgn(\sigma)X_{\sigma(1)}\otimes\cdots\otimes
X_{\sigma(n)}, \label{eq:HKR}\\
\textstyle\hhkr (f_0\otimes f_1\otimes\cdots\otimes f_n)
=f_0 df_1\cdots df_n. \label{eq:HKR1}
\end{gather}
It is a classical result \cite{MR1303779,MR142598} that the pair $(\hkr,\hhkr^{-1})$ defines an isomorphism of calculi from $\calculus_C (M)$ to $\calculus_H (M)$.

For a dg manifold $(\cM,Q)$, one defines the respective cohomological analogues
\begin{multline*} \HTT,\quad \HOM, \\ \HDD,\quad\text{and}\quad \HCC \end{multline*}
of
\[ \Tpoly{\bullet}(M),\quad \Omega^{-\bullet}(M),\quad H^{\bullet}_{\smooth}(C^\infty(M),C^\infty(M)),
\quad\text{and}\quad H_{-\bullet}^{\smooth}(C^\infty(M),C^\infty(M)) ,\]
by incorporating the internal differential $Q$.
This yields two calculi, $\calculus_C(\cM,Q)$ and $\calculus_H(\cM,Q)$.
In contrast to the classical case, the HKR maps fail to respect the algebraic structures
and must be corrected by the square root of the Todd class of $(\cM,Q)$.
Our main theorem establishes the resulting Duflo--Kontsevich-type isomorphism.

\begin{introthm}
\label{thm:main}
Let $(\cM, Q)$ be any dg manifold.
Then the pair of maps
\begin{equation}
\label{eq:Duroc}
\hkr\circ(\toddclassQ)^{\frac{1}{2}}: \HTT\xto{\cong} \HDD
\end{equation}
and
\begin{equation}
\label{eq:Odeon}
\big((\toddclassQ)^{\frac{1}{2}}\circ\hhkr\big)^{-1}:\HOM\xto{\cong} \HCC
\end{equation}
is an isomorphism of calculi from $\calculus_C(\cM,Q)$ to $\calculus_H(\cM,Q)$.
Here the square root of the Todd class of the dg manifold,
$(\toddclassQ)^{\frac{1}{2}}\in \cohomology{0}\big(\totApolyM{\bullet}, \cQ\big)$,
acts on $\HTT$ by contractions, and on $\HOM$ by wedge product.
\end{introthm}

The Todd class of a dg manifold provides a unifying framework for various seemingly unrelated objects including the Duflo element of a Lie algebra and the Todd class of a complex manifold.
Indeed, the Todd class of the dg manifold $(\frakg[1],d_{\CE})$ arising from a Lie algebra $\frakg$ coincides with the Duflo element of $\frakg$ \cite{MR4584414,MR3319134,MR4276044}, and the Todd class of the dg manifold $(T^{0,1}_X[1],\bar{\partial})$ arising from a Kähler manifold $X$ is isomorphic to the (usual) Todd class of $X$ \cite{MR3877426,MR4393962}.

As a stepping stone in the proof of Theorem~\ref{thm:main}, we establish a formality theorem for dg manifolds.

\begin{introthm}
\label{thm:formalitydg}
Let $(\cM,Q)$ be a dg manifold, and let $\nabla$ be a torsion-free affine connection on $\cM$.
\begin{enumerate}
\item \label{thm:formalitydg-K}
There exists an $L_\infty$ quasi-isomorphism of dglas
\begin{equation}\label{Nilfisk}
\dgkont : \big( \totTpolyM{\bullet}[1] \big)_Q \inftymorphism \big( \totDpolyM{\bullet} [1] \big)_Q
\end{equation}
whose first Taylor coefficient is given by
\[ \dgkont_1 = \hkr \circ (\toddcocycleQ)^{\frac{1}{2}} : \totTpolyM{\bullet}[1] \to \totDpolyM{\bullet}[1] .\]
\item \label{thm:formalitydg-S}
There exists a quasi-isomorphism of $L_\infty$ modules over the dgla $\big( \totTpolyM{\bullet}[1] \big)_Q$:
\[ \dgshoi : \dgkont^\ast\big( \totCpolyM{\bullet} \big)_Q \inftymorphism \big( \totApolyM{\bullet} \big)_Q \]
whose zeroth Taylor coefficient is given by
\[ \dgshoi_0 = (\toddcocycleQ)^{\frac{1}{2}} \circ \hhkr : \totCpolyM{\bullet} \to \totApolyM{\bullet} .\]
\end{enumerate}
Here, the element $(\toddcocycleQ)^{\frac{1}{2}} \in \totApolyM{0}$ acts on $\totTpolyM{\bullet}[1]$ by contraction and on $\totApolyM{\bullet}$ by multiplication.
We refer to Equations~\eqref{eq:Cluny} and~\eqref{eq:Mirabeau} for the precise dgla structures on the space polyvector fields and the space of polydifferential operators, to Equations~\eqref{eq:Vaneau} and~\eqref{eq:Michel-Ange} for the corresponding dgla module structures on the space of differential forms and the space of polyjets, and to Remark~\ref{rmk:LooModOnCpolyM} below for the structure of $L_\infty$ module (over the dgla of polyvector fields) on the space of polyjets.
\end{introthm}

\begin{remark}\label{rmk:LooModOnCpolyM}
The space of polyjets $\totCpolyM{\bullet}$ is naturally an $L_\infty$ module over the dgla of polydifferential
operators via the action \eqref{eq:Michel-Ange}.
In Theorem~\ref{thm:formalitydg}~\ref{thm:formalitydg-S}, the symbol $\dgkont^\ast\big( \totCpolyM{\bullet} \big)_Q$ denotes the pullback $L_\infty$ module structure on $\totCpolyM{\bullet}$
induced by the $L_\infty$ morphism \eqref{Nilfisk}.
Explicitly, the sequence of multibrackets $(\lambda_n)_{n=0}^\infty$ defining
the $L_\infty$ module structure on $\dgkont^\ast\big( \totCpolyM{\bullet} \big)_Q$ is given by
\begin{equation}
\lambda_n(\gamma_1, \cdots, \gamma_n; \zeta) =
\begin{cases}
(\hochschildb + \iL_Q)\zeta, & \text{if } n = 0, \\[1ex]
\iL_{\dgkont_n(\gamma_1, \cdots, \gamma_n)}\zeta, & \text{if } n \ge 1,
\end{cases}
\end{equation}
for all $\gamma_1, \cdots, \gamma_n \in \totTpolyM{\bullet}[1]$ and $\zeta \in \totCpolyM{\bullet}$.
See Appendix~\ref{sec:PullBack} for the general construction of pullback $L_\infty$ modules.
\end{remark}

Indeed, our approach to Theorem~\ref{thm:main} relies fundamentally on the formality theorems of Kontsevich \cite{MR2062626} and Shoikhet--Tsygan \cite{MR1729368,MR2004726} together with the formal geometry of dg manifolds developed in~\cite{paper-1A,paper-1B} via Fedosov dg Lie algebroids.

The isomorphism \eqref{eq:Duroc} at the level of Gerstenhaber algebras was known as the Kontsevich--Shoikhet conjecture \cite{arXiv:math/9812009} and its proof was announced in a Comptes Rendus note \cite{MR3754617}; see also \cite{MR4276044}.
As special cases, the isomorphism \eqref{eq:Duroc} at the level of associative algebras recovers two fundamental results:
(1) the Duflo--Kontsevich theorem for Lie algebras \cite{arXiv:math/9812009,MR4584414} in the case of $(\frakg [1],d_{\CE})$, and (2) the Kontsevich theorem for complex manifolds \cite{MR4504932} in the case of $(T^{0,1}_X[1],\bar{\partial})$. These two theorems are thus unified within a single conceptual framework.

Note that a formality theorem for $\ZZ$-graded manifolds was obtained by Cattaneo--Felder \cite{MR2304327},
who applied it to the quantization of coisotropic submanifolds of Poisson manifolds.
We expect that Theorem~\ref{thm:formalitydg} can be used to study deformation quantization of ($0$-shifted) derived Poisson manifolds a.k.a. $P_\infty$-manifolds \cite{MR2180451,MR2304327,MR4091493}.

\subsection*{Notations and conventions}

Graded means $\ZZ$-graded, and dg means differential graded in the present paper.

Given a $\ZZ$-bigraded $R$-module $C = \bigoplus_{p,q \in \ZZ} C^{p,q}$, we set
\[ \tot_\oplus^n(C) := \bigoplus_{\substack{p+q=n,\ p,q \in \ZZ}} C^{p,q}
\qquad \text{and} \qquad
\tot_\Pi^n(C) := \prod_{\substack{p+q=n,\ p,q \in \ZZ}} C^{p,q} ,\]
which are called the \emph{direct-sum total space} and the \emph{direct-product total space}, respectively.
The notations $\tot_\oplus$ and $\tot_\Pi$ are inspired by the two types of total complexes of a double complex.
Explicitly, in this paper, a double complex $(C,d_1,d_2)$ is a $\ZZ$-bigraded $R$-module $C = C^{\bullet,\bullet}$ together with operators $d_1: C^{p,q} \to C^{p+1,q}$ and $d_2:C^{p,q} \to C^{p,q+1}$ such that both of the operators square to zero and they anti-commute.
Given a double complex $(C,d_1,d_2)$, one has two types of total complexes: the direct-sum total complex $(\tot_\oplus^\bullet(C), d_1+d_2)$ and the direct-product total complex $(\tot_\Pi^\bullet(C), d_1+d_2)$.

Let $V$ be a graded $R$-module and $i,j \in \ZZ$.
Let $V[i]$ be the graded $R$-module whose homogeneous components are $(V[i])^j = V^{i+j}$.
We denote by $\nshift: V \to V[1]$ the degree-shifting map of degree $-1$, and denote by $\pshift: V \to V[-1]$ the degree-shifting map of degree $+1$.
Thus, we have $V[1] = \nshift V$ and $V[-1] = \pshift V$.

\section{Preliminary}

Dg Lie algebroids provide a framework for conceptually unifying the construction of calculi across diverse settings. In particular, associated with any dg Lie algebroid $(\cL \to \cM, \cQ)$, there exists a dual pair of algebraic structures: the Cartan calculus $\calculus_C(\cL,\cQ)$ and the Tamarkin--Tsygan calculus $\calculus_H(\cL,\cQ)$.
Here, $\calculus_C(\cL,\cQ)$ denotes the calculus structure on the pair of spaces comprising $\cL$-polyvector fields and $\cL$-differential forms: $\big( \cohomology{}\big(\totTpolyL{\bullet},\cQ\big), \, \cohomology{}\big(\totApolyL{\bullet},\cQ\big) \big)$. Concurrently, $\calculus_H(\cL,\cQ)$ denotes the calculus structure on the pair of spaces comprising $\cL$-polydifferential operators and $\cL$-polyjets: $\big( \cohomology{}\big( \totDpolyL{\bullet}, \hochschild+\dQ\big), \, \cohomology{}\big(\totCpolyL{\bullet}, \hochschildb+\dQ \big) \big)$. For a comprehensive treatment, we refer the reader to~\cite{paper-1B}. Two special classes of dg Lie algebroids bear particular significance within the scope of this paper: the tangent dg Lie algebroid of a dg manifold and the Fedosov dg Lie algebroid of a dg manifold. We review their explicit constructions in further detail below.

\subsection{Calculi of dg manifolds}\label{corniche}

This section is devoted to introducing the main objects of study in the present paper: the Cartan and Tamarkin--Tsygan calculi on dg manifolds.

Let $(\cM,Q)$ be a dg manifold.
Then $T_{\cM}\to\cM$ is naturally a dg Lie algebroid, where the homological vector field $\tilde{Q}$ on $T_{\cM}$ is the complete lift of $Q$.
It corresponds to the dg module map $\liederivative{Q}:\Gamma(T_{\cM})\to\Gamma(T_{\cM})$.
Applying the general constructions \cite{paper-1B}, we will be able to obtain the calculi $\calculus_C(\cM,\cQ)$ and $\calculus_H(\cM,\cQ)$.

In what follows, we briefly summarize various spaces involved and their algebraic structures.
The main goal of the paper is to explore the relationship between these algebraic structures.

\subsubsection{Cartan calculus of dg manifolds}

We begin with the Cartan calculus.

Let $(\cM,Q)$ be a dg manifold over $\KK$.
For all $p\geq 0$, let $\Tpolym{p}$ denote the space $\sections{S^p(T_{\cM}[-1])}$, which is isomorphic to $\sections{\Lambda^{p}T_{\cM}}[-p]$, of $p$-vector fields on $\cM$.
We use the symbol $\TpolyM{p}{n}$ to denote the degree~$n$ subspace of $\Tpolym{p}$.
Let $\totTpolyM{n}$ be the direct-sum total space
$\totTpolyM{n} = \bigoplus_{p = 0}^\infty \TpolyM{p}{n}$.
The homological vector field $Q$ on $\cM$ induces a degree~$+1$ differential
$\tQ: \TpolyM{p}{n} \to \TpolyM{p}{n+1}$,
namely the Lie derivative $\tQ=\iL_{Q}=\schouten{Q}{\argument}$ w.r.t. the homological vector field $Q$.
Thus we obtain a dg Gerstenhaber algebra $\big(\totTpolyM{\bullet}, \tQ ,\schouten{\argument}{\argument},\wedge\big)$, where $\schouten{\argument}{\argument}$ and $\wedge$ denote the Schouten bracket and the wedge product on polyvector fields of $\cM$.
Its cohomology $\cohomology{}\big(\totTpolyM{\bullet},\tQ\big)$ is a Gerstenhaber algebra.
In particular, the triple
\begin{equation}\label{eq:Cluny}
\big(\totTpolyM{\bullet}[1]\big)_Q :=
\big(\totTpolyM{\bullet}[1],\schouten{Q}{\argument},\schouten{\argument}{\argument}\big) .
\end{equation}
is a dgla.

Similarly, let $\Apolym{-p} = \sections{S^{p}(T^\vee_{\cM} [1])}$ be the space of differential $p$-forms on $\cM$ with negative degree shifting, which is $\Omega^p(\cM)[p]$.
We denote the degree $n$ subspace of $\Apolym{-p}$ by $\ApolyM{-p}{n}$. By $\totApolyM{n}$, we denote the direct-product total space $ \totApolyM{n} = \prod_{p = 0}^\infty \ApolyM{-p}{n}.$
The homological vector field $Q$ induces a degree $+1$ differential
$\tQ = \iL_Q: \ApolyM{-p}{n} \to \ApolyM{-p}{n+1}$ via the Lie derivative.
By $\big(\totApolyM{\bullet}\big)_Q$, we denote the pair
\begin{equation}\label{eq:Vaneau}
\big(\totApolyM{\bullet}\big)_Q := \big(\totApolyM{\bullet}, \iL_Q \big)
.\end{equation}

By $d_{\dR}$, we denote the de~Rham differential
\begin{equation}\label{DR0}
d_{\dR} : \Apolym{-p} \to \Apolym{-p-1}.
\end{equation}
For any $X \in \Tpolym{d}$, denote by
\begin{equation}
\label{eq:IX0}
\iI_X : \Apolym{-p} \to \Apolym{-p+d}
\end{equation}
the natural contraction operator, and by
\begin{equation}
\label{eq:LX10}
\iL_{X} : \Apolym{-p} \to \Apolym{-p+d-1}
\end{equation}
the Lie derivative defined by the Cartan formula
\begin{equation}
\label{eq:LX20}
\iL_{X} = (-1)^{|X|-1} \commutator{d_{\dR}}{\iI_X}.
\end{equation}

\begin{proposition}[{\cite[]{paper-1B}}]
Let $(\cM, Q)$ be a dg manifold. Then
\begin{itemize}
\item The triple $\big(\totTpolyM{\bullet}[1]\big)_Q$ is a dgla.
\item The pair $\big(\totApolyM{\bullet}\big)_Q$ is a dgla module over
$\big(\totTpolyM{\bullet}[1]\big)_Q$ for the representation \eqref{eq:LX10}.
\item
The pair $(\aTheta^{\bullet}, \Xi_\bullet)$, where $\aTheta^{\bullet}:=\HTT$ and $\Xi_\bullet=\HOM$, admit a calculus structure, whose operators $d$, $\iI_X$, and $\iL_X$ are respectively the de~Rham differential \eqref{DR0}, the interior product \eqref{eq:IX0} and the Lie derivative \eqref{eq:LX10} w.r.t. the polyvector field $X$.
\end{itemize}
\end{proposition}

This calculus is called the Cartan calculus of the dg manifold $(\cM,Q)$ and is denoted by $\calculus_C(\cM,Q)$.

\subsubsection{Tamarkin--Tsygan calculus of dg manifolds}

Next we briefly recall the Tamarkin--Tsygan calculus of a dg manifold.

Let $\pshift\DD(\cM):=\DD(\cM)[-1]$ be the suspended space of linear differential operators on $\cM$, which is naturally an $\cR$-module.
The space $\sD^p_{\poly}(\cM)$ of $p$-differential operators on $\cM$ is defined as
$\sD^p_{\poly}(\cM)=\big(\pshift\DD(\cM)\big)^{\otimes_{\cR}^p}$
for $p\geq 1$ and $\sD^0_{\poly}(\cM)=\cR$.
By $\DpolyM{p}{n}$, we denote the space of $p$-differential operators on $\cM$ of total degree~$n$.
Let $\totDpolyM{n}$ be the direct-sum total space
$\totDpolyM{n}= \bigoplus_{p = 0}^\infty \DpolyM{p}{n}.$
As in the classical case \cite{MR161898}, endowing the space $\totDpolyM{\bullet}$ of
polydifferential operators with the Gerstenhaber
bracket $\gerstenhaber{\argument}{\argument}$ and the Hochschild
differential $\hochschild:=\gerstenhaber{m}{\argument}:\DpolyM{p}{n} \to\DpolyM{p+1}{n}$
makes the triple
\begin{equation}\label{eq:Mirabeau}
\big(\totDpolyM{\bullet}[1]\big)_Q
:= \big(\totDpolyM{\bullet}[1],\hochschild+\dQ,
\gerstenhaber{\argument}{\argument}\big) .
\end{equation}
into a dgla.
Here, $\dQ = \gerstenhaber{Q}{\argument} : \totDpolyM{\bullet} \to \totDpolyM{\bullet+1}$, and $m = -\pshift 1 \otimes_{\cR} \pshift 1 \in \DpolyM{2}{0}$ is the shifted multiplication, defined explicitly by
\begin{equation}\label{eq:ShiftedMultiplication}
m(\nshift f, \nshift g) = (-1)^\degree{f} fg \qquad \forall f, g \in \cR.
\end{equation}
The map $\nshift : \cR \to \cR[1]$ denotes the degree-shifting map of degree $-1$, and $\Dpolym{2} = \pshift \DD(\cM) \otimes_{\cR} \pshift \DD(\cM)$ is identified with a subspace of $\Hom(\cR[1] \otimes_\KK \cR[1], \cR)$.
See, for example, \cite[Appendix~B]{MR4584414}.

The tensor product of left $\cR$-modules determines a cup product
\[ \cupproduct:\DpolyM{p}{n} \times\DpolyM{p'}{n'}\to\DpolyM{p+p'}{n+n'} .\]

It is standard that the Gerstenhaber bracket $\gerstenhaber{\argument}{\argument}$
and the cup product $\cupproduct$ descend to the Hochschild cohomology: $\HDD$.
When endowed with the cup product and the Gerstenhaber bracket, $\HDD$ becomes a Gerstenhaber algebra.
See~\cite[Appendix]{MR2304327}, \cite[Section~2.2]{MR2986860}.

Denote by $\JJ(\cM):=\Hom_{\cR}(\DD(\cM),\cR)$ the space of $\infty$-jets of $\cM$,
which is naturally an $\cR$-module.
For every $p\geq 1$, let $\Cpolym{-p}:= \big(\nshift \JJ(\cM) \big)^{\cotimes_{\cR} \, p}= \big(\JJ(\cM) [1] \big)^{\cotimes_{\cR} \, p}$, called the space of \emph{$p$-polyjets} of $\cM$.
For $p=0$, let $\Cpolym{0}:=\cR$.
For every $p\geq 0$, $\Cpolym{-p}$ can be identified with $\cJ^\infty_\Delta \big( \cM^{\times (p+1)} \big)[p]$, the space of $\infty$-jets along the diagonal $\Delta$ of the product $\cM\times\cdots\times\cM$ of $(p+1)$-copies of the graded manifold $\cM$.
The isomorphism is induced by the map
\begin{equation}\label{eq:Phi}
\Phi: \big(C^\infty(\cM) \otimes_\KK (C^\infty(\cM))^{\otimes_\KK p}\big)[p] \to \Cpolym{-p}
\end{equation}
defined by
\begin{equation}\label{eq:Phi1}
\duality{\Phi(\nshift^p(a_0\otimes a_1\otimes\cdots\otimes a_p))}{\pshift u_1\otimes\cdots\otimes \pshift u_p}
= \pm\, a_0 \cdot u_1(a_1) \cdots u_p(a_p)
,\end{equation}
for $a_0, \cdots, a_p \in C^\infty(\cM)$, and $u_0, \cdots, u_p\in \DD(\cM)$.
See~\cite{paper-zero} for more details.
By $\CpolyM{-p}{n}$, we denote the subspace of $\Cpolym{-p}$ consisting of elements of total degree~$n$.
By $\totCpolyM{n}$, we denote the direct-product total space
$\totCpolyM{n}= \prod_{p = 0}^\infty \CpolyM{-p}{n}$.

Let $a_0, \cdots, a_p \in C^\infty(\cM)$ and let $\iD \in \Dpolym{d}$. We define the linear operators $B : \CpolyM{-p}{n} \to \CpolyM{-p-1}{n-1}$, $\iI_\iD : \Cpolym{-p} \to \Cpolym{-p+d},$ and $\iL_\iD : \Cpolym{-p} \to \Cpolym{-p+d-1}$ as the unique operators determined via the isomorphism \eqref{eq:Phi} by the following explicit formulas:
\begin{multline}\label{eq:Connes1}
B(a_0 \otimes a_1 \otimes \cdots \otimes a_p)
= \sum_{i=0}^p \pm\, \big(1 \otimes a_i \otimes a_{i+1} \otimes \cdots \otimes a_p \otimes a_0 \otimes \cdots \otimes a_{i-1} \\
\qquad \pm\, a_i \otimes 1 \otimes a_{i+1} \otimes \cdots \otimes a_p \otimes a_0 \otimes \cdots \otimes a_{i-1}\big)
,\end{multline}
\begin{equation}\label{eq:iI}
\iI_\iD(a_0 \otimes a_1 \otimes \cdots \otimes a_p)
= (-1)^{|\iD||a_0|} a_0 \iD(a_1, \cdots, a_d) \otimes a_{d+1} \otimes \cdots \otimes a_p
,\end{equation}
\begin{multline}\label{eq:iL}
\iL_\iD(a_0 \otimes \cdots \otimes a_p)
= \sum_{k=0}^{p-d} (-1)^{(|\iD|-1)(\eta_k - 1)} a_0 \otimes \cdots \otimes \iD(a_{k+1}, \cdots, a_{k+d}) \otimes \cdots \otimes a_p \\
+ \sum_{k=p+1-d}^{p} (-1)^{\eta_p (\eta_p - \eta_k)} \iD(a_{k+1}, \cdots, a_p, a_0, \cdots, a_{k+d-p-1}) \otimes \cdots \otimes a_k
.\end{multline}
The second sum in Equation~\eqref{eq:iL} is taken over all cyclic permutations where $a_0$ appears as an argument of $\iD$, and we adopt the grading shorthand notation $\eta_k = |a_0| + \cdots + |a_k| - k$. For a comprehensive discussion of these operations, we refer the reader to \cite{MR1261901,paper-zero}.

The Hochschild boundary differential is defined by
\begin{equation}\label{eq:hochschildb}
\hochschildb = \iL_m :\CpolyM{-p}{n}\to\CpolyM{-p+1}{n+1},
\end{equation}
where $m$ is the shifted multiplication defined in Equation~\eqref{eq:ShiftedMultiplication}.

The homological vector field $Q$ induces a differential $\ccQ$ by Lie derivative
$\ccQ = \liederivative{Q}:\CpolyM{-p}{n}\to\CpolyM{-p}{n+1}$.
By $\big(\totCpolyM{\bullet}\big)_Q$, we denote the pair
\begin{equation}\label{eq:Michel-Ange}
\big(\totCpolyM{\bullet}\big)_Q := \big(\totCpolyM{\bullet}, \hochschildb+\dQ\big)
.\end{equation}

We have the following

\begin{proposition}[{\cite{paper-1B}}]
Let $(\cM,Q)$ be a dg manifold. Then
\begin{itemize}
\item The triple $\big(\totDpolyM{\bullet}[1]\big)_Q$ is a dgla.
\item The pair $\big(\totCpolyM{\bullet}\big)_Q$ is a dgla module over
$\big(\totDpolyM{\bullet}[1]\big)_Q$ for the representation \eqref{eq:iL}.
\item The pair \[ \Theta^{\bullet}=\HDD \quad\text{and}\quad \Xi_\bullet=\HCC \] admits a calculus structure, whose differential $d$ is the Connes--Rinehart operator $B$ defined by Equation~\eqref{eq:Connes1} and whose action operators $\iI_D$ and $\iL_D$ are defined explicitly as in Equations~\eqref{eq:iI} and~\eqref{eq:iL}.
\end{itemize}
\end{proposition}
This calculus is called the Tamarkin--Tsygan calculus of the dg manifold $(\cM,Q)$ and is denoted by $\calculus_H(\cM,Q)$.

\subsection{Calculi of Fedosov dg Lie algebroid}
\label{sec:FedosovMfd}

\subsubsection{Fedosov dg Lie algebroid}

Let $\cM$ be a graded manifold. The \emph{Fedosov manifold} associated with $\cM$ is the graded manifold $\cN = T_{\cM}[1] \oplus \Tformal\cM$ supported on $M$, whose structure sheaf $\fA_{\cN}$ is given by the sections of the graded vector bundle $\hat{S}(T_{\cM}[1] \oplus T_{\cM})\dual \to \cM$. Here, $\Tformal\cM$ denotes the formal neighborhood of the zero section in the tangent bundle $T_{\cM}$, whose structure sheaf (supported on $M$) is given by the sections of $\hat{S}T_{\cM}\dual \to \cM$, while $T_{\cM}[1]$ denotes the shifted tangent bundle, whose structure sheaf (supported on $M$) is induced by the sections of $\hat{S}(T_{\cM}[1])\dual \to \cM$.

Heuristically, we often use the notation $\Omega\big(\cM, \hat{S}(T^\vee_{\cM})\big)$ to denote the algebra of global functions $C^\infty(\cN)$.

Choose a local chart $(x_1,\cdots, x_{\mr})$ on $\cM$, where $(x_1, \cdots, x_m)$ are smooth coordinates on an open subset $U$ in the support $M$, and $(x_{m+1}\cdots, x_{\mr})$ are virtual homogeneous coordinate functions on $\cM$.
It induces local fiberwise linear functions $(\xi_1,\cdots, \xi_{\mr})$ on $T_{\cM}[1]$ and $(y_1,\cdots, y_{\mr})$ on $\Tformal\cM$, respectively.
The local functions
\begin{equation}\label{eq:InducedCoordFedosovMfd}
(x_1,\cdots, x_{\mr}, \xi_1,\cdots, \xi_{\mr}, y_1, \cdots, y_{\mr})
\end{equation}
form a local chat on $\cN = T_{\cM}[1]\oplus \Tformal\cM$
with degrees $|\xi_k|=|x_k|+1$ and $|y_k|=|x_k|$, respectively.
A function on $\cN$ locally can be identified with an element in
\[ C^\infty(U)\llbracket x_{m+1},\cdots,x_{\mr},\xi_1,\cdots,\xi_{\mr},
y_1,\cdots,y_{\mr}\rrbracket .\]

Let $\nabla$ be a torsion-free affine connection on $\cM$. Denote by $\tauTpolyF: \XX(\cM)\to \sections{\cN;\cF}$ the injection in~\cite{paper-1B}.
It is known that $(\cN,\fedosov+\tauTpolyF(Q))$ is a dg manifold \cite{paper-1B},
where
\begin{equation}\label{Fhvf}
\fedosov: \ \Omega^\bullet(\cM,\SM)\to\Omega^{\bullet+1}(\cM,\SM),
\quad \quad \fedosov=-\delta+d^{\nabla}+A^\nabla.
\end{equation}
Here
\begin{equation}\label{Cotonou}
\delta:\OO^p\big(\cM,S^q(T^\vee_{\cM})\big)
\to\OO^{p+1}\big(\cM,S^{q-1}(T^\vee_{\cM})\big), \quad \quad
\delta=\sum_{k=1}^{\mr} \xi_k\otimes\frac{\partial}{\partial y_k},
\end{equation}
$d^{\nabla}$ denotes the induced covariant derivative corresponding
to the affine connection $\nabla$, i.e.
\begin{equation}\label{eq:LocCovDer}
d^\nabla = \sum_{k=1}^{\mr} \xi_k
\Bigg( \frac{\partial}{\partial x_k}\otimes 1
- \sum_{l=1}^{\mr}\sum_{j=1}^{\mr} (-1)^{|y_l|+ |y_l||y_j|} \Gamma_{k,l}^j
\otimes y_l\frac{\partial}{\partial y_j}\Bigg),
\end{equation}
where $\Gamma_{i,j}^k$ are the Christoffel symbols: $\nabla_{\frac{\partial}{\partial x_i}}\frac{\partial}{\partial x_j } = \sum_{k=1}^{\mr} \Gamma_{i,j}^k \frac{\partial}{\partial x_k}$; and \begin{equation}\label{Conakry}
A^\nabla=\sum_{k=1}^{\mr} \xi_k\Bigg(\sum_{\substack{L\in\ZZ_{\geq 0}^{\mr} \\ \abs{L}\geq 2}}
\sum_{j=1}^{\mr} A^k_{L,j} \otimes y^L\frac{\partial}{\partial y_j}\Bigg)
,\end{equation}
with $ A^k_{L,j}\in C^\infty (\cM)$, is an element of degree~$+1$ of
$\OO^1\big(\cM,\hat{S}^{\geq 2}(T_{\cM}^\vee)\otimes T_{\cM}\big)$
uniquely determined by an iteration formula
--- regarded as an operator acting on the algebra $\OO\big(\cM,\hat{S}(T_{\cM}^\vee)\big)$ by derivation.
By $\cNQ$, we denote the dg manifold $(\cN,\fedosov+\tauTpolyF(Q))$.

Consider the Lie subalgebroid $\cF\to\cN$ of $T_{\cN}\to\cN$ whose sections are the vector fields
$\mathcal{Y}$ on $\cN$ of type
\[ \mathcal{Y}=\sum_{k=1}^{m+r}\sum_{J\in\ZZ_{\geq 0}^{m+r}} \nu_{k,J} \cdot y^J \frac{\partial}{\partial y_k}
\qquad\text{with}\ \nu_{k,J}\in\OO(\cM) .\]
As a vector bundle, $\cF\to\cN$ is the pullback of the vector bundle $T_{\cM}\to\cM$
through the surjective submersion $\cN\onto\cM$;
the pullback of a (local) section $\sum_k f_k(x) \frac{\partial}{\partial x_k}$ of $T_{\cM}\to\cM$
is the (local) section $\sum_k f_k(x) \frac{\partial}{\partial y_k}$ of $\cF\to\cN$.
It is a graded vector bundle whose total space $\cF$ is a graded manifold with support $M$.
We have the canonical identification
\begin{equation}\label{eq:Fedosov}
\sections{\cF} \cong C^\infty(\cN)\otimes_{C^\infty(\cM)}\sections{T_{\cM}}
\cong \OO\big(\cM,\hat{S}(T_{\cM}^\vee)\otimes T_{\cM}\big) = \prod_{k=0}^\infty \OO\big(\cM,S^{ k}(T^\vee_{\cM})\otimes T_{\cM}\big)
.\end{equation}
It is straightforward that $\Gamma(\cF)$ is stable under $\schouten{\fedosov+\tauTpolyF(Q)}{\argument}$.
In other words, $\cF$ is a dg foliation of the dg manifold $\cNQ:=(\cN,\fedosov+\tauTpolyF(Q))$.
Therefore $\cF\to\cN$ is a dg Lie subalgebroid of the tangent dg Lie algebroid $T_{\cNQ}\to\cNQ$ of the Fedosov dg manifold $(\cN,\fedosov +\tauTpolyF(Q))$.
This dg Lie subalgebroid is called a \emph{Fedosov dg Lie algebroid}, and is denoted by $\cFQ\to\cNQ$.
Adapting somewhat the construction of~\cite[Section~2]{paper-1B} to the dg Lie algebroid $\cFQ\to\cNQ$,
we obtain two calculi $\calculus_C(\cF,\fedosov+\tauTpolyF(Q))$
and $\calculus_H(\cF,\fedosov+\tauTpolyF(Q))$, which we recall below.

\subsubsection{\texorpdfstring{Cartan calculus $\calculus_C\big(\cF,\fedosov+\tauTpolyF(Q)\big)$}{Tamarkin-Tsygan calculus}}
\label{sec:CalCFedosov}

Let $\cFQ \to \cNQ$ be a Fedosov dg Lie algebroid of a dg manifold $(\cM,Q)$. Recall that
\begin{align}
\Tpolyf{p} &= \GG\big(S^p (\cF[-1])\big) = \OO\big(\cM,\hat{S}(T^\vee_{\cM}) \otimes S^p (T_{\cM}[-1])\big), \label{eq:TpolypF} \\
\Apolyf{-p} &= \GG\big(S (\cF\dual[1])\big) = \OO\big(\cM,\hat{S}(T^\vee_{\cM})\otimes S^p (T_{\cM}\dual[1])\big). \label{eq:ApolypF}
\end{align}
Consider the following total spaces:
\begin{align*}
\totTpolyF{n} &= \bigoplus_{\substack{p,q,r\in\ZZ,\ p\geq 0,\ r\geq 0 \\ p+q+r=n }} \prescript{\cF}{}{\sT}^{r,q,p}, \\
\totcApolyF{n} &= \bigoplus_{r\in\ZZ, \, r\geq 0} \prod_{\substack{p,q\in\ZZ, \, p\geq 0 \\ -p+q=n-r }} \prescript{\cF}{}{\sA}^{r,q,-p}, \\
\totbApolyF{n} &= \prod_{p \in \ZZ,\ p \geq 0} \bigg(\bigoplus_{\substack{q,r\in\ZZ, \, r\geq 0 \\ q+r=n+p}} \prescript{\cF}{}{\sA}^{r,q,-p}\bigg),
\end{align*}
where the components are given by
\begin{align}
\prescript{\cF}{}{\sT}^{r,q,p} &= \Big(\OO^r\big(\cM,\hat{S}(T^\vee_{\cM})\otimes S^p( T_{\cM}[-1])\big) \Big)^{p+q+r} , \label{eq:TpolyF-component}
\\
\prescript{\cF}{}{\sA}^{r,q,-p} &= \Big(\OO^r\big(\cM,\hat{S}(T^\vee_{\cM}) \otimes S^p( T_{\cM}\dual [1])\big) \big)\Big)^{-p+q+r} . \label{eq:ApolyF-component}
\end{align}
These total spaces are equipped with the operations $\wedge$, $\schouten{\argument}{\argument}$, $\iI$, $\iL$, and $d$. Moreover, the pair of canonical maps
\[ \totTpolyF{\bullet} \xto{\id} \totTpolyF{\bullet} \quad \text{and} \quad \totcApolyF{\bullet} \into \totbApolyF{\bullet} \]
are compatible with $\wedge$, $\schouten{\argument}{\argument}$, $\iI$, $\iL$, and $d$; see~\cite{paper-1B}.

Let $\tauQ:=\tauTpolyF(Q)+\fedosov$. We denote by $\big(\totTpolyF{\bullet}[1]\big)_\tauQ$ the triple
\[ \big(\totTpolyF{\bullet}[1]\big)_\tauQ := \big(\totTpolyF{\bullet}[1], \schouten{\tauQ}{\argument}, \schouten{\argument}{\argument}\big), \]
by $\big(\totcApolyF{\bullet}\big)_\tauQ$ the pair
\[ \big(\totcApolyF{\bullet}\big)_\tauQ := \big(\totcApolyF{\bullet}, \iL_\tauQ \big), \]
and by $\big(\totbApolyF{\bullet}\big)_\tauQ$ the pair
\[ \big(\totbApolyF{\bullet}\big)_\tauQ := \big(\totbApolyF{\bullet}, \iL_\tauQ \big). \]

\begin{proposition}[{\cite{paper-1B}}]\label{pro:Genoa}
Let $(\cM,Q)$ be a dg manifold, and let $\cFQ \to \cNQ$ be its Fedosov dg Lie algebroid associated with a torsion-free affine connection $\nabla$. Then:
\begin{itemize}
\item The triple $\big(\totTpolyF{\bullet}[1]\big)_\tauQ$ is a dgla.
\item The Lie derivative turns both $\big(\totcApolyF{\bullet}\big)_\tauQ$
and $\big(\totbApolyF{\bullet}\big)_\tauQ$
into dgla modules over $\big(\totTpolyF{\bullet}[1]\big)_\tauQ$.
\item The canonical inclusion $\big(\totcApolyF{\bullet}\big)_\tauQ \into \big(\totbApolyF{\bullet}\big)_\tauQ$
is a quasi-isomorphism of dgla modules.
\item The pair $\big( \cohomology{}\big(\totTpolyF{\bullet},\schouten{\tauQ}{\argument}\big),\, \cohomology{}\big(\totcApolyF{\bullet},\cL_{\tauQ}\big) \big)$ carries a calculus structure.
\end{itemize}
\end{proposition}

This calculus structure is called the Cartan calculus of the Fedosov dg Lie algebroid $\cFQ \to \cNQ$ and is denoted by $\calculus_C(\cF,\tauQ)$ or $\calculus_C(\cF, \fedosov + \tauTpolyF(Q))$.

\subsubsection{\texorpdfstring{Tamarkin--Tsygan calculus $\calculus_H(\cF,\fedosov+\tauTpolyF(Q))$}{Tamarkin--Tsygan calculus}}
\label{sec:CalHFedosov}

Now we turn our discussion to the space of polydifferential operators and the space of polyjets of the Fedosov dg Lie algebroid $\cFQ\to\cNQ$.

Since the Lie algebroid $\cF \to \cN$ is a foliation on $\cN$, its universal enveloping algebra $\enveloping{\cF}$ can be thought of as the algebra of leafwise differential operators on $\cN$. Hence, $\enveloping{\cF}$ can be identified in a natural way with the $C^\infty(\cN)$-module $\OO\big(\cM, \hat{S}(T^\vee_{\cM}) \otimes S(T_{\cM})\big)$. Dually, the space $\jet{\cF}$ of $\cF$-jets can be identified with the $C^\infty(\cN)$-module $\OO\big(\cM, \hat{S}(T^\vee_{\cM}) \otimes \hat{S}(T_{\cM}^\vee)\big)$. We denote
\begin{align*}
\Dpolyf{p} &=\big(\pshift\,\enveloping{\cF}\big)^{\otimes p}
\cong \OO\big(\cM, \hat{S}(T^\vee_{\cM}) \otimes (S(T_{\cM})[-1])^{\otimes p}\big), \\
\Cpolyf{-p} &= \big(\nshift\jet{\cF}\big)^{\cotimes p}
\cong \OO\big(\cM, \hat{S}(T^\vee_{\cM}) \otimes (\hat{S}(T_{\cM}^\vee)[1])^{\cotimes p}\big).
\end{align*}
Consider the total spaces
\begin{align*}
\totDpolyF{n} =\bigoplus_{\substack{p,q,r\in\ZZ \\ p+q+r=n \\
p\geq 0,\ r\geq 0}} \prescript{\cF}{}{\sD}^{r,q,p}
\qquad\text{and}\qquad
\totcCpolyF{n} =\bigoplus_{r\in\ZZ, \, r\geq 0} \prod_{\substack{p,q\in\ZZ, \, p\geq 0 \\ -p+q=n-r }} \prescript{\cF}{}{\sC}^{r,q,-p},
\end{align*}
where
\begin{align}
\prescript{\cF}{}{\sD}^{r,q,p} & =
\Big(\OO^r\big(\cM,\hat{S}(T^\vee_{\cM})
\otimes ((ST_{\cM})[-1])^{\otimes p}\big) \Big)^{p+q+r} , \label{eq:DpolyF-component} \\
\prescript{\cF}{}{\sC}^{r,q,-p} & =
\Big(\OO^r\big(\cM,\hat{S}(T^\vee_{\cM})
\otimes ((\hat{S}T_{\cM}\dual)[1])^{\cotimes p}\big)\Big)^{-p+q+r}. \label{eq:CpolyF-component}
\end{align}
These spaces are equipped with the operations
$\cupproduct$, $\gerstenhaber{\argument}{\argument}$, $\iI$, $\iL$ and $B$. See~\cite{paper-1B}.

Denote, by $\big(\totDpolyF{\bullet}[1]\big)_\tauQ$,
the triple
\[ \big(\totDpolyF{\bullet}[1]\big)_\tauQ :=
\big(\totDpolyF{\bullet}[1],\gerstenhaber{\tauQ}{\argument},\gerstenhaber{\argument}{\argument}\big) .\]
and by $\big(\totcCpolyF{\bullet}\big)_\tauQ$ the pair
\[ \big(\totcCpolyF{\bullet}\big)_\tauQ :=
\big(\totCpolyF{\bullet}, \iL_\tauQ \big) .\]

\begin{proposition}[{\cite{paper-1B}}]\label{pro:Luxembourg}
Let $(\cM, Q)$ be a dg manifold, and let $\cFQ \to \cNQ$ be its Fedosov dg Lie algebroid associated with a torsion-free affine connection $\nabla$. Then:
\begin{itemize}
\item The triple $\big(\totDpolyF{\bullet}[1]\big)_\tauQ$ is a dgla.
\item The Lie action \eqref{eq:iL} turns $\big(\totcCpolyF{\bullet}\big)_\tauQ$ into a dgla module over $\big(\totDpolyF{\bullet}[1]\big)_\tauQ$.
\item The pair $\big( \cohomology{}\big(\totDpolyF{\bullet}, \gerstenhaber{\tauQ}{\argument}\big),\, \cohomology{}\big(\totcCpolyF{\bullet}, \iL_{\tauQ}\big) \big)$ carries a calculus structure.
\end{itemize}
\end{proposition}

This calculus structure is called the Tamarkin--Tsygan calculus of the Fedosov dg Lie algebroid $\cFQ\to\cNQ$ and is denoted by $\calculus_H(\cF,\tauQ)$ or $\calculus_H(\cF,\fedosov+\tauTpolyF(Q))$.

\subsubsection{Fedosov contractions on a dg manifold}

Heuristically, the Fedosov dg Lie algebroid $\cF_Q \to \cN_Q$ can be viewed as being homotopy equivalent to the tangent dg Lie algebroid $T_{\cM} \to \cM$ of the dg manifold $(\cM, Q)$. Consequently, their corresponding calculi are expected to be isomorphic.

In~\cite{paper-1B}, we established this isomorphism as our main result, showing that $\calculus_C(\cF,\fedosov+\tauTpolyF(Q))$ and $\calculus_H(\cF,\fedosov+\tauTpolyF(Q))$ are indeed isomorphic to the calculi $\calculus_C (\cM, Q)$ and $\calculus_H(\cM, Q)$, respectively, by constructing two pairs of contractions.

\begin{theorem}[{\cite{paper-1B}}]\label{thm:mainT}
Given a dg manifold $(\cM,Q)$ and a torsion-free affine connection $\nabla$ on $\cM$, let $\cF \to \cN$ be the Fedosov dg Lie algebroid corresponding to the $\ZZ$-graded manifold $\cM$ and let $\fedosov$ be the associated Fedosov homological vector field on $\cN$.
\begin{enumerate}
\item \label{thm:mainT-TA}
There exists a pair of contractions
\begin{equation}\label{eq:Contraction-Tpoly}
\begin{tikzcd}[cramped]
\Big(\totTpolyM{\bullet}, \schouten{Q}{\argument}\Big) \arrow[r, "\tauTpolyQ", shift left] & \Big(\totTpolyF{\bullet}, \schouten{\fedosov+\tauTpolyF(Q)}{\argument} \Big) \arrow[l, "\sigmaTpolyQ", shift left] \arrow[loop, "\hTpolyQ", out=5, in=-5, looseness=3]
\end{tikzcd}
\end{equation}
and
\begin{equation}\label{eq:Contraction-Apoly}
\begin{tikzcd}[cramped]
\Big(\totApolyM{\bullet},\iL_{Q}\Big) \arrow[r, "\tauTpolyQ", shift left] & \Big(\totbApolyF{\bullet},\iL_{\fedosov+\tauTpolyF(Q)}\Big) \arrow[l, "\sigmaTpolyQ", shift left] \arrow[loop, "\hTpolyQ", out=5, in=-5, looseness=3]
\end{tikzcd}
\end{equation}
such that the pair of injections $\tauTpolyQ$ preserves the operations $\wedge$, $\schouten{\argument}{\argument}$, $\iI$, $\iL$, and $d$. Furthermore, this statement still holds if $\totbApolyF{\bullet}$ is replaced by the subspace $\totcApolyF{\bullet}$.

\item \label{thm:mainT-AA}
The diagram
\begin{equation}\label{eq:diag-bApoly}
\begin{tikzcd}[row sep=small]
& \big(\totTpolyF{\bullet},\totbApolyF{\bullet}\big) \\
\big(\totTpolyM{\bullet},\totApolyM{\bullet} \big)
\ar[ru, "{( \id , \tauTpolyQ )}"] \ar[rd, "{( \id , \tauTpolyQ )}"'] & \\
& \big(\totTpolyF{\bullet},\totcApolyF{\bullet}\big) \ar[uu,hook]
\end{tikzcd}
\end{equation}
commutes as a diagram of dg calculi.

\item \label{thm:mainT-DC}
There exists a pair of contractions
\begin{equation}\label{eq:Contraction-Dpoly}
\begin{tikzcd}[cramped]
\Big( \totDpolyM{\bullet} , \gerstenhaber{Q}{\argument}+\hochschild \Big) \arrow[r, "\tauDpolyQ", shift left] & \Big( \totDpolyF{\bullet} , \gerstenhaber{\fedosov+\tauTpolyF(Q)}{\argument}+\hochschild \Big) \arrow[l, "\sigmaDpolyQ", shift left] \arrow[loop, "\hDpolyQ", out=5, in=-5, looseness=3]
\end{tikzcd}
\end{equation}
and
\begin{equation}\label{eq:Contraction-Cpoly}
\begin{tikzcd}[cramped]
\Big(\totCpolyM{\bullet},\iL_{Q}+\hochschildb \Big) \arrow[r, "\tauDpolyQ", shift left] & \Big(\totcCpolyF{\bullet}, \iL_{\fedosov+\tauTpolyF(Q)}+\hochschildb \Big) \arrow[l, "\sigmaTpolyQ", shift left] \arrow[loop, "\hTpolyQ", out=5, in=-5, looseness=3]
\end{tikzcd}
\end{equation}
such that the pair of injections $\tauDpolyQ$ preserves the operations $\cupproduct$, $\gerstenhaber{\argument}{\argument}$, $\iI$, $\iL$, and $B$.
\end{enumerate}
\end{theorem}

\subsection{Characteristic classes of a dg Lie algebroid}\label{sec:Atiyah}

Let $\cL\to\cM$ be a dg Lie algebroid with anchor $\rho:\cL\to T_{\cM}$.
An $\cL$-connection on $\cL$ is a degree~$0$ bilinear map
$\nabla:\sections{\cL}\times\sections{\cL}\to\sections{\cL}$
satisfying the pair of relations
\begin{gather*}
\nabla_{fX} Y = f\nabla_X Y , \qquad \qquad \nabla_X (fY) = \rho_X (f) Y + (-1)^{\degree{X}\degree{f}} f \nabla_X Y
,\end{gather*}
for all homogeneous elements $f\in C^\infty(\cM)$, $X\in\sections{\cL}$, and $Y\in\sections{\cL}$.
Note that such connections always exist since the standard partition of unity argument holds in the context of graded manifolds.

Consider the bundle map $\atiyahcocycle{\nabla}{\cL}:\cL\otimes\cL\to\cL$ of degree~$+1$ defined by
\[ \atiyahcocycle{\nabla}{\cL} (X,Y) = \cQ(\nabla_X Y)
-\nabla_{\cQ(X)}Y -(-1)^{\degree{X}}\nabla_X\big(\cQ(Y)\big),
\quad \forall X\in\sections{\cL}, Y\in\sections{\cL} .\]
The bundle map $\atiyahcocycle{\nabla}{\cL}$ can be regarded as a section of degree $+1$
of $\cL^\vee\otimes\End\cL$, and hence as a $1$-cochain in the cochain complex
$\big(\sections{\cL^\vee\otimes\End\cL}^\bullet,\mathcal{Q}\big)$.
It is simple to check that the $\atiyahcocycle{\nabla}{\cL}$ is indeed a 1-cocycle:
$\mathcal{Q}(\atiyahcocycle{\nabla}{\cL})=0$,
and its cohomology class is independent of the choice of the connection $\nabla$ \cite{MR3319134}.
The cohomology class $\atiyahclass{\cL}:=[\atiyahcocycle{\nabla}{\cL}]$
in $\cohomology{1}\big(\Gamma(\cL^\vee\otimes\End\cL)^\bullet,\cQ\big)$ is
the obstruction class to the existence of an $\cL$-connection on $\cL$ compatible with the dg structure,
called the \emph{Atiyah class} of the dg Lie algebroid $\cL\to\cM$ \cite{MR3319134}.

By the natural identification $\big(\sections{\cL^\vee \otimes \End\cL}\big)^1 = \big(\sections{(\cL^\vee[1]) \otimes \End\cL}\big)^0$, the Atiyah cocycle $\atiyahcocycle{\nabla}{\cL}$ can be regarded as a degree zero section of $(\cL^\vee[1]) \otimes \End\cL$.
The \emph{Todd cocycle} and \emph{$\widehat{A}$ cocycle} are the elements $\toddcocycle{\nabla}{\cL}$ and $\Ahatcocycle{\nabla}{\cL}$ in $\totApolyL{0} = \prod_{p=0}^\infty \big(\sections{S^p(\cL^\vee[1])} \big)^0$ defined by
\[ \toddcocycle{\nabla}{\cL} = \Ber\left(\frac{\atiyahcocycle{\nabla}{\cL}}{1-e^{-\atiyahcocycle{\nabla}{\cL}}}\right) \qquad \text{and} \qquad \Ahatcocycle{\nabla}{\cL} = \Ber\left(\frac{\atiyahcocycle{\nabla}{\cL}}{e^{\half\atiyahcocycle{\nabla}{\cL}} - e^{-\half\atiyahcocycle{\nabla}{\cL}}}\right) ,\]
where $\Ber$ denotes the Berezinian \cite{MR914369,MR2275685}.
Both $\toddcocycle{\nabla}{\cL}$ and $\Ahatcocycle{\nabla}{\cL}$ are cocycles, satisfying $\cQ(\toddcocycle{\nabla}{\cL}) = 0 = \cQ(\Ahatcocycle{\nabla}{\cL})$.

The cohomology classes $\toddclass{\cL}$ and $\Ahatclass{\cL}$ in $\cohomology{0}\big( \totApolyL{\bullet}, \cQ\big)$ of the cocycles $\toddcocycle{\nabla}{\cL}$ and $\Ahatcocycle{\nabla}{\cL}$ are independent of the choice of the connection $\nabla$ and are, respectively, called the \emph{Todd class} and \emph{$\widehat{A}$ class} of the dg Lie algebroid $\cL$. Hence, the Todd class and the $\widehat{A}$ class of the dg Lie algebroid $\cL$ are, respectively, the elements
\begin{equation}
\toddclass{\cL} = \Ber\left(\frac{\atiyahclass{\cL}}{1-e^{-\atiyahclass{\cL}}}\right) \qquad \text{and} \qquad \Ahatclass{\cL} = \Ber\left(\frac{\atiyahclass{\cL}}{e^{\half\atiyahclass{\cL}} - e^{-\half\atiyahclass{\cL}}}\right)
\end{equation}
in $\cohomology{0}\big( \totApolyL{\bullet}, \cQ\big)$.

Both $\toddclass{\cL}$ and $\Ahatclass{\cL}$ can be expressed in terms of the scalar Atiyah classes
\[ c_p = \frac{1}{p!} \left(\frac{i}{2\pi}\right)^p \supertrace\big(\atiyahclass{\cL}^p\big) \quad \in \cohomology{0}\big(\Apolyl{-p}, \cQ\big), \quad \forall p \geq 0 .\]
Here $\supertrace : \End(\cL) \to C^\infty(\cM)$ denotes the supertrace. Note that $\supertrace(\atiyahclass{\cL}^p) \in \cohomology{0}\big(\Apolyl{-p}, \cQ\big)$ since $(\atiyahcocycle{\nabla}{\cL})^p \in \big(\sections{S^p(\cL^\vee[1])^{\otimes p} \otimes \End(\cL)}\big)^0$. For details, see~\cite{MR3319134}.

Now consider a finite-dimensional dg manifold $(\cM,Q)$. The tangent bundle $T_{\cM}\to\cM$ is naturally a dg Lie algebroid. By definition, the \emph{Atiyah cocycle} $\atiyahcocycleQ$ (associated with an affine connection $\nabla$) and the \emph{Atiyah class} $\atiyahclassQ$ of the dg manifold $(\cM,Q)$ are, respectively, the Atiyah cocycle (associated with $\nabla$) and the Atiyah class of the dg Lie algebroid $T_{\cM}$.

The \emph{Todd cocycle} and the \emph{$\widehat{A}$ cocycle} of the dg manifold $(\cM,Q)$ associated with an affine connection $\nabla$ are, respectively, the elements
\[ \toddcocycleQ = \Ber\left(\frac{\atiyahcocycleQ}{1-e^{-\atiyahcocycleQ}}\right)
\qquad\text{and}\qquad
\Ahatcocycle{\nabla}{(\cM,Q)} = \Ber\left(\frac{\atiyahcocycleQ}{e^{\half\atiyahcocycleQ}
-e^{-\half\atiyahcocycleQ}}\right) \]
in $\totApolyM{0}
=\prod_{p = 0}^\infty \ApolyM{-p}{0} = \prod_{p = 0}^\infty \big( \sections{S^p(T_{\cM}\dual[1])} \big)^0$.
Their respective cohomology classes in
$\cohomology{0}\big(\totApolyM{\bullet},\iL_Q\big)$
are independent of the choice of the connection $\nabla$
and will be referred to as the \emph{Todd class} and the \emph{$\widehat{A}$ class}
of the dg manifold $(\cM,Q)$, and denoted by $\toddclassQ$ and $\Ahatclass{(\cM,Q)}$, respectively.

\begin{example}[{\cite[Example~3.4]{MR3319134}}]
Let $(x_1, \cdots, x_m; x_{m+1}, \cdots, x_{m+n})$ be coordinate functions on $\RR^{m|n} = \RR^m \times \RR^n[-1]$, and let $Q = \sum_k Q_k \frac{\partial}{\partial x_k}$ be a homological vector field on $\RR^{m|n}$. The Atiyah $1$-cocycle of the dg manifold $(\RR^{m|n}, Q)$ associated with the trivial connection $\nabla$ (which is characterized by the property $\nabla_{\frac{\partial}{\partial x_i}}\frac{\partial}{\partial x_j} = 0$) is then
\begin{equation}\label{velociraptor}
\atiyahcocycle{\nabla}{(\RR^{m|n},Q)} \left(\frac{\partial}{\partial x_i}, \frac{\partial}{\partial x_j}\right) = (-1)^{|x_i|+|x_j|} \sum_k \frac{\partial^2 Q_k}{\partial x_i \partial x_j} \frac{\partial}{\partial x_k}.
\end{equation}
The Atiyah $1$-cocycle $\atiyahcocycle{\nabla}{(\RR^{m|n},Q)}$ captures the second-order derivatives of the components of the homological vector field $Q$.
\end{example}

\subsection{Hochschild--Kostant--Rosenberg theorems}

In this section, we establish the Hochschild--Kos\-tant--Rosenberg theorems for dg manifolds and Fedosov dg Lie algebroids. These results play a fundamental role in this paper.

\subsubsection{Hochschild--Kostant--Rosenberg theorem for dg manifolds}

Let $\cM$ be a graded manifold. We first recall that the \emph{(cohomological) Hochschild--Kostant--Rosenberg map}
\[ \hkr: \totTpolyM{\bullet} \to \totDpolyM{\bullet} \]
is defined by the formula
\begin{equation}\label{eq:hkr-coh}
\hkr(\pshift X_1 \odot \cdots \odot \pshift X_k) = \frac{1}{k!} \sum_{\sigma \in S_k} \pm \, \pshift X_{\sigma(1)} \otimes \cdots \otimes \pshift X_{\sigma(k)},
\end{equation}
where $S_k$ is the symmetric group, the sign $\pm$ is determined by the Koszul sign rule, and $X_1, \cdots, X_k \in \sections{T_{\cM}}$ are regarded as first-order differential operators on the right-hand side.

The \emph{(homological) Hochschild--Kostant--Rosenberg map}
\[ \hhkr: \totCpolyM{\bullet} \to \totApolyM{\bullet} \]
is determined by the map \eqref{eq:Phi1} and the following formula:
\[ \hhkr\big(\Phi(\nshift(a_0\otimes a_1\otimes\cdots\otimes a_p))\big) = a_0 da_1 \cdots da_p \in\Apolym{-p} ,\]
where $a_0, \cdots, a_p \in C^\infty(\cM)$ and $d = d_{\dR}$ is the de~Rham differential. Equivalently, if $\xi_1, \cdots, \xi_p \in \jm = \Hom_{\cR}(\cU(\cM), \cR)$ with $\cR = C^\infty(\cM)$, then
\begin{equation}\label{eq:hkr-homology}
\hhkr(\nshift\xi_1 \otimes \cdots \otimes \nshift\xi_p) = \nshift \bar\xi_1 \odot \cdots \odot \nshift \bar\xi_p \in \Apolym{-p}.
\end{equation}
Here $\bar\xi_i \in \Hom_{\cR}(\XX(\cM), \cR) = \sections{T_{\cM}\dual}$ is defined by $\pair{\bar\xi_i}{X} = \pair{\xi_i}{X}$ for all $X \in \XX(\cM)$.

The following Hochschild--Kostant--Rosenberg theorem for dg manifolds follows from the corresponding theorem for graded manifolds (see \cite{arXiv:2410.15903,MR2112623,MR2202177} and \cite[Lemma~A.2]{MR2304327}) together with spectral sequence arguments.

\begin{proposition}\label{pro:hkrfordg}
Let $(\cM, Q)$ be a finite-dimensional dg manifold. The maps \eqref{eq:hkr-coh} and~\eqref{eq:hkr-homology} induce isomorphisms of vector spaces
\begin{equation}\label{eq:HKR-Dpoly}
\hkr : \HTT \xto{\cong} \HDD
\end{equation}
and
\begin{equation}\label{eq:HKR-Cpoly}
\hhkr : \HCC \xto{\cong} \HOM
\end{equation}
on the cohomology level.
\end{proposition}

\begin{proof}
It is straightforward to show that, on the cochain level, the maps \eqref{eq:hkr-coh} and~\eqref{eq:hkr-homology}
arise from the pair of morphisms of double complexes
\[ \hkr : \big( \sT^{0,\bullet,\bullet}, 0, \cQ \big) \to \big( \sD^{0,\bullet,\bullet}, \hochschild, \cQ\big)
\qquad\text{and}\qquad
\hhkr : \big( \sC^{0,\bullet,\bullet}, \hochschildb, \cQ \big) \to \big( \sA^{0,\bullet,\bullet}, 0, \cQ\big), \]
where
\begin{align*}
\sT^{0,q,p} &= \big(\GG\big(S^p( T_{\cM}[-1])\big)\big)^{p+q},
& \sA^{0,q,-p} &= \big(\GG\big(S^p( T_{\cM}\dual [1])\big)\big)^{-p+q}, \\
\sD^{0,q,p} &= \Big(\big(\pshift\cD(\cM)\big)^{\otimes p}\Big)^{p+q},
& \sC^{0,q,-p} &= \Big(\big(\nshift\jm\big)^{\hat\otimes p}\Big)^{-p+q}.
\end{align*}
According to the Hochschild--Kostant--Rosenberg theorems
\cite{arXiv:2410.15903,MR2112623,MR2202177,MR2304327}, the maps
\[ \hkr : \big( \sT^{0,\bullet,q}, 0 \big) \to \big( \sD^{0,\bullet,q}, \hochschild \big)
\qquad\text{and}\qquad
\hhkr : \big( \sC^{0,\bullet,q}, \hochschildb \big) \to \big( \sA^{0,\bullet,q}, 0 \big) \]
are quasi-isomorphisms for each $q \in \ZZ$. Since $\sT^{0,\bullet,\bullet}$ and $\sD^{0,\bullet,\bullet}$ are right-half-plane double complexes, standard spectral sequence arguments (e.g., \cite[Section~5.6]{MR1269324}) imply that \eqref{eq:HKR-Dpoly} is an isomorphism. As $\sA^{0,\bullet,\bullet}$ and $\sC^{0,\bullet,\bullet}$ are left-half-plane double complexes, \eqref{eq:HKR-Cpoly} is likewise an isomorphism.
\end{proof}

\begin{remark}
The cohomologies in \eqref{eq:HKR-Dpoly} are understood as direct sum total cohomologies. If one uses direct product total cohomologies instead, the map \eqref{eq:HKR-Dpoly} is not necessarily an isomorphism; a counterexample can be found in~\cite{MR2202177}.
\end{remark}

\subsubsection{Hochschild--Kostant--Rosenberg theorem for Fedosov dg Lie algebroids}

Let $\cFQ \to \cNQ$ be a Fedosov dg Lie algebroid associated with a dg manifold $(\cM, Q)$. The Hochschild--Kostant--Rosenberg maps for Fedosov dg Lie algebroids are defined by formulas analogous to \eqref{eq:hkr-coh} and \eqref{eq:hkr-homology}:
\begin{align*}
\hkr(\pshift X_1 \odot \cdots \odot \pshift X_k) &= \frac{1}{k!} \sum_{\sigma \in S_k} \pm \, \pshift X_{\sigma(1)} \otimes \cdots \otimes \pshift X_{\sigma(k)} \in \Dpolyf{k}, \\
\hhkr(\nshift\xi_1 \otimes \cdots \otimes \nshift\xi_p) &= \nshift \bar\xi_1 \odot \cdots \odot \nshift \bar\xi_p \in \Apolyf{-p},
\end{align*}
where $X_1, \cdots, X_k \in \sections{\cF} \subset \enveloping{\cF}$ and $\xi_1, \cdots, \xi_p \in \jet{\cF}$.

\begin{lemma}\label{lem:Tau-HKR}
The diagrams
\begin{equation}\label{eq:lem:Tau-HKR-TD}
\begin{tikzcd}
\totTpolyF{\bullet} \ar[r,"\hkr"] & \totDpolyF{\bullet} \\
\totTpolyM{\bullet} \ar[r,"\hkr"] \ar[u,"\tauTpolyQ"] & \totDpolyM{\bullet} \ar[u,"\tauDpolyQ"']
\end{tikzcd}
\end{equation}
and
\begin{equation}\label{eq:lem:Tau-HKR-CA}
\begin{tikzcd}
\totcCpolyF{\bullet} \ar[r,"\hhkr"] & \totcApolyF{\bullet} \\
\totCpolyM{\bullet} \ar[r,"\hhkr"] \ar[u,"\tauDpolyQ"] & \totApolyM{\bullet} \ar[u,"\tauTpolyQ"']
\end{tikzcd}
\end{equation}
commute.
\end{lemma}

\begin{proof}
Let $X_1, \cdots, X_k \in \XX(\cM)$.
Since the map $\tauTpolyQ$ preserves tensor products \cite{paper-1B}, we have
\begin{align*}
\tauTpolyQ (\hkr(\pshift X_1 \odot \cdots \odot \pshift X_k))
&= \frac{1}{k!} \sum_{\sigma \in S_k} \pm \tauTpolyQ(\pshift X_{\sigma(1)} \otimes \cdots \otimes \pshift X_{\sigma(k)}) \\
&= \frac{1}{k!} \sum_{\sigma \in S_k} \pm \tauTpolyQ(\pshift X_{\sigma(1)} ) \otimes \cdots \otimes \tauTpolyQ(\pshift X_{\sigma(k)}) \\
&= \hkr (\tauTpolyQ(\pshift X_1 \odot \cdots \odot \pshift X_k)),
\end{align*}
which verifies the commutativity of the first diagram.

For the second diagram, let $\xi_1, \cdots, \xi_p \in \jm = \Hom_{\cR}(\cU(\cM),\cR)$, where $\cR = C^\infty(\cM)$.
Since the map $\tauTpolyQ$ preserves the pairing $\pair{\argument}{\argument}$ \cite{paper-1B}, we have
\begin{equation}\label{eq:lem:Tau-HKR-proof-pairing}
\pair{\tauTpolyQ(\bar\xi_i)}{\tauTpolyQ(X)} = \tauTpolyQ(\pair{\bar \xi_i}{X}) = \tauTpolyQ(\pair{ \xi_i}{X}) = \pair{\tauTpolyQ(\xi_i)}{\tauTpolyQ(X)}
\end{equation}
for any $X \in \XX(\cM)$. Recall from~\cite{paper-1B} that $\Gamma(\cN;\cF)$ is the $C^\infty(\cN)$-span of $\tauTpolyF(\XX(\cM))$. By Equation~\eqref{eq:lem:Tau-HKR-proof-pairing}, it follows that $\tauTpolyQ(\bar\xi_i) = \overline{\tauTpolyQ(\xi_i)}$. Therefore,
\begin{align*}
\tauTpolyQ(\hhkr(\nshift\xi_1 \otimes \cdots \otimes \nshift\xi_p))
&= \tauTpolyQ(\nshift \bar\xi_1 \odot \cdots \odot \nshift \bar\xi_p) = \nshift \tauTpolyQ(\bar\xi_1) \odot \cdots \odot \nshift \tauTpolyQ(\bar\xi_p) \\
& = \nshift \overline{\tauTpolyQ(\xi_1)} \odot \cdots \odot \nshift \overline{\tauTpolyQ(\xi_p)}
= \hhkr(\nshift\tauTpolyQ(\xi_1) \otimes \cdots \otimes \nshift\tauTpolyQ(\xi_p)) \\
& = \hhkr(\tauTpolyQ(\nshift\xi_1 \otimes \cdots \otimes \nshift\xi_p)),
\end{align*}
which completes the proof.
\end{proof}

As a consequence, we obtain the following Hochschild--Kostant--Rosenberg theorems for Fedosov dg Lie algebroids:

\begin{corollary}\label{cor:HKR-Fedosov}
The maps
\[ \hkr: \cohomology{}\big(\totTpolyF{\bullet},\schouten{\tauTpolyF(Q)+\fedosov}{\argument} \big) \xto{\cong} \cohomology{}\big(\totDpolyF{\bullet},\gerstenhaber{\tauTpolyF(Q)+\fedosov}{\argument} + \hochschild\big) \]
and
\[ \hhkr:\cohomology{}\big(\totcCpolyF{\bullet},\iL_{\tauTpolyF(Q)+\fedosov} + \hochschildb\big) \xto{\cong} \cohomology{}\big(\totcApolyF{\bullet},\iL_{\tauTpolyF(Q)+\fedosov} \big) \]
are isomorphisms of vector spaces.
\end{corollary}

\begin{proof}
Consider the commutative diagrams \eqref{eq:lem:Tau-HKR-TD} and~\eqref{eq:lem:Tau-HKR-CA} of Lemma~\ref{lem:Tau-HKR}.
According to Theorem~\ref{thm:mainT}, the vertical maps are quasi-isomorphisms.
According to Proposition~\ref{pro:hkrfordg}, the lower horizontal maps are also quasi-isomorphisms.
Therefore, the upper horizontal maps are quasi-isomorphisms as well.
\end{proof}

\begin{remark}
It follows from~\eqref{eq:diag-bApoly} and~\eqref{eq:lem:Tau-HKR-CA} that the diagram of cochain complexes
\[ \begin{tikzcd}
\totcCpolyF{\bullet} \ar[r,"\hhkr"] & \totcApolyF{\bullet} \ar[r, hook] & \totbApolyF{\bullet} \\
\totCpolyM{\bullet} \ar[r,"\hhkr"] \ar[u,"\tauDpolyQ"]
& \totApolyM{\bullet} \ar[u,"\tauTpolyQ"] \ar[r, "\id"] & \totApolyM{\bullet} \ar[u, "\tauTpolyQ"]
\end{tikzcd} \]
commutes.
According to Theorem~\ref{thm:mainT}, its three vertical maps are quasi-isomorphisms.
According to Proposition~\ref{pro:hkrfordg}, the lower horizontal map is also a quasi-isomorphism.
Therefore, the upper horizontal map
$\hhkr: \big(\totcCpolyF{\bullet}\big)_\tauQ \to \big(\totbApolyF{\bullet}\big)_\tauQ$
is a quasi-isomorphism as well.
\end{remark}

\section{Formality theorems for Fedosov dg Lie algebroids}

Given a finite-dimensional dg manifold $(\cM,Q)$,
let $\cF_{\cQ}\to\cN_{\cQ}$ be the Fedosov dg Lie algebroid corresponding to a torsion-free affine connection $\nabla$ on $\cM$.
The symbol $\cQ$ denotes the induced homological vector field
\begin{equation}\label{eq:cQ-Fedosov}
\tauQ:=\tauTpolyF(Q)+\fedosov
\end{equation}
on the Fedosov manifold $\cN$.
According to~\eqref{eq:Fedosov}, we have $\Gamma(\cN;\cF) = C^\infty(\cN)\otimes_{\cR}\Gamma(T_{\cM}) = \Omega(\cM,\SM\otimes T_{\cM}) $.
The canonical inclusion $\etendu{\tau}:\Gamma(T_{\cM})\to\Gamma(\cN;\cF)$ determines a ``canonical'' $\cF$-connection $\nabla^{\can}$ on $\cF$ through the relation
\[ \nabla^{\can}_{\etendu{\tau}(a)} \etendu{\tau}(b) = 0, \quad \forall a, b \in \sections{T_{\cM}} .\]

The corresponding Atiyah, Todd, and \Aroof\ cocycles are respectively denoted $\atiyahcocycle{\can}{\cF_{\cQ}}$,
$\toddcocycle{\can}{\cF_{\cQ}}$, and $\Ahatcocycle{\can}{\cF_{\cQ}}$.

In this section, we aim to prove the following three propositions:

\begin{proposition}[Kontsevich-type formality theorem for Fedosov dg Lie algebroids]
\label{thm:KontsevichFormalityFedosov}
There exists an $L_\infty$ quasi-isomorphism
\[ \ukontsevich : \big( \totTpolyF{\bullet}[1] \big)_{\tauQ}
\inftymorphism \big( \totDpolyF{\bullet} [1] \big)_{\tauQ} \]
from the dgla \[ \big(\totTpolyF{\bullet}[1]\big)_{\tauQ} = \Big(\totTpolyF{\bullet}[1],\
\schouten{\tauTpolyF(Q)+\fedosov}{\argument},\ \schouten{\argument}{\argument}\Big) \] to the dgla
\[ \big(\totDpolyF{\bullet}[1]\big)_{\tauQ} =
\Big(\totDpolyF{\bullet}[1],\
\gerstenhaber{\tauDpolyF(Q)+\fedosov+m_2}{\argument},\ \gerstenhaber{\argument}{\argument}\Big) ,\]
whose first Taylor coefficient is the composition
\[ \ukontsevich_1=\hkr\circ(\todd^{\can}_{\cF_{\tauQ}})^{\frac{1}{2}}:
\totTpolyF{\bullet}[1]\to\totDpolyF{\bullet}[1] ,\]
of the action (by contraction) of $(\todd^{\can}_{\cF_{\tauQ}})^{\frac{1}{2}} \in \totbApolyF{0}$ on $\totTpolyM{\bullet}[1]$ with the cohomological HKR map.
\end{proposition}

The space of differential forms is an $L_\infty$ module over the dgla of polyvector fields
$\big(\totTpolyF{\bullet}[1]\big)_{\tauQ}$, which we denote by the symbol $\big(\totbApolyF{\bullet}\big)_{\tauQ}$.
The sequences of multibrackets $(\lambda^{A,\tauQ}_n)_{n=0}^\infty$
defining the $L_\infty$ module structure on differential forms is given by
\begin{equation}\label{eq:TwistedApolyFedosov}
\begin{cases}
\lambda^{A,\tauQ}_0(\alpha) = \iL_{\tauTpolyF(Q)+\fedosov}\alpha, & \text{if $n=0$} , \\
\lambda^{A,\tauQ}_1( \gamma_1;\alpha) = \iL_{\gamma_1}\alpha, & \text{if $n=1$} , \\
\lambda^{A,\tauQ}_n(\gamma_1,\cdots,\gamma_n;\alpha) = 0, & \text{if $n\geq 2$} ,
\end{cases}
\end{equation}
for all $\gamma_1,\cdots,\gamma_n\in\totTpolyF{\bullet}[1]$ and $\alpha\in\totbApolyF{\bullet}$.

The space of polyjets admits a structure of $L_\infty$ module over the dgla
$\big(\totTpolyF{\bullet}[1]\big)_{\tauQ}$ of polyvector fields
induced by the pullback process via the $L_\infty$ quasi-isomorphism $\ukontsevich$,
which we denote by the symbol $\ukontsevich^\ast \big( \totcCpolyF{\bullet} \big)_{\tauQ}$. (See Appendix~\ref{sec:PullBack} for the general construction of pullback $L_\infty$ modules.)
The sequence of multibrackets $(\breve\lambda^{C,\tauQ}_n)_{n=0}^\infty$ defining the $L_\infty$ module structures on polyjets is given by
\begin{equation}\label{eq:ToddTwistedCpolyFedosov}
\begin{cases}
\breve\lambda^{C,\tauQ}_0(\zeta)=\iL_{\tauTpolyF(Q)+\fedosov+m_2}\zeta, & \text{if $n=0$} , \\
\breve\lambda^{C,\tauQ}_n(\gamma_1,\cdots,\gamma_n;\zeta)
=\iL_{\ukontsevich_n(\gamma_1,\cdots,\gamma_n)}\zeta, & \text{if $n\geq 1$} .
\end{cases}
\end{equation}
for all $\gamma_1,\cdots,\gamma_n\in\totTpolyF{\bullet}[1]$ and $\zeta\in\totcCpolyF{\bullet}$.

\begin{proposition}[Tsygan-type formality theorem for Fedosov dg Lie algebroids]
\label{thm:TsyganFormalityFedosov}
There exists a quasi-isomorphism of $L_\infty$ modules
\[ \ushoikhet : \ukontsevich^\ast \big( \totcCpolyF{\bullet} \big)_{\tauQ}
\inftymorphism \big( \totbApolyF{\bullet} \big)_{\tauQ} \]
over the dgla $\big(\totTpolyF{\bullet}[1]\big)_{\tauQ}$, whose zeroth Taylor coefficient is the composition
\[ \ushoikhet_0 =(\todd^{\can}_{\cF_{\tauQ}})^{\frac{1}{2}} \circ \hhkr:
\totcCpolyF{\bullet}\to\totbApolyF{\bullet} \]
of the homological HKR map with the action (by multiplication)
of $(\todd^{\can}_{\cF_{\tauQ}})^{\frac{1}{2}}\in\totbApolyF{0}$ on $\totbApolyF{\bullet}$.
\end{proposition}

The use of $\totbApolyF{\bullet}$ instead of $\totcApolyF{\bullet}$ is necessary here because the square root $(\todd^{\can}_{\cF_{\tauQ}})^{\frac{1}{2}}$ of the Todd cocycle resides in $\totbApolyF{0} \setminus \totcApolyF{0}$.

Passing to cohomology, the square root
$(\toddclass{\cF_{\tauQ}})^{\frac{1}{2}} \in\cohomology{0}\big(\totbApolyF{\bullet},\iL_{\tauQ}\big)$ of the Todd class of the Fedosov dg Lie algebroid $\cF_{\cQ}\to\cN_{\cQ}$ acts on
$\cohomology{}\big(\totTpolyF{\bullet},\schouten{\tauQ}{\argument}\big)$ by contraction
and on $\cohomology{}\big(\totbApolyF{\bullet},\iL_{\tauQ}\big)$ by multiplication.

\begin{proposition}[Duflo--Kontsevich-type theorem for Fedosov dg Lie algebroids]\label{thm:DK_Fedosov}
The pair of maps
\begin{multline}\label{eq:uKont1_coh}
\hkr\circ(\toddclass{\cF_{\cQ}})^{\frac{1}{2}}:
\cohomology{}\big(\totTpolyF{\bullet},\schouten{\tauTpolyF(Q)+\fedosov}{\argument}\big)
\\
\xto{\cong} \cohomology{}\big(\totDpolyF{\bullet},\gerstenhaber{\tauTpolyF(Q)+\fedosov+m_2}{\argument}\big)
\end{multline}
and
\begin{multline} \label{eq:uShoi0_coh}
((\toddclass{\cF_{\cQ}})^{\frac{1}{2}} \circ
\hhkr)^{-1}:
\cohomology{}\big(\totbApolyF{\bullet},\iL_{\tauTpolyF(Q)+\fedosov} \big)
\xto{\cong} \cohomology{}\big(\totcCpolyF{\bullet},\iL_{\tauTpolyF(Q)+\fedosov+m_2}\big)
\end{multline}
is an isomorphism of calculi from $\calculus_C(\cF,\tauQ)$ to $\calculus_H(\cF,\tauQ)$.
\end{proposition}

\subsection{Formality morphisms for Fedosov dg Lie algebroids}
\label{sec:ConstructionFormalityFedosov}

\subsubsection{Construction of the formality morphisms}

Following Dolgushev \cite{MR2102846,MR2199629},
we construct formality morphisms for Fedosov dg Lie algebroids in two successive steps.

The first step consists in applying the quasi-isomorphisms of Kontsevich \eqref{eq:KontsevichFormality}
and Shoikhet \eqref{eq:ShoikhetFormality} $\cF$-leafwise on $\cN$.
Each leaf of this foliation is essentially diffeomorphic to a fixed $\ZZ$-graded vector space:
the standard fiber of the vector bundle $T_{\cM}\to\cM$.
This yields a morphism of $L_\infty$ algebras
\begin{equation}\label{Masaki}
\fkontsevich: \big(\totTpolyF{\bullet}[1],0,
\schouten{\argument}{\argument}\big) \inftyto \big(\totDpolyF{\bullet}[1],\hochschild,
\gerstenhaber{\argument}{\argument}\big)
\end{equation}
and a morphism
\begin{equation}\label{Kenji}
\fshoikhet: (\fkontsevich)^\ast \big(\totcCpolyF{\bullet}\big) \inftyto \totbApolyF{\bullet}
\end{equation}
of $L_\infty$ modules over the dgla
$\big(\totTpolyF{\bullet}[1],0,\schouten{\argument}{\argument}\big)$.

More explicitly, we choose a local coordinate system $(x_1, \cdots, x_{\mr})$ on $\cM$, which induces a local coordinate system $(x_1, \cdots, x_{\mr}, \xi_1, \cdots, \xi_{\mr}, y_1, \cdots, y_{\mr})$ on $\cN$ (see~\eqref{eq:InducedCoordFedosovMfd}). Given $\gamma_1, \cdots, \gamma_n \in \totTpolyF{\bullet}[1]$, we can express each element in a chosen local coordinate system as
\[ \gamma_i = \nshift \Big( \sum_{J, K, L} f_{i, J, K}^L \, \xi^J y^K (\pshift \partial_y)^{\odot L} \Big) ,\]
where $J, K, L$ are multi-indices, $f_{i, J, K}^L$ are local functions of $x_1, \cdots, x_{\mr}$.
Here we write $\xi^J = \xi_1^{j_1} \cdots \xi_{\mr}^{j_{\mr}}$ for $J = (j_1, \cdots, j_{\mr})$, $y^K = y_1^{k_1} \cdots y_{\mr}^{k_{\mr}}$ for $K = (k_1, \cdots, k_{\mr})$, and
\[ (\pshift \partial_y)^{\odot L} = (\pshift \partial_{y_1})^{\odot l_1} \odot \cdots \odot (\pshift \partial_{y_{\mr}})^{\odot l_{\mr}} \]
for $L = (l_1, \cdots, l_{\mr})$.
In this type of local coordinate system, we have, for $n \geq 1$,
\begin{equation}\label{Kyoji}
\fkontsevich_n( \gamma_1, \cdots, \gamma_n) = \sum_{i, J_i, K_i, L_i} \pm \Big(\prod_i f_{i, J_i, K_i}^{L_i} \, \xi^{J_i}\Big) \cdot \kontsevich_n\big(\nshift(y^{K_1} (\pshift \partial_y)^{\odot L_1}), \cdots, \nshift( y^{K_n} (\pshift \partial_y)^{\odot L_n})\big).
\end{equation}
Similarly, an element $\zeta \in \totcCpolyF{\bullet}$ is locally of the form
\[ \zeta = \sum_{P, Q_1, \cdots, Q_s} g_{P, Q_1, \cdots, Q_s} \, \xi^P \nshift(y^{Q_1})
\otimes \cdots \otimes \nshift(y^{Q_s}) ,\]
where $P, Q_1, \cdots, Q_s$ are multi-indices, and $g_{P, Q_1, \cdots, Q_s}$ are local functions of $x_1, \cdots, x_{\mr}$. Then for $n \geq 0$,
\begin{multline}\label{Kaoru}
\fshoikhet_n(\gamma_1, \cdots, \gamma_n; \zeta) = \sum_{\substack{i, J_i, K_i, L_i \\ s, P, Q_1, \cdots, Q_s}} \pm \Big(\prod_i f_{i, J_i, K_i}^{L_i} \, \xi^{J_i}\Big) \cdot g_{P, Q_1, \cdots, Q_s} \, \xi^P \\
\cdot \shoikhet_n\big(\nshift(y^{K_1} (\pshift \partial_y)^{\odot L_1}), \cdots, \nshift(y^{K_n} (\pshift \partial_y)^{\odot L_n}); \nshift(y^{Q_1}) \otimes \cdots \otimes \nshift(y^{Q_s})\big).
\end{multline}
Here the signs $\pm$ are determined by the Koszul sign rule.

The second step of the construction consists in twisting the morphims \eqref{Kyoji} and~\eqref{Kaoru}.

Consider the vector field of degree $1$
\begin{equation}\label{eq:globalOmega}
\vfedosov = \tauTpolyF(Q) + \fedosov - d^\nabla =\tauTpolyF(Q) -\delta+A^\nabla
,\end{equation}
where $\tauTpolyF:\XX(\cM)\to\Gamma(\cF)$ is the injection appearing in the contraction
\eqref{eq:Contraction-Tpoly}.
Note that the difference between the vector fields $\fedosov$ and $d^\nabla$ is tangent to $\cF$. Therefore, $\vfedosov \in \big(\Tpolyf{1}[1]\big)^1$ is $\cF$-longitudinal.

We introduce two sequences $(\tkontsevich_n)_{n\in\NN}$ and $(\tshoikhet_n)_{n\in\NO}$ of maps
\begin{small}
\begin{gather*}
\tkontsevich_n: \overbrace{\totTpolyF{\bullet}[1] \times \cdots \times \totTpolyF{\bullet}[1]}^{\text{$n$ factors}} \to \totDpolyF{\bullet}, \\
\tshoikhet_n: \overbrace{\totTpolyF{\bullet}[1] \times \cdots \times \totTpolyF{\bullet}[1]}^{\text{$n$ factors}} \times \totcCpolyF{\bullet} \to \totbApolyF{\bullet},
\end{gather*}
\end{small}
defined as follows:
\begin{gather}
\tkontsevich_n(\gamma_1,\cdots,\gamma_n) = \sum_{k=0}^\infty \dfrac{1}{k!} \fkontsevich_{n+k}(\vfedosov,\cdots, \vfedosov,\gamma_1,\cdots,\gamma_n), \label{eq:TwistedVerKont} \\
\tshoikhet_n(\gamma_1,\cdots,\gamma_n;\zeta) = \sum_{k=0}^\infty \dfrac{1}{k!} \fshoikhet_{n+k}(\vfedosov,\cdots, \vfedosov,\gamma_1,\cdots,\gamma_n;\zeta), \label{eq:TwistedVerShoi}
\end{gather}
for $\gamma_1,\cdots,\gamma_n\in\totTpolyF{\bullet}[1]$ and $\zeta\in\totcCpolyF{\bullet}$.

\begin{lemma}
The maps $\tkontsevich_n$ and $\tshoikhet_n$ are well-defined.
\end{lemma}
\begin{proof}
Let $\prescript{\cF}{}{\sT}^{r,q,p}$, $\prescript{\cF}{}{\sA}^{r,q,-p}$, $\prescript{\cF}{}{\sD}^{r,q,p}$, and $\prescript{\cF}{}{\sC}^{r,q,-p}$ be the component spaces defined in Equations~\eqref{eq:TpolyF-component}, \eqref{eq:ApolyF-component}, \eqref{eq:DpolyF-component}, and \eqref{eq:CpolyF-component}, respectively. Without loss of generality, we may assume that $\zeta \in \prescript{\cF}{}{\sC}^{\tilde{r},\tilde{q},-\tilde{p}}$ and $\gamma_i \in \prescript{\cF}{}{\sT}^{r_i,q_i,p_i}[1]$ for each $i \in \{1, \cdots, n\}$. Note that $\vfedosov \in (\prescript{\cF}{}{\sT}^{1,0,1} \oplus \prescript{\cF}{}{\sT}^{0,1,1})[1]$.

Examining the definition of the Kontsevich morphism closely, we note that
\[ \fkontsevich_{n+k}(\underset{\text{$k$ copies}}{\underbrace{\vfedosov,\cdots,\vfedosov}},\gamma_1,\cdots,\gamma_n)
\in \big(\Dpolyf{p_1+\cdots+p_n-2n-k+2}[1]\big)^{d_1+\cdots+d_n-2n+1} ,\]
where $d_i = p_i + q_i + r_i$ for each $i \in \{1,\cdots, n\}$.
This degree counting shows that, as $k$ increases,
$\fkontsevich_{n+k}(\vfedosov,\cdots,\vfedosov,\gamma_1,\cdots,\gamma_n)$
eventually vanishes for $k$ sufficiently large.
Consequently, since only finitely many terms of the sum are nonzero, the sum $\sum_{k=0}^\infty\fkontsevich_{n+k}(\vfedosov,\cdots,\vfedosov,\gamma_1,\cdots,\gamma_n)$
belongs to $\totDpolyF{\bullet}[1]$.
The map $\tkontsevich_n$ is therefore well-defined.

Inspecting the definition of Shoikhet's formality morphism carefully, we note that
\[ \fshoikhet_{n+k}(\vfedosov,\cdots,\vfedosov,\gamma_1,\cdots,\gamma_n;\zeta) \in \bigoplus_{j=0}^k
\prescript{\cF}{}{\sA}^{(\check\mathbf{r} +j),(\check\mathbf{q}+k-j),(\check\mathbf{p}-2n-k)} ,\]
where $\check\mathbf{p} = p_1+ \cdots + p_n - \tilde{p}$;
$\check\mathbf{q}= q_1 + \cdots + q_n + \tilde{q}$;
and $\check\mathbf{r}=r_1 + \cdots + r_n + \tilde{r}$.
Thus,
\[ \tshoikhet_n(\gamma_1,\cdots,\gamma_n;\zeta) \in \prod_{k=0}^\infty \bigg(\bigoplus_{j=0}^k
\prescript{\cF}{}{\sA}^{(\check\mathbf{r}+j),(\check\mathbf{q}+k-j),(\check\mathbf{p}-2n-k)}\bigg)
\subset \totbApolyF{\check\mathbf{p} + \check\mathbf{q} + \check\mathbf{r}-2n} ,\]
and the proof is complete.
\end{proof}

\begin{proposition}\label{prop:FormalityMor-Fedosov}
\begin{enumerate}
\item
The maps $(\tkontsevich_n)_{n\in\NN}$ defined by Equation~\eqref{eq:TwistedVerKont}
are the Taylor coefficients of an $L_\infty$ morphism of dglas
\begin{equation}\label{eq:tkonsevich}
\tkontsevich: \big(\totTpolyF{\bullet}[1]\big)_{\tauQ} \inftyto \big(\totDpolyF{\bullet}[1]\big)_{\tauQ} .
\end{equation}
\item
The maps $(\tshoikhet_n)_{n\in\NO}$ defined by Equation~\eqref{eq:TwistedVerShoi}
are the Taylor coefficients of a morphism
\begin{equation}\label{eq:tshoikhet}
\tshoikhet: \tkontsevich^\ast\big(\totcCpolyF{\bullet}\big)_{\tauQ} \inftyto \big(\totbApolyF{\bullet}\big)_{\tauQ} ,
\end{equation}
of $L_\infty$ modules over the dgla $\big(\totTpolyF{\bullet}\big)_{\tauQ}$.
The sequences of multibrackets $(\lambda^{A,\tauQ}_n)_{n=0}^\infty$
and $(\lambda^{C,\tauQ}_n)_{n=0}^\infty$
defining the $L_\infty$ module structures on differential forms and polyjets
are given by Equation~\eqref{eq:TwistedApolyFedosov} and
\begin{equation}\label{eq:TwistedCpolyFedosov}
\begin{cases}
\lambda^{C,\tauQ}_0(\zeta)=\iL_{\tauTpolyF(Q)+\fedosov+m_2}\zeta, & \text{if $n=0$} , \\
\lambda^{C,\tauQ}_n(\gamma_1,\cdots,\gamma_n;\zeta)
=\iL_{\tkontsevich_n(\gamma_1,\cdots,\gamma_n)}\zeta, & \text{if $n\geq 1$} .
\end{cases}
\end{equation}
for all $\gamma_1,\cdots,\gamma_n\in\totTpolyF{\bullet}[1]$ and $\zeta\in\totcCpolyF{\bullet}$.
\end{enumerate}
\end{proposition}

\subsubsection{Proof of Proposition~\ref{prop:FormalityMor-Fedosov}}

The reader will have noticed that the expressions \eqref{eq:TwistedVerKont} and~\eqref{eq:TwistedVerShoi} defining the sequences of twisted maps $(\tkontsevich_n)_{n\in\NN}$ and $(\tshoikhet_n)_{n\in\NO}$ are clearly motivated by the theory of twisted morphisms of $L_\infty$ algebras and $L_\infty$ modules.
However, we need to emphasize that $\vfedosov$ is \emph{not} a Maurer--Cartan element.
Nevertheless, we will show that, when restricted to a local coordinate system, the maps $(\tkontsevich_n)_{n\in\NN}$ and $(\tshoikhet_n)_{n\in\NO}$ are indeed the Taylor coefficients of $L_\infty$ morphisms twisted by a \emph{locally defined} Maurer-Cartan element.

To see this, let us restrict the discussion to a system of local coordinates:
Let $(U,\phi)$ be a local chart on $\cM$,
i.e. $U$ is an open subset of the support $M$ of $\cM$,
and $\phi$ is an isomorphism of graded manifolds
\begin{equation}\label{eq:LocChartOnM}
\phi: \cM|_U \to \cV,
\end{equation}
identifying $\cM|_U$ with the \emph{trivialized} $\ZZ$-graded manifold $\cV$ (with support $\RR^m$)
arising from the $\ZZ$-graded vector bundle $\RR^m\times V\to\RR^m$,
where $m=\dim(M)$ and $V$ is a graded vector space of dimension $r$.
See Appendix~\ref{BlackBoxes} for the notion of trivialized graded manifold.
Let $(x_1,\cdots,x_{\mr})$ be the associated local coordinates on $\cM$,
where $(x_1,\cdots,x_m)$ are smooth coordinates on $U$,
and $(x_{m+1},\cdots,x_{\mr})$ form a homogeneous basis for $V\dual$.
The latter are also regarded as virtual homogeneous coordinate functions on $\cM$.

The local chart $(U,\phi)$ on $\cM$ induces a local chart $(U,\psi)$
on the associated Fedosov manifold $\cN= T_{\cM} [1]\times_{\cM} T_{\cM}$,
i.e. an isomorphism of graded manifolds
\begin{equation}
\label{eq:LocChartOnN}
\psi: \cN|_U \to \cW
\end{equation}
identifying $\cN|_U$ with the \emph{trivialized} $\ZZ$-graded manifold $\cW$ (with support $\RR^m$)
arising from the $\ZZ$-graded vector bundle
$\RR^m\times (V\oplus \RR^m[1]\oplus V[1]\oplus \RR^m \oplus V) \to \RR^m$.

As in~\eqref{eq:InducedCoordFedosovMfd},
we denote the induced local coordinates on $\cN|_U$ by
$(x_1,\cdots,x_{\mr},\xi_1,\cdots,\xi_{\mr},y_1,\cdots,y_{\mr})$.

We think of $\cW$ as the product of two trivialized graded manifolds:
\begin{itemize}
\item the trivialized manifold to which $T_{\cM}[1]|_U$ is identified,
which arises from the $\ZZ$-graded vector bundle $\RR^m\times (V\oplus \RR^m[1]\oplus V[1]) \to \RR^m$
and corresponds to the coordinates $(x_1,\cdots,x_{\mr},\xi_1,\cdots,\xi_{\mr})$;
\item and the trivialized manifold $\cVformal$ corresponding to the coordinates $(y_1,\cdots,y_{\mr})$
arising from the $\ZZ$-graded vector bundle $\RR^m \oplus V \to \{*\}$ (over the one-point space).
\end{itemize}
In this trivialization,
the algebra of functions on $T_{\cM}[1]|_U$ is the algebra $\Omega(\cV)$ of differential forms on $\cV$,
and the leaves of the $\cF$-foliation of $\cN|_U$ are the $\cVformal$-slices of $\cW$.
Therefore, we get the natural isomorphisms of vector spaces
\begin{gather}
\Tpolyf{p}\big|_U \cong \Omega(\cV) \cotimes_\KK \Tpoly{p}(\cV_{\formal}) ,\label{eq:Tpoly_loc} \\
\Dpolyf{p}\big|_U \cong \Omega(\cV) \cotimes_\KK \Dpoly{p}(\cV_{\formal}) ,\label{eq:Dpoly_loc} \\
\Apolyf{-p}\big|_U \cong \Omega(\cV) \cotimes_\KK \Apoly{-p}(\cV_{\formal}) ,\label{eq:Apoly_loc} \\
\Cpolyf{-p}\big|_U \cong \Omega(\cV) \cotimes_\KK \Cpoly{-p}(\cV_{\formal}) .\label{eq:Cpoly_loc}
\end{gather}

Since $\Omega(\cV)$ endowed with the de~Rham differential
$d_{\DR}=\sum_k\xi_k\frac{\partial}{\partial x_k}$ is a dgca,
the canonical $L_\infty$ algebra structures carried by
$\Tpoly{}(\cV_{\formal})[1]$ and $\Dpoly{}(\cV_{\formal})[1]$
extend to $\Omega(\cV) \cotimes_\KK \Tpoly{}(\cV_{\formal})[1]$
and $\Omega(\cV) \cotimes_\KK \Dpoly{}(\cV_{\formal})[1]$.
Likewise, the canonical $L_\infty$ module structures
on $\Cpoly{}(\cV_{\formal})$ and $\Apoly{}(\cV_{\formal})$
extend to $\Omega(\cV) \cotimes_\KK \Cpoly{}(\cV_{\formal})$
and $\Omega(\cV) \cotimes_\KK \Apoly{}(\cV_{\formal})$.
Furthermore, Kontsevich's $L_\infty$ morphism
$\kontsevich:\Tpoly{}(\cV_{\formal})[1]\inftyto\Dpoly{}(\cV_{\formal})[1]$
for the $\ZZ$-graded vector space $\cV_{\formal}$
extends to a morphism of $L_\infty$ algebras
\begin{equation}\label{Tatsuki}
\id_{\Omega(\cV)} \otimes \kontsevich: \Omega(\cV) \cotimes_\KK \Tpoly{}(\cV_{\formal})[1]
\inftyto \Omega(\cV) \cotimes_\KK \Dpoly{}(\cV_{\formal})[1]
\end{equation}
and Shoikhet's $L_\infty$ morphism $\shoikhet: \kontsevich^\ast(\cCpoly{}(\cV_{\formal}))\inftyto\cApoly{}(\cV_{\formal})$
extends to a morphism of $L_\infty$ modules
\begin{equation}\label{Noriaki}
\id_{\Omega(\cV)} \otimes \shoikhet: (\id_{\Omega(\cV)} \otimes \kontsevich)^\ast\big(\Omega(\cV) \cotimes_\KK \cCpoly{}(\cV_{\formal})\big)
\inftyto \Omega(\cV) \cotimes_\KK \cApoly{}(\cV_{\formal})
\end{equation}
over the $L_\infty$ algebra
$\Omega(\cV) \cotimes_\KK \Tpoly{}(\cV_{\formal})[1]$.

It is immediate that, under the identifications \eqref{eq:Tpoly_loc} and \eqref{eq:Dpoly_loc},
the $L_\infty$ morphism $\id_{\Omega(\cV)}\otimes\kontsevich$
is the restriction to the open set $U$ of the $L_\infty$ morphism $\fkontsevich$
defined globally over $M$ --- see~\eqref{Masaki}.
Likewise, under the identifications \eqref{eq:Cpoly_loc} and~\eqref{eq:Apoly_loc},
$\id_{\Omega(\cV)}\otimes\shoikhet$ is the restriction to $U$ of the $\cF$-leafwise Shoikhet morphism
$\fshoikhet$ --- see~\eqref{Kenji}.

Consider the local section
\begin{equation}\label{eq:LocMC}
\localMC = \tauTpolyF(Q) + \fedosov - \sum_{k=1}^{\mr} \xi_k \frac{\partial}{\partial x_k}
\end{equation}
of $\cF\to\cN$.

Since the vector field $\tauTpolyF(Q)+\fedosov$ on $\cN$ is homological, we have
$(d_{\DR}+\localMC)^2=0$ and, consequently,
\[ \big(d_{\DR}\otimes\id_{\Tpoly{}(\cV_{\formal})[1]}\big)(\localMC)
+\frac{1}{2} \schouten{\localMC}{\localMC}=0 .\]
That is, the vector field $\localMC$, which is tangent to the foliation $\cF$,
is indeed a Maurer--Cartan element of the dgla $\Omega(\cV) \cotimes_\KK \Tpoly{}(\cV_{\formal})[1]$.

Twisting the $L_\infty$ morphisms \eqref{Tatsuki} and~\eqref{Noriaki} by the Maurer-Cartan element $\localMC$,
we obtain the pair of twisted morphisms
\begin{gather}
(\id_{\Omega(\cV)}\otimes\kontsevich)_{\localMC} :
\big(\Omega(\cV)\cotimes_\KK\Tpoly{}(\cV_{\formal})[1]\big)_{\localMC}
\inftyto \big(\Omega(\cV)\cotimes_\KK\Dpoly{}(\cV_{\formal})[1]\big)_Y
\\
(\id_{\Omega(\cV)}\otimes\shoikhet)_{\localMC} :
\big((\id_{\Omega(\cV)} \otimes \kontsevich)^\ast(\Omega(\cV) \cotimes_\KK \cCpoly{}(\cV_{\formal}))\big)_{\localMC}
\inftyto \big(\Omega(\cV)\cotimes_\KK\Apoly{}(\cV_{\formal})\big)_Y
,\end{gather}
where the Maurer--Cartan element $Y$ of the dgla $\Omega(\cV) \cotimes_\KK \Dpoly{}(\cV_{\formal})[1]$
is actually equal to
\[ Y = \sum_{k=1}^\infty \dfrac{1}{k!} \big(\id_{\Omega(\cV)}\otimes\kontsevich_{k}\big)(\localMC,\cdots,\localMC)
= \big(\id_{\Omega(\cV)}\otimes\kontsevich_{1}\big)(\localMC) = \localMC \]
according to Theorem~\ref{KontsevichFormality}~\ref{biloba}.

\begin{lemma}\label{lem:LooStr_loc}
The vector space isomorphisms \eqref{eq:Tpoly_loc} and~\eqref{eq:Dpoly_loc} identify the restrictions to $U$ of the dglas
\[ \big( \totTpolyF{\bullet}[1] \big)_{\tauQ} \quad\text{and}\quad \big( \totDpolyF{\bullet}[1] \big)_{\tauQ} \]
with the $\localMC$-twisted dglas
\[ \big(\Omega(\cV)\cotimes_\KK\Tpoly{}(\cV_{\formal})[1]\big)_{\localMC} \quad\text{and}\quad \big(\Omega(\cV)\cotimes_\KK\Dpoly{}(\cV_{\formal})[1]\big)_{\localMC}, \]
respectively. Furthermore, the vector space isomorphisms \eqref{eq:Cpoly_loc} and~\eqref{eq:Apoly_loc} identify the restrictions to $U$ of the $L_\infty$ modules
\[ \tkontsevich^\ast\big(\totcCpolyF{\bullet}\big)_{\tauQ} \quad\text{and}\quad \big(\totbApolyF{\bullet}\big)_{\tauQ} \]
over the dgla $\big( \totTpolyF{\bullet}[1] \big)_{\tauQ}$ with the $\localMC$-twisted $L_\infty$ modules
\[ \big((\id_{\Omega(\cV)} \otimes \kontsevich)^\ast(\Omega(\cV) \cotimes_\KK \cCpoly{}(\cV_{\formal}))\big)_{\localMC} \quad\text{and}\quad \big(\Omega(\cV)\cotimes_\KK\cApoly{}(\cV_{\formal})\big)_{\localMC} \]
over the dgla $\big(\Omega(\cV)\cotimes_\KK\Tpoly{}(\cV_{\formal})[1]\big)_{\localMC}$, respectively.
\end{lemma}

\begin{proof}
We verify the lemma for $\tkontsevich^\ast\big(\totcCpolyF{\bullet}\big)_{\tauQ}$, as the verifications for the remaining cases follow immediately from standard definitions. Let $\lambda^{C}_{n}$ and $\lambda^{C,\localMC}_{n}$ denote the $L_\infty$ module structures on $(\id_{\Omega(\cV)} \otimes \kontsevich)^\ast(\Omega(\cV) \cotimes_\KK \cCpoly{}(\cV_{\formal}))$ before and after twisting by $\localMC$, respectively. By definition, in any system of local coordinates on $\cN$ arising from a local chart $(U,\phi)$ of $\cM$, the twisted zeroth structure map evaluates to
\begin{align*}
& \lambda^{C,\localMC}_{0}(\zeta|_U) = \sum_{k=0}^\infty \frac{1}{k!} \lambda^{C}_{k}(\localMC, \cdots, \localMC;\zeta|_U) && \\
&\qquad = (d_{\DR} \otimes \id)(\zeta|_U) +\hochschildb(\zeta)|_U + \sum_{k=1}^\infty \frac{1}{k!} \iL_{(\id_{\Omega(\cV)}\otimes\kontsevich)_k(\localMC, \cdots, \localMC)}(\zeta|_U) && \\
&\qquad = (d_{\DR} \otimes \id)(\zeta|_U) +\hochschildb(\zeta)|_U + \iL_{\localMC}(\zeta|_U) && \text{by Theorem~\ref{KontsevichFormality}~\ref{LinVF-kontsevich}} \\
&\qquad = \iL_{\tauTpolyF(Q)+\fedosov+m_2}(\zeta) |_U = \lambda^{C,\tauQ}_0(\zeta)|_U && \text{by Equation~\eqref{eq:TwistedCpolyFedosov}}
\end{align*}
for all $\zeta \in \totCpolyF{\bullet}$.

For the higher structure maps, observe that according to Equations~\eqref{eq:globalOmega}, \eqref{eq:LocMC}, \eqref{Fhvf}, and~\eqref{eq:LocCovDer},
\begin{equation}\label{eq:LocMC-Fedosov}
\vfedosov|_U - \localMC = \sum_{l=1}^{\mr}\sum_{j=1}^{\mr} {(-1)^{\degree{y_l}+\degree{y_l}\degree{y_j}}} \xi_k \Gamma_{k,l}^j \otimes y_l\frac{\partial}{\partial y_j}
\end{equation}
is a \emph{linear} $\cF$-longitudinal vector field.
Therefore, it follows from Theorem~\ref{KontsevichFormality}~\ref{LinVF-kontsevich} that
\begin{align*}
\tkontsevich_n(\gamma_1, \cdots, \gamma_n)|_U
&= \sum_{k=0}^\infty \frac{1}{k!} \fkontsevich_{k+n}(\vfedosov, \cdots, \vfedosov, \gamma_1, \cdots, \gamma_n)|_U \\
&= \sum_{k=0}^\infty \frac{1}{k!} \big(\id_{\Omega(\cV)}\otimes\kontsevich\big)_{k+n} (\vfedosov|_U, \cdots, \vfedosov|_U, \gamma_1|_U, \cdots, \gamma_n|_U) \\
&= \sum_{k=0}^\infty \frac{1}{k!} \big(\id_{\Omega(\cV)}\otimes\kontsevich\big)_{k+n} (\localMC, \cdots, \localMC, \gamma_1|_U, \cdots, \gamma_n|_U)
\end{align*}
for all $\gamma_1, \cdots, \gamma_n \in \totTpolyF{\bullet}[1]$. That is, we have
\begin{equation}\label{eq:tKont-local}
\tkontsevich_n(\gamma_1, \cdots, \gamma_n)|_U = \big(\big(\id_{\Omega(\cV)}\otimes\kontsevich\big)_{\localMC}\big)_{n}(\gamma_1|_U, \cdots, \gamma_n|_U).
\end{equation}
Consequently, for $n \geq 1$, the higher structure maps yield
\begin{align*}
\lambda^{C,\localMC}_{n}(\gamma_1|_U, \cdots, \gamma_n|_U;\zeta|_U)
&= \sum_{k=0}^\infty \frac{1}{k!} \iL_{(\id_{\Omega(\cV)}\otimes\kontsevich)_{n+k}(\localMC, \cdots, \localMC, \gamma_1|_U, \cdots, \gamma_n|_U)}(\zeta|_U) \\
&= \iL_{\tkontsevich_n(\gamma_1, \cdots, \gamma_n)|_U}(\zeta|_U) = \lambda^{C,\tauQ}_{n}(\gamma_1, \cdots, \gamma_n;\zeta)|_U
\end{align*}
by Equation~\eqref{eq:TwistedCpolyFedosov}, which completes the proof.
\end{proof}

We are now ready to prove Proposition~\ref{prop:FormalityMor-Fedosov}.

\begin{proof}[Proof of Proposition~\ref{prop:FormalityMor-Fedosov}]
By Lemma~\ref{lem:LooStr_loc} and Equation~\eqref{eq:tKont-local}, the vector space isomorphisms \eqref{eq:Tpoly_loc} and \eqref{eq:Dpoly_loc} associated with any choice of a local coordinate system identify the local restrictions of the sequence of globally defined maps $(\tkontsevich_n)_{n\in\NN}$ with the Taylor coefficients of an $L_\infty$ algebra morphism $(\id_{\Omega(\cV)}\otimes\kontsevich)_{\localMC}$. Consequently, the globally defined maps $\tkontsevich_n$ are themselves the Taylor coefficients of a globally defined morphism of $L_\infty$ algebras $\tkontsevich$.

Similarly to the derivation of Equation~\eqref{eq:tKont-local}, the linearity of $\vfedosov|_U - \localMC$ and Theorem~\ref{ShoikhetFormality}~\ref{palmatum} imply that
\[ \tshoikhet_n(\gamma_1, \cdots, \gamma_n;\zeta)|_U = \big(\big(\id_{\Omega(\cV)}\otimes\shoikhet\big)_{\localMC}\big)_{n}(\gamma_1|_U, \cdots, \gamma_n|_U;\zeta|_U) \]
for all $\gamma_1, \cdots, \gamma_n \in \totTpolyF{\bullet}[1]$ and $\zeta \in \totCpolyF{\bullet}$. Therefore, the local restrictions of the maps $(\tshoikhet_n)_{n\in\NO}$ to any local coordinate system on $\cN$ arising from a local chart $(U,\phi)$ of $\cM$ form the locally defined morphism $\big(\id_{\Omega(\cV)}\otimes\shoikhet\big)_{\localMC}$ of $L_\infty$ modules. This observation, together with Lemma~\ref{lem:LooStr_loc}, implies that the maps $\tshoikhet_n$ are themselves the Taylor coefficients of a globally defined morphism $\tshoikhet$ of $L_\infty$ modules from $\tkontsevich^\ast\big(\totcCpolyF{\bullet}\big)_{\tauQ}$ to $\big(\totbApolyF{\bullet}\big)_{\tauQ}$. This completes the proof.
\end{proof}

\subsection{\texorpdfstring{The maps $\tkontsevich_1$ and $\tshoikhet_0$}{The maps tkontsevich and tshoikhet}}

The aim of this section is to prove the following
\begin{proposition}\label{prop:Florence}
The first ``Taylor coefficient'' of the $L_\infty$ algebra morphism $\tkontsevich$ is the composition
\begin{equation} \label{Fukuoka} \tkontsevich_1 = \hkr \circ (\Ahatcocycle{\can}{\cF_{\cQ}})^{\frac{1}{2}} \end{equation}
of the action of the square root of the \Aroof\ cocycle on polyvector fields (by contraction)
with the cohomological HKR map.
Similarly, the zeroth ``Taylor coefficient'' of the $L_\infty$ module morphism $\tshoikhet$ is the composition
\begin{equation} \label{Fukushima} \tshoikhet_0 = (\Ahatcocycle{\can}{\cF_{\cQ}})^{\frac{1}{2}} \circ \hhkr \end{equation}
of the homological HKR map with the action of the square root of the \Aroof\ cocycle
on differential forms (by multiplication).
\end{proposition}

As before, let $\phi:\cM|_U\to\cV$ be a trivialization of the restriction of $\cM$ to an open set $U$
and let $\psi:\cN|_U\to\cW$ be the induced trivialization of the restriction of $\cN$ to the same open set $U$
--- see~\eqref{eq:LocChartOnM} and~\eqref{eq:LocChartOnN}.
Consider the Kontsevich formality morphism and Shoikhet formality morphism
on the trivialized graded manifold $\cW$:
\[ \kontsevich': \Tpoly{}(\cW)[1] \inftyto \Dpoly{}(\cW)[1]
\qquad\text{and}\qquad
\shoikhet': \cCpoly{}(\cW) \inftyto \cApoly{}(\cW) .\]

The restriction to $\cW=\cN|_U$ of the homological vector field $\tauQ=\tauTpolyF(Q)+\fedosov$
is a Maurer-Cartan element of the dgla $\Tpoly{}(\cW)[1]$,
which we use to twist the $L_\infty$ morphisms $\kontsevich'$ and $\shoikhet'$.
Propositions~\ref{prop:DK_vs} and~\ref{prop:DKS_vs} assert that
the first Taylor coefficients of the resulting twisted $L_\infty$ morphisms
$\kontsevich'_{\tauQ}$ and $\shoikhet'_{\tauQ}$ are
\begin{equation}\label{eq:DK_Kont'_Shoi'}
(\kontsevich'_{\tauQ})_1 = \hkr \circ (\Ahatcocycle{\std}{(\cW,\tauQ)})^{\frac{1}{2}}
\qquad\text{and}\qquad
(\shoikhet'_{\tauQ})_0 = (\Ahatcocycle{\std}{(\cW,\tauQ)})^{\frac{1}{2}} \circ \hhkr
.\end{equation}

Consider the natural inclusions
\begin{equation}\label{eq:cite}
\cI: \Tpolyf{} \hookrightarrow \Tpoly{}(\cN)
\qquad\text{and}\qquad
\cJ: \Dpolyf{} \hookrightarrow \Dpoly{}(\cN)
\end{equation}
of $\cF$-longitudinal polyvector fields and polydifferential operators into
all polyvector fields and polydifferential operators
and the dual surjections
\begin{equation}\label{eq:Opera}
\cI^\top: \Apoly{}(\cN) \twoheadrightarrow \Apolyf{}
\qquad\text{and}\qquad
\cJ^\top: \Cpoly{}(\cN) \twoheadrightarrow \Cpolyf{}
\end{equation}
for differential forms and polyjets.

\begin{lemma}\label{otowa}
For all $\gamma_1, \cdots, \gamma_n \in \Tpolyf{}[1]$ and $\zeta \in \cCpolyf{}$, we have
\[ \kontsevich_n'\big(\cI(\gamma_1)|_U, \cdots, \cI(\gamma_n)|_U\big) = \big(\cJ \circ \fkontsevich_n(\gamma_1, \cdots, \gamma_n)\big)\big|_U \]
and
\[ \cI^\top \circ \shoikhet_n'\big(\cI(\gamma_1)|_U, \cdots, \cI(\gamma_n)|_U;\zeta|_U\big) = \fshoikhet_n\big(\gamma_1, \cdots, \gamma_n;\cJ^\top(\zeta)\big)\big|_U. \]
\end{lemma}

\begin{proof}
The assertions follow from an inspection of the construction of $\kontsevich'$, $\fkontsevich$, $\shoikhet'$, and $\fshoikhet$.
\end{proof}

\begin{lemma}\label{Hakodate}
For all $\gamma\in\Tpolyf{}[1]$ and $\zeta\in\cCpolyf{}$, we have
\[ (\kontsevich'_{\tauQ})_1\big(\cI(\gamma)|_U\big) = \big(\cJ\circ\tkontsevich_1(\gamma)\big)\big|_U \]
and
\[ \cJ^\top\circ(\shoikhet'_{\tauQ})_0\big(\zeta|_U\big) = \big(\tshoikhet_0\circ\cI^\top(\zeta)\big)\big|_U .\]
\end{lemma}

\begin{proof}
Note that the difference of the vector fields $\tauQ$ and $\cI(\localMC)$,
respectively defined by Equations~\eqref{eq:cQ-Fedosov} and~\eqref{eq:LocMC},
is a linear vector field on $\cW$.
Therefore, we have
\begin{align*}
(\kontsevich'_{\tauQ})_1\big(\cI(\gamma)|_U\big) & = \sum_{k=0}^\infty \dfrac{1}{k!}
\kontsevich'_{1+k}\big(\tauQ|_U,\cdots,\tauQ|_U,\cI(\gamma)|_U\big) && \\
& = \sum_{k=0}^\infty \dfrac{1}{k!}
\kontsevich'_{1+k}\big(\cI(\localMC)|_U,\cdots,\cI(\localMC)|_U,\cI(\gamma)|_U\big)
&& \text{by Theorem~\ref{KontsevichFormality}~\ref{LinVF-kontsevich}} \\
& = \sum_{k=0}^\infty \dfrac{1}{k!} \cJ\big(\fkontsevich_{1+k}(\localMC,\cdots,\localMC,\gamma) \big)\big|_U = \big(\cJ\circ\tkontsevich_1(\gamma)\big)\big|_U
&& \text{by Lemma~\ref{otowa}}
\end{align*}
and
\begin{align*}
\cI^\top\circ(\shoikhet'_{\tauQ})_0(\zeta|_U) & = \sum_{k=0}^\infty \dfrac{1}{k!}
\cI^\top\circ\shoikhet'_{k}(\tauQ|_U,\cdots,\tauQ|_U;\zeta|_U) && \\
& = \sum_{k=0}^\infty \dfrac{1}{k!}
\cI^\top\circ\shoikhet'_{k}(\cI(\localMC)|_U,\cdots,\cI(\localMC)|_U;\zeta|_U)
&& \text{by Theorem~\ref{ShoikhetFormality}~\ref{palmatum}} \\
& = \sum_{k=0}^\infty \dfrac{1}{k!}
\fshoikhet_{k}\big(\localMC,\cdots,\localMC;\cJ^\top(\zeta)\big)\big|_U = \big(\tshoikhet_0\circ\cJ^\top(\zeta)\big)\big|_U
&& \text{by Lemma~\ref{otowa}}
\end{align*}
for all $\gamma\in\Tpolyf{}[1]$ and $\zeta\in\cCpolyf{}$.
\end{proof}

\begin{lemma}\label{lem:ToddFedosov}
In $\totbApolyF{0}\big|_U$, we have the equality
\[ \cI\transpose\big(\Ahatcocycle{\std}{(\cW,\tauQ)}\big) = \Ahatcocycle{\can}{\cF_{\cQ}|_U} .\]
\end{lemma}

\begin{proof}
Since the \Aroof\ cocycle
can be expressed in terms of scalar Atiyah cocycles, it suffices to prove that, for all $l\in\NN$,
we have the equality
\[ \supertrace\big((\atiyahcocycle{\can}{\cF_{\cQ}|_U})^l\big)
= \cI\transpose \big(\supertrace\big((\atiyahcocycle{\std}{(\cW,\tauQ)})^l\big)\big) \]
in $\Gamma\big(\cN|_U;S^l(\cF|_U^\vee[1])\big)^0$
--- recall that
\begin{gather*}
\atiyahcocycle{\can}{\cF_{\cQ}|_U} \in \Gamma\big(\cN|_U;\cF|_U^\vee \otimes \End(\cF|_U)\big)^1
= \Gamma\big(\cN|_U;S^1 (\cF|_U^\vee[1]) \otimes \End (\cF|_U)\big)^0 \\
\atiyahcocycle{\std}{(\cW,\tauQ)} \in \Gamma\big(\cW; T^\vee_{\cW} \otimes \End (T_{\cW})\big)^1
= \Gamma\big(\cW;S^1 (T^\vee_{\cW}[1]) \otimes \End (T_{\cW})\big)^0 .
\end{gather*}

The system of local coordinates
\[ (z_1,\cdots, z_{3 \mr}) = (x_1,\cdots, x_{\mr}, \xi_1,\cdots, \xi_{\mr}, y_1, \cdots, y_{\mr}) \]
on $\cN|_U=\cW$ mentioned earlier --- see~\eqref{eq:InducedCoordFedosovMfd} ---
induces dual local frames for $T_{\cW}$ and $T^\vee_{\cW}$ respectively:
\[ \frac{\partial}{\partial z_1}, \cdots, \frac{\partial}{\partial z_{3\mr}} \in \sections{\cW; T_{\cW}}; \qquad d z_1,\cdots, d z_{3\mr} \in \sections{\cW; T^\vee_{\cW}} .\]
Note that $\degree{d z_i} = \degree{z_i} = -\degree{\frac{\partial}{\partial z_i}}$.

Recall that, in this system of local coordinates, the $\cF$-longitudinal vector field $\tauTpolyF(Q)$ reads
\[ \tauTpolyF(Q) = \sum_{j=1}^{\mr} \sum_{L\in\ZZ_{\geq 0}^{\mr}} g_{L,j} y^L \frac{\partial}{\partial y_j} ,\]
where the $g_{L,j} \in C^\infty(\cW)$ are functions depending on the variables $x_1,\cdots,x_{\mr}$ only.
It follows from Equations~\eqref{Fhvf}, \eqref{Cotonou}, \eqref{eq:LocCovDer} and~\eqref{Conakry} that
the cohomological vector field $\tauQ=\tauTpolyF(Q)+\fedosov$ reads
\[ \tauQ = \sum_{j=1}^{3 \mr} \tauQ(z_j) \frac{\partial}{\partial z_j} = \sum_{j=1}^{\mr} \tauQ(x_j) \frac{\partial}{\partial x_j} + \sum_{j=1}^{\mr} \tauQ(\xi_j) \frac{\partial}{\partial \xi_j}+ \sum_{j=1}^{\mr} \tauQ(y_j) \frac{\partial}{\partial y_j} \]
with \[ \tauQ(x_j) = \xi_j, \qquad \tauQ(\xi_j) = 0 ,\]
and
\[ \tauQ(y_j) = \sum_{L\in\ZZ_{\geq 0}^{\mr}}g_{L,j} y^L - \xi_j - \sum_{k}^{\mr}\sum_{l}^{\mr} (-1)^{|y_l|+ |y_l||y_j|} \xi_k \Gamma_{k,l}^j y_l
+\sum_{k=1}^{\mr} \sum_{\substack{L\in\ZZ_{\geq 0}^{\mr} \\ \abs{L}\geq 2}}\xi_k A^k_{L,j} y^L ,\]
where the functions $g_{L,j}$, $\Gamma_{k,l}^j$, and $A^k_{L,j}$ depend on the variables $x_1,\cdots,x_{\mr}$ only.

According to Equation~\eqref{velociraptor}, we have
\[ \atiyahcocycle{\std}{(\cW,\tauQ)} = \sum_{i,j,k = 1}^{3\mr} \pm
\frac{\partial^2\big(\tauQ(z_j)\big)}{\partial z_i \partial z_k} \nshift(d z_i)
\otimes \left(\frac{\partial}{\partial z_j} \otimes dz_k \right) .\]

Working in the frame $\frac{\partial}{\partial z_1}, \cdots, \frac{\partial}{\partial z_{3\mr}}$ for $T_{\cW}$,
we can think of the endomorphism $\atiyahcocycle{\std}{(\cW,\tauQ)}$ as a $3\mr\times3\mr$ matrix $\mathsf{M}$
with coefficients in $\Gamma\big(\cW;S^1 (T^\vee_{\cW}[1])\big)$ and rows and columns indexed by the
variables $z_1,z_2,\cdots,z_{3\mr}$:
\[ \mathsf{M}_{z_j z_k} = \sum_{i=1}^{3\mr} \pm
\frac{\partial^2\big(\tauQ(z_j)\big)}{\partial z_i \partial z_k} \nshift(d z_i) .\]
Since the variables $(z_1,z_2,\cdots,z_{3\mr})$ are divided in the three subgroups
$(x_1,x_2,\cdots,x_{\mr})$; $(\xi_1,\xi_2,\cdots,\xi_{\mr})$; and $(y_1,y_2,\cdots,y_{\mr})$,
the matrix $\mathsf{M}$ decomposes into nine blocks of size $\mr\times\mr$:
\[ \mathsf{M}=\begin{bmatrix} \mathsf{M}_{XX} & \mathsf{M}_{X\Xi} & \mathsf{M}_{XY} \\
\mathsf{M}_{\Xi X} & \mathsf{M}_{\Xi\Xi} & \mathsf{M}_{\Xi Y} \\
\mathsf{M}_{YX} & \mathsf{M}_{Y\Xi} & \mathsf{M}_{YY} \end{bmatrix} \]

Since $\tauQ(x_j)=\xi_j$, we have $\mathsf{M}_{XX} = \mathsf{M}_{X\Xi} = \mathsf{M}_{XY} = 0$.
Furthermore, since $\tauQ(\xi_j)=0$, we have
$\mathsf{M}_{\Xi X} = \mathsf{M}_{\Xi\Xi} = \mathsf{M}_{\Xi Y} = 0$.
Therefore, we have
\[ \mathsf{M} = \begin{bmatrix} 0 & 0 & 0 \\ 0 & 0 & 0 \\ \mathsf{M}_{YX} & \mathsf{M}_{Y\Xi} & \mathsf{M}_{YY} \end{bmatrix} \]
and \[ \supertrace\big(\mathsf{M}^l\big) = \supertrace\big((\mathsf{M}_{YY})^l\big), \quad\forall l\in\NN .\]

On the other hand, we have
\[ \atiyahcocycle{\can}{\cF_{\cQ}|_U} = \sum_{i,j,k = 1}^{\mr} \pm
\frac{\partial^2\big(\tauQ(y_j)\big)}{\partial y_i \partial y_k} \nshift(d y_i)
\otimes \left(\frac{\partial}{\partial y_j} \otimes dy_k \right) \]
and, working in the frame $\frac{\partial}{\partial y_1},\cdots,\frac{\partial}{\partial y_{\mr}}$
for $\cF_{\cQ}|_U$, we can think of the endomorphism $\atiyahcocycle{\can}{\cF_{\cQ}|_U}$
as a $\mr\times\mr$ matrix $\mathsf{F}$ with coefficients in $\Gamma\big(\cW;S^1(\cF|_U^\vee[1])\big)$
and rows and columns indexed by the variables $y_1,y_2,\cdots,y_{\mr}$:
\[ \mathsf{F}_{y_j y_k} = \sum_{i=1}^{\mr} \pm
\frac{\partial^2\big(\tauQ(y_j)\big)}{\partial y_i \partial y_k} \nshift(d y_i) .\]

Since
\[ \cI\transpose(\mathsf{M}_{y_j y_k}) = \cI\transpose\left( \sum_{i=1}^{3\mr} \pm
\frac{\partial^2\big(\tauQ(y_j)\big)}{\partial z_i \partial y_k} \nshift(d z_i) \right)
= \sum_{i=1}^{\mr} \pm
\frac{\partial^2\big(\tauQ(y_j)\big)}{\partial y_i \partial y_k} \nshift(d y_i)
= \mathsf{F}_{y_j y_k} ,\]
we conclude that
\[ \cI\transpose \supertrace\big((\atiyahcocycle{\std}{T_{(\cW,\tauQ)}})^l\big)
= \cI\transpose \supertrace\big(\mathsf{M}^l\big)
= \cI\transpose \supertrace\big((\mathsf{M}_{YY})^l\big)
= \supertrace\big(\mathsf{F}^l\big)
= \supertrace\big((\atiyahcocycle{\can}{\cF_{\cQ}|_U})^l\big) \]
and the proof is complete.
\end{proof}

We have the following immediate corollary:

\begin{corollary}\label{cor:local td}
The diagrams
\[ \begin{tikzcd}
\cApoly{}(\cW) \arrow[r, "\Ahatcocycle{\std}{(\cW,\tauQ)}"] \arrow[d, "\cI^\top"', two heads]
& \cApoly{}(\cW) \ar[d,"\cI\transpose"', two heads] \\
\cApolyf{}\big|_U \arrow[r, "\Ahatcocycle{\can}{\cF_{\cQ}|_U}", swap]
& \cApolyf{}\big|_U
\end{tikzcd}
\quad\text{and}\quad
\begin{tikzcd}
\Tpoly{}(\cW) \arrow[r, "\Ahatcocycle{\std}{(\cW,\tauQ)}"] & \Tpoly{}(\cW) \\
\Tpolyf{
}\big|_U \arrow[u, "\cI", hook]
\arrow[r, "\Ahatcocycle{\can}{\cF_{\cQ}|_U}", swap] &
\Tpolyf{}\big|_U \arrow[u, "\cI", hook]
\end{tikzcd} \]
commute.
The \Aroof\ cocycles $\Ahatcocycle{\std}{(\cW,\tauQ)}$ and $\Ahatcocycle{\can}{\cF_{\cQ}|_U}$ act by multiplication on differential forms
and by contraction on polyvector fields.
\end{corollary}

The next lemma is immediate.

\begin{lemma}\label{Sapporo}
The diagrams
\[ \begin{tikzcd}
\cCpoly{}(\cW) \arrow[r, "\hhkr"] \arrow[d, "\cJ^\top"', two heads]
& \cApoly{}(\cW) \ar[d,"\cI\transpose"', two heads] \\
\cCpolyf{}\big|_U \arrow[r, "\hhkr", swap]
& \cApolyf{}\big|_U
\end{tikzcd}
\quad\text{and}\quad
\begin{tikzcd}
\Tpoly{}(\cW) \arrow[r, "\hkr"] & \Dpoly{}(\cW) \\
\Tpolyf{}\big|_U \arrow[u, "\cI", hook]
\arrow[r, "\hkr", swap] &
\Dpolyf{}\big|_U \arrow[u, "\cJ", hook]
\end{tikzcd} \]
commute.

\end{lemma}

\begin{proof}[Proof of Proposition~\ref{prop:Florence}]
It suffices to verify Equations~\eqref{Fukuoka} and~\eqref{Fukushima} locally over any open subset $U$ of the support manifold $M$.

Since
\begin{align*}
\cJ \circ \tkontsevich_1 & = (\kontsevich'_{\tauQ})_1\circ\cI && \text{by Lemma~\ref{Hakodate}} \\
& = \hkr \circ (\Ahatcocycle{\std}{(\cW,\tauQ)})^{\frac{1}{2}}\circ\cI
&& \text{by Equation~\eqref{eq:DK_Kont'_Shoi'}} \\
& = \hkr \circ \cI \circ (\Ahatcocycle{\can}{\cF_{\cQ}|_U})^{\frac{1}{2}}
&& \text{by Corollary~\ref{cor:local td}} \\
& = \cJ \circ \hkr \circ (\Ahatcocycle{\can}{\cF_{\cQ}|_U})^{\frac{1}{2}}
&& \text{by Lemma~\ref{Sapporo}}
\end{align*}
and
\begin{align*}
\tshoikhet_0\circ\cJ^\top & = \cI^\top\circ(\shoikhet'_{\tauQ})_0 && \text{by Lemma~\ref{Hakodate}} \\
& = \cI^\top\circ(\Ahatcocycle{\std}{(\cW,\tauQ)})^{\frac{1}{2}}\circ\hhkr
&& \text{by Equation~\eqref{eq:DK_Kont'_Shoi'}} \\
& = (\Ahatcocycle{\can}{\cF_{\cQ}|_U})^{\frac{1}{2}}\circ\cI^\top\circ\hhkr
&& \text{by Corollary~\ref{cor:local td}} \\
& = (\Ahatcocycle{\can}{\cF_{\cQ}|_U})^{\frac{1}{2}}\circ\hhkr\circ\cJ^\top
&& \text{by Lemma~\ref{Sapporo}},
\end{align*}
we have
$\cJ \circ \big( \tkontsevich_1 - \hkr \circ (\Ahatcocycle{\can}{\cF_{\cQ}})^{\frac{1}{2}} \big)=0$
and
$\big( \tshoikhet_0 - (\Ahatcocycle{\can}{\cF_{\cQ}})^{\frac{1}{2}}\circ\hhkr \big) \circ\cJ^\top=0$.

It follows from the injectivity of $\cJ$ that $\tkontsevich_1
= \hkr \circ (\Ahatcocycle{\can}{\cF_{\cQ}})^{\frac{1}{2}}$
and from the surjectivity of $\cJ^\top$ that $\tshoikhet_0
= (\Ahatcocycle{\can}{\cF_{\cQ}})^{\frac{1}{2}}\circ\hhkr$.
\end{proof}

\begin{corollary}\label{Yonago}
The $L_\infty$ morphisms $\tkontsevich$ and $\tshoikhet$ are $L_\infty$ \emph{quasi}-isomorphisms.
\end{corollary}

\begin{proof}
It follows from Equations~\eqref{Fukuoka} and~\eqref{Fukushima} and Corollary~\ref{cor:HKR-Fedosov}
that the cochain maps $\tkontsevich_1$ and $\tshoikhet_0$ are quasi-isomorphisms.
\end{proof}

\subsection{Compatibility with the calculus structures}

In this subsection, we prove that the isomorphisms in cohomology induced by the $L_\infty$ quasi-isomorphisms
$\tkontsevich$ and $\tshoikhet$ define an isomorphism of calculi:

\begin{proposition}\label{prop:CompatibleWithCal_Fedosov}
The pair of maps
\begin{multline} \label{eq:tKont1_coh}
\hkr\circ(\Ahatclass{\cF_{\cQ}})^{\frac{1}{2}}:
\cohomology{}\big(\totTpolyF{\bullet},\schouten{\tauTpolyF(Q)+\fedosov}{\argument}\big) \\
\xto{\cong} \cohomology{}\big(\totDpolyF{\bullet},\gerstenhaber{\tauTpolyF(Q)+\fedosov+m_2}{\argument}\big)
\end{multline}
and
\begin{multline}\label{eq:tShoi0_coh_inv}
((\Ahatclass{\cF_{\cQ}})^{\frac{1}{2}} \circ
\hhkr)^{-1}:\cohomology{}\big(\totbApolyF{\bullet},\iL_{\tauTpolyF(Q)+\fedosov}\big) \xto{\cong} \cohomology{}\big(\totcCpolyF{\bullet},\iL_{\tauTpolyF(Q)+\fedosov+m_2}\big)
\end{multline}
is an isomorphism of calculi.
\end{proposition}

We first need the following

\begin{lemma}\label{Kyushu}
\begin{enumerate}
\item
The map \eqref{eq:tKont1_coh} is a morphism of associative algebras.
\item
The map \eqref{eq:tShoi0_coh_inv} is a morphism of modules over the associative algebra
$\cohomology{}\big(\totTpolyF{\bullet},\schouten{\tauTpolyF(Q)+\fedosov}{\argument}\big)$.
\end{enumerate}
\end{lemma}
\begin{proof}
Under the identifications \eqref{eq:Tpoly_loc}, \eqref{eq:Dpoly_loc}, \eqref{eq:Apoly_loc}
and~\eqref{eq:Cpoly_loc} induced by a local chart $(U,\phi)$ on $\cM$,
the restrictions to the open set $U$ of the $L_\infty$ morphisms $\tkontsevich$ and $\tshoikhet$ are identified
with $(\id_{\OO(\cV)}\otimes\kontsevich)_{\localMC}$ and $(\id_{\OO(\cV)}\otimes\shoikhet)_{\localMC}$ respectively.
According to Propositions~\ref{prop:HmtOpKont} and~\ref{prop:HmtOpShoi},
there exist locally defined maps
\[ H_{\localMC}^{\tkontsevich}: \totTpolyF{\bullet} \Big|_U \times \totTpolyF{\bullet} \Big|_U \to \totDpolyF{\bullet} \Big|_U \]
and
\[ H_{\localMC}^{\tshoikhet}: \totTpolyF{\bullet} \Big|_U \times \totcCpolyF{\bullet} \Big|_U \to \totbApolyF{\bullet} \Big|_U \]
satisfying homotopy relations parallel to Equations~\eqref{eq:HtpEqKont} and~\eqref{eq:HtpEqShoi}, respectively.
Note that, if $(U',\phi')$ is another local chart such that $U \cap U' \neq \emptyset$,
then the difference $\localMC_U|_{U \cap U'} - \localMC_{U'}|_{U \cap U'}$ is,
according to Equation~\eqref{eq:LocMC-Fedosov}, a linear vector field.
According to Propositions~\ref{prop:HmtOpKont} and~\ref{prop:HmtOpShoi},
the locally defined operators $H_{\localMC}^{\tkontsevich}$ and $H_{\localMC}^{\tshoikhet}$
are therefore independent of the choice of local chart.
Consequently, they glue together uniquely to form a pair of well-defined global operators
\[ H_{\localMC}^{\tkontsevich}: \totTpolyF{\bullet} \times \totTpolyF{\bullet} \to \totDpolyF{\bullet} \]
and
\[ H_{\localMC}^{\tshoikhet}: \totTpolyF{\bullet} \times \totcCpolyF{\bullet} \to \totbApolyF{\bullet} \]
satisfying the relations
\begin{multline}\label{Honshu}
\pshift\big(\tkontsevich_1 ( \nshift(\gamma_1 \odot \gamma_2) )\big) - \pshift\big(\tkontsevich_1(\nshift\gamma_1)\big) \cupproduct \pshift\big(\tkontsevich_1(\nshift\gamma_2)\big)
= \gerstenhaber{\tauTpolyF(Q)+\fedosov+m_2}{H_{\localMC}^{\tkontsevich}(\gamma_1,\gamma_2)} \\
+ H_{\localMC}^{\tkontsevich}\big(\schouten{\tauTpolyF(Q)+\fedosov}{\gamma_1},\gamma_2\big)
+ (-1)^{|\gamma_1|} H_{\localMC}^{\tkontsevich} \big(\gamma_1,\schouten{\tauTpolyF(Q)+\fedosov}{\gamma_2}\big),
\end{multline}
and
\begin{multline}\label{Shikoku}
\tshoikhet_0 \big(\iI_{\pshift\tkontsevich_1(\nshift\gamma)} (\zeta) \big) - \iI_\gamma \big( \tshoikhet_0(\zeta) \big)
= ( \iL_{\tauTpolyF(Q)+\fedosov})\big(H_{\localMC}^{\tshoikhet}(\gamma; \zeta) \big) \\
\qquad\quad + H_{\localMC}^{\tshoikhet}\big( \schouten{\tauTpolyF(Q)+\fedosov}{\gamma} ; \zeta\big)
+ (-1)^{|\gamma|} H_{\localMC}^{\tshoikhet}\big(\gamma;\iL_{\tauTpolyF(Q)+\fedosov+m_2}(\zeta)\big),
\end{multline}
for all $\gamma,\gamma_1,\gamma_2\in\totTpolyF{\bullet}$ and $\zeta\in\totcCpolyF{\bullet}$.
Upon applying the cohomology functor to Equations~\eqref{Honshu} and~\eqref{Shikoku},
the desired result follows immediately from Equations~\eqref{Fukuoka} and~\eqref{Fukushima}.
\end{proof}

We are now ready to complete the proof of Proposition~\ref{prop:CompatibleWithCal_Fedosov}.

\begin{proof}[Proof of Proposition~\ref{prop:CompatibleWithCal_Fedosov}]
According to Proposition~\ref{prop:Florence} and Corollary~\ref{Yonago},
the quasi-isomorphisms $\tkontsevich_1=\hkr\circ(\Ahatclass{\cF_{\cQ}})^{\frac{1}{2}}$
and $\tshoikhet_0=(\Ahatclass{\cF_{\cQ}})^{\frac{1}{2}} \circ \hhkr$ are respectively
the first Taylor coefficient of the morphism of $L_\infty$ algebras $\tkontsevich$
and the zeroth Taylor coefficient of the morphism of $L_\infty$ modules $\tshoikhet$.
Consequently, on the cohomology level, the map \eqref{eq:tKont1_coh} is an isomorphism of Lie algebras
and the map \eqref{eq:tShoi0_coh_inv} is an isomorphism of Lie modules (over the Lie algebra
$\cohomology{}\big(\totTpolyF{\bullet}[1],\schouten{\tauTpolyF(Q)+\fedosov}{\argument}\big)$).
Furthermore, Lemma~\ref{Kyushu} asserts that the map \eqref{eq:tKont1_coh} is a morphism
of associative algebras (since it preserves the cup products) and the map \eqref{eq:tShoi0_coh_inv}
is a morphism of modules (over the associative algebra
$\cohomology{}\big(\totTpolyF{\bullet},\schouten{\tauTpolyF(Q)+\fedosov}{\argument}\big)$).
Finally, it follows from Proposition~\ref{prop:CompatibleWithB_local} that the map \eqref{eq:tShoi0_coh_inv}
intertwines the operators $B$ acting on polyjets and $d_{\dR}$ acting on differential forms.
\end{proof}

\subsection{From the \Aroof\ class to the Todd class}

To prove Proposition~\ref{thm:DK_Fedosov}, note that the difference between Propositions~\ref{prop:CompatibleWithCal_Fedosov} and~\ref{thm:DK_Fedosov}
is the modification from $\Ahatclass{\cF_{\cQ}}$ to $\toddclass{\cF_{\cQ}}$.
On the cochain level, the cocycles $\Ahatcocycle{\can}{\cF_{\cQ}}$ and $\toddcocycle{\can}{\cF_{\cQ}}$
differ by a factor $e^{\frac{1}{2}\supertrace(\atiyahcocycle{\can}{\cF_{\cQ}})}$:
\[ \toddcocycle{\can}{\cF_{\cQ}} = e^{\frac{1}{2}\supertrace(\atiyahcocycle{\can}{\cF_{\cQ}})}
\cdot \Ahatcocycle{\can}{\cF_{\cQ}} .\]

In a system of coordinates of type \eqref{eq:InducedCoordFedosovMfd} on the Fedosov manifold $\cN$
associated with a local chart $(U,\phi)$ of $\cM$, consider the $\cF$-longitudinal vector field
\[ \tauTpolyF(Q) + A^\nabla = \sum_{j=1}^{\mr} f_j \frac{\partial}{\partial y_j}
\qquad \text{with $f_j \in C^\infty(\cN|_U)$} ,\]
where $A^\nabla$ is the $\cF$-longitudinal vector field on $\cN$ appearing
in Equations~\eqref{Fhvf} and~\eqref{Conakry}.
Its divergence relatively to the locally defined $\cF$-longitudinal volume form
$d y_1 \wedge \cdots \wedge dy_{\mr} \in \sections{\Lambda^{\mr}\cF\dual}$
is the function
\begin{equation}\label{locdivfunc}
\divergence(\tauTpolyF(Q) + A^\nabla) = \sum_{j=1}^{\mr}
\frac{\partial f_j}{\partial y_j} \in C^\infty(\cN|_U)
.\end{equation}
It is straightforward to see that, if $(x_1,\cdots,x_{\mr},\xi_1,\cdots,\xi_{\mr},y_1,\cdots,y_{\mr})$
and $(x'_1,\cdots,x'_{\mr},\xi'_1,\cdots,\xi'_{\mr},y'_1,\cdots,y'_{\mr})$
are two systems of local coordinates on $\cN$
associated with two overlapping local charts $(U,\phi)$ and $(U',\phi')$ of $\cM$,
the partial Jacobian matrix $\frac{\partial(y'_1,\cdots,y'_{\mr})}{\partial(y_1,\cdots,y_{\mr})}$
depends exclusively on the variables $x_1,\cdots,x_{\mr}$ on the open set $U\cap U'$.
Therefore, the local divergence function \eqref{locdivfunc} is independent
of the choice of local coordinates \eqref{eq:InducedCoordFedosovMfd}
on the Fedosov manifold $\cN$. Consequently, the local divergence functions
of the globally defined $\cF$-longitudinal vector field $\tauTpolyF(Q)+A^\nabla$ glue together
to yield a globally well-defined divergence function
$\divergence(\tauTpolyF(Q)+A^\nabla)$.

Similarly to \cite[Lemma~2.19]{MR3964152}, we have

\begin{lemma}\label{lem:ScalarAtiyah-Div}
$\supertrace\big(\atiyahcocycle{\can}{\cF_{\cQ}}\big) = d_{\cF}\big(\divergence(\tauTpolyF(Q)+A^\nabla)\big).$
\end{lemma}

\begin{proof}
In a system of local coordinates of type \eqref{eq:InducedCoordFedosovMfd} on $\cN$,
a direct computation shows that
\[ \atiyahcocycle{\can}{\cF_{\cQ}}\left(\dfrac{\partial}{\partial y_i},\dfrac{\partial}{\partial y_j}\right)
= (-1)^{|y_i|+|y_j|} \sum_{k=1}^{n} \dfrac{\partial^2 \tauQ(y_k)}{\partial y_i \partial y_j}
\dfrac{\partial}{\partial y_k}
= (-1)^{|y_i|+|y_j|} \sum_{k=1}^{n} \dfrac{\partial^2 f_k}{\partial y_i \partial y_j}
\dfrac{\partial}{\partial y_k} .\]
Therefore, we obtain
\begin{align*}
\iI_{\frac{\partial}{\partial y_i}} \supertrace\big(\atiyahcocycle{\can}{\cF_{\cQ}}\big) & = (-1)^{|y_i|}
\supertrace\left(\atiyahcocycle{\can}{\cF_Q}\left(\dfrac{\partial}{\partial y_i},\argument\right)\right)
= \sum_{j=1}^n \dfrac{\partial^2 f_j}{\partial y_i \partial y_j} \\
& = \iI_{\frac{\partial}{\partial y_i}} d_{\cF}\left(\sum_{j=1}^n \frac{\partial f_j}{\partial y_j}\right)
= \iI_{\frac{\partial}{\partial y_i}} d_{\cF}\big(\divergence(\tauTpolyF(Q)+A^\nabla)\big).
\end{align*}
This completes the proof.
\end{proof}

We also need the following
\begin{lemma} \label{cherry}
Given a dg Lie algebroid $\cL \to \cM$, let $\cQ$ denote the endomorphism of $\Gamma\big(\hat{S}(\cL^\vee[1])\big)$ encoding the dg structure, and let $d_{\cL} : \Gamma\big(\hat{S}^\bullet(\cL^\vee[1])\big) \to \Gamma\big(\hat{S}^{\bullet+1}(\cL^\vee[1])\big)$ be the Chevalley--Eilenberg differential of the underlying (graded) Lie algebroid.
Suppose that $\alpha \in \big(\Gamma(\cL^\vee[1])\big)^0$ is a $\cL$-longitudinal $1$-form of degree $0$ on $\cM$ satisfying $\cQ(\alpha) = 0$ and $d_{\cL}(\alpha) = 0$.
Denote by $\alpha^\flat$ the action of $\alpha$ on the space $\Theta^\bullet = \tot_{\Pi}^\bullet\big(\prescript{\cL}{}{\sA}^{\poly}_{\bullet}(\cM)\big)$ of $\cL$-longitudinal differential forms by multiplication.
Likewise, denote by $\alpha^\sharp$ the action of $\alpha$ on the space
$\Xi_\bullet = \tot_{\oplus}^\bullet\big(\prescript{\cL}{}{\sT}_{\poly}^{\bullet}(\cM)\big)$
of $\cL$-longitudinal polyvector fields by contraction.
Then the pair of maps $(e^{\alpha^\sharp}, e^{-\alpha^\flat})$
constitutes an automorphism of the dg calculus $(\Theta^\bullet,\Xi_\bullet)$.
\end{lemma}

\begin{proof}
It is simple to show that $\alpha^\sharp$ is a derivation of the dg Gerstenhaber algebra $\Theta^\bullet$.
It thus follows that $e^{\alpha^\sharp}$ is an automorphism of $\Theta^\bullet$.
It remains to show that, for all $\gamma\in\Theta^\bullet$ and $\omega\in\Xi_\bullet$,
\begin{gather}
e^{-\alpha^\flat} (\iI_\gamma\omega) = \iI_{e^{\alpha^\sharp}(\gamma)} \big(e^{-\alpha^\flat}(\omega)\big), \label{eq:cherry-InteriorProd} \\
e^{-\alpha^\flat} (\iL_\gamma\omega) = \iL_{e^{\alpha^\sharp}(\gamma)} \big(e^{-\alpha^\flat}(\omega)\big), \label{eq:cherry-LieDer} \\
e^{-\alpha^\flat} \big(d_{\cL}(\omega)\big) = d_{\cL} \big(e^{-\alpha^\flat}(\omega)\big), \label{eq:cherry-CE} \\
\intertext{and}
e^{-\alpha^\flat} \big(\cQ(\omega)\big) = \cQ \big(e^{-\alpha^\flat}(\omega)\big). \label{eq:cherry-DG}
\end{gather}
In order to prove Equation~\eqref{eq:cherry-InteriorProd}, we start with the observation that
\begin{equation}\label{ARIN}
\iI_\gamma\circ\alpha^\flat-\alpha^\flat\circ\iI_\gamma=\iI_{\alpha^\sharp(\gamma)}
.\end{equation}
Indeed, for all $X_1,X_2,\cdots,X_p\in\Gamma(\cL[-1])$ and $\omega\in\Gamma\big(\hat{S}(\cL^\vee[1])\big)$, we have
\begin{multline*}
\iI_{X_1\cdots X_p}\circ\alpha^\flat(\omega)=\iI_{X_1\cdots X_p}(\alpha\cdot\omega)
=\sum_{k=1}^p \pm \iI_{X_k}(\alpha) \cdot \iI_{X_1\cdots\widehat{X_k}\cdots X_p}(\omega)
+\alpha\cdot \iI_{X_1\cdots X_p}(\omega)
\\
=\iI_{\sum_{k=1}^p \pm \iI_{X_k}(\alpha) \cdot X_1\cdots\widehat{X_k}\cdots X_p}(\omega)
+\alpha^\flat\big(\iI_{X_1\cdots X_p}(\omega)\big)
=\iI_{\alpha^\sharp(X_1\cdots X_p)}\omega +\alpha^\flat\circ \iI_{X_1\cdots X_p}(\omega)
.\end{multline*}
It follows from Equation~\eqref{ARIN} that
\[ \frac{d}{dt}\big(e^{t\alpha^\flat}\circ\iI_{e^{t\alpha^\sharp}(\gamma)}\circ e^{-t\alpha^\flat}\big)
= e^{t\alpha^\flat}\circ
(\alpha^\flat\circ\iI_{e^{t\alpha^\sharp}(\gamma)}
+\iI_{\alpha^\sharp\circ e^{t\alpha^\sharp}(\gamma)}
-\iI_{e^{t\alpha^\sharp}(\gamma)}\circ\alpha^\flat)
\circ e^{-t\alpha^\flat} = 0 .\]
Therefore, $e^{t\alpha^\flat}\circ\iI_{e^{t\alpha^\sharp}(\gamma)}\circ e^{-t\alpha^\flat}$ is independent of $t$
and we can conclude that
\[ \iI_{e^{\alpha^\sharp}(\gamma)}\circ e^{-\alpha^\flat}=e^{-\alpha^\flat}\circ\iI_{\gamma} .\]
Equations~\eqref{eq:cherry-CE} and~\eqref{eq:cherry-DG} can be verified directly.
Finally, Equation~\eqref{eq:cherry-LieDer} follows immediately from Cartan's formula.
\end{proof}

Applying Lemma~\ref{cherry} to the dg Lie algebroid $\cFQ\to\cNQ$ with $\alpha =\frac{1}{2}\supertrace\big(\atiyahcocycle{\can}{\cF_{\cQ}}\big)$, we have the following:

\begin{corollary}\label{cor:AroofToTodd}
Let $\cFQ\to\cNQ$ be a Fedosov dg manifold associated with a dg manifold $(\cM,Q)$, and $\alpha =\frac{1}{2}\supertrace\big(\atiyahcocycle{\can}{\cF_{\cQ}}\big)$.
Then the pair of maps $(e^{\alpha^\sharp},e^{-\alpha^\flat})$ constitutes an automorphism of the dg calculus $\big(\big( \totTpolyF{\bullet}[1] \big)_{\tauQ},\big( \totbApolyF{\bullet}\big)_{\tauQ}\big).$
\end{corollary}
\begin{proof}
From Lemma~\ref{lem:ScalarAtiyah-Div}, it follows that the degree-zero $\cF$-longitudinal $1$-form
$\alpha:=\frac{1}{2}\supertrace\big(\atiyahcocycle{\can}{\cF_{\cQ}}\big)$
satisfies $d_{\cF}(\alpha)=0$.
Since the Lie derivative $\iL_{\fedosov+\tauTpolyF(Q)}$ commutes with the supertrace,
we also have $\iL_{\fedosov+\tauTpolyF(Q)}(\alpha)=0$.
Therefore, applying Lemma~\ref{cherry} to the Fedosov dg Lie algebroid $\cF_{\cQ}\to\cN_{\cQ}$
and the section $\alpha\in\big(\Gamma(\cF^\vee[1])\big)^0$,
we can conclude that the pair of maps $(e^{\alpha^\sharp},e^{-\alpha^\flat})$ constitutes
an automorphism of the Cartan dg calculus of $\cF$-longitudinal polyvector fields and differential forms.
\end{proof}

Now we are ready to prove the three main propositions of this section.

\begin{proof}[Proof of Proposition~\ref{thm:KontsevichFormalityFedosov}]
Since the map $e^{\alpha^\sharp}$ is an automorphism of the dgla $\big(\totTpolyF{\bullet}[1]\big)_{\tauQ}$ by Corollary~\ref{cor:AroofToTodd}, the composition
\[ \ukontsevich := \tkontsevich \circ e^{\alpha^\sharp}: \big(\totTpolyF{\bullet}[1]\big)_{\tauQ} \inftyto \big(\totDpolyF{\bullet}[1]\big)_{\tauQ} \]
defines an $L_\infty$ morphism. Furthermore, by Equation~\eqref{Fukuoka}, its first Taylor coefficient satisfies $\ukontsevich_1=\tkontsevich_1\circ e^{\alpha^\sharp}=\hkr\circ\big(\todd^{\can}_{\cF_{\cQ}}\big)^{\frac{1}{2}}$, where $\tkontsevich$ is the $L_\infty$ morphism \eqref{eq:tkonsevich}. This completes the proof.
\end{proof}

\begin{proof}[Proof of Proposition~\ref{thm:TsyganFormalityFedosov}]
To prove the proposition, first observe that the map
\[ e^{\alpha^\flat}: (e^{\alpha^\sharp})^\ast\big( \totbApolyF{\bullet} \big)_{\tauQ}
\to \big( \totbApolyF{\bullet}\big)_{\tauQ} \]
is a strict $L_\infty$ module morphism over $\big(\totTpolyF{\bullet}[1]\big)_{\tauQ}$ by Equation~\eqref{eq:cherry-LieDer}.
In addition, pulling back the $L_\infty$ module morphism \eqref{eq:tshoikhet} via the dgla automorphism $e^{\alpha^\sharp}$ yields the $L_\infty$ module morphism
\[ (e^{\alpha^\sharp})^\ast \tshoikhet: \ukontsevich^\ast\big(\totcCpolyF{\bullet}\big)_{\tauQ} \inftyto (e^{\alpha^\sharp})^\ast\big(\totbApolyF{\bullet}\big)_{\tauQ} .\]
Let $\ushoikhet$ be the composition of these $L_\infty$ module morphisms:
\[ \ushoikhet:= e^{\alpha^\flat} \circ \big((e^{\alpha^\sharp})^\ast \tshoikhet\big):\ukontsevich^\ast\big(\totcCpolyF{\bullet}\big)_{\tauQ} \inftyto \big( \totbApolyF{\bullet} \big)_{\tauQ} .\]
Since $\big((e^{\alpha^\sharp})^\ast\tshoikhet\big)_0=\tshoikhet_0$ (see Appendix~\ref{sec:PullBack}),
it follows from Equation~\eqref{Fukushima} that
$\ushoikhet_0=\big(\todd^{\can}_{\cF_{\cQ}}\big)^{\frac{1}{2}}\circ\hhkr$.
This completes the proof.
\end{proof}

\begin{proof}[Proof of Proposition~\ref{thm:DK_Fedosov}]
By Corollary~\ref{cor:AroofToTodd}, the pair of maps $(e^{\alpha^\sharp}, e^{-\alpha^\flat})$ defines an automorphism of the calculi at the level of cohomology. The conclusion thus follows directly from Corollary~\ref{cor:AroofToTodd} and Proposition~\ref{prop:CompatibleWithCal_Fedosov}.
\end{proof}

\section{Formality theorems for dg manifolds}\label{sec:FormalityDGmfd}

This section is devoted to the proof of the main results in this paper: Theorem~\ref{thm:main} and Theorem~\ref{thm:formalitydg}.

\subsection{Important lemmas}
In this section, we establish a series of technical lemmas required for the proofs of our main theorems. We begin by introducing an $\cF$-connection $\nablapushed$ on $\cF$ induced by an affine connection $\nabla$ on $\cM$. Recall from~\cite{paper-1B} that the space of sections $\sections{\cN;\cF}$ is generated over $C^\infty(\cN)$ by $\tauTpolyF\big(\XX(\cM)\big)$. Consequently, we can uniquely specify $\nablapushed$ via the relation
\begin{equation}\label{eq:nablapushed}
\nablapushed_{\tauTpolyQ(X)} \tauTpolyQ(Y) = \tauTpolyQ(\nabla_X Y) \qquad \text{for all } X, Y \in \XX(\cM).
\end{equation}

We have the following

\begin{lemma}\label{maple}
The following identities hold:
\begin{enumerate}
\item $\tauTpolyQ(\atiyahcocycleQ) = \atiyahcocycle{\nablapushed}{\cFQ}$;
\item $\tauTpolyQ(\toddcocycleQ) = \toddcocycle{\nablapushed}{\cFQ}$.
\end{enumerate}
\end{lemma}

\begin{proof}
For the first assertion, since $\sections{\cN;\cF}$ is $C^\infty(\cN)$-spanned by the image $\tauTpolyF\big(\XX(\cM)\big)$, it suffices to show that
\[
\tauTpolyQ(\atiyahcocycleQ)(\tauTpolyQ(X),\tauTpolyQ(Y)) = \atiyahcocycle{\nablapushed}{\cFQ}(\tauTpolyQ(X),\tauTpolyQ(Y))
\]
for all $X,Y \in \XX(\cM)$. Since $\tauTpolyF(\Phi)(\tauTpolyF(X),\tauTpolyF(Y)) = \tauTpolyF(\Phi(X,Y))$ for any $\Phi \in \sections{T_{\cM}\dual \otimes T_{\cM}\dual \otimes T_{\cM}}$ and $X,Y \in \XX(\cM)$ (see~\cite{paper-1B}), Equation~\eqref{eq:nablapushed} implies that
\begin{align*}
\tauTpolyQ(\atiyahcocycleQ)&(\tauTpolyQ(X),\tauTpolyQ(Y)) = \tauTpolyQ(\iL_Q \nabla_X Y -\nabla_{\iL_Q X} Y -(-1)^{|X|} \nabla_X (\iL_Q Y)) \\
&= \iL_{\tauTpolyQ(Q)} (\nablapushed_{\tauTpolyQ(X)} \tauTpolyQ(Y)) - \nablapushed_{ (\iL_{\tauTpolyQ(Q)} \tauTpolyQ(X))} \tauTpolyQ(Y) -(-1)^{|X|} \nablapushed_{\tauTpolyQ(X)} ( \iL_{\tauTpolyQ(Q)} \tauTpolyQ(Y)).
\end{align*}
On the other hand, since $L_{\fedosov} \circ \tauTpolyQ = 0$, we have $\iL_{\fedosov } (\nablapushed_{\tauTpolyQ(X)} \tauTpolyQ(Y)) = \iL_{\fedosov } \tauTpolyQ(\nablapushed_{X} Y) = 0$, as well as $\iL_{\fedosov } \tauTpolyQ(X) = 0$ and $\iL_{\fedosov } \tauTpolyQ(Y) = 0$. Consequently,
\begin{align*}
\atiyahcocycle{\nablapushed}{\cFQ}(\tauTpolyQ(X),\tauTpolyQ(Y))
&= \iL_{\fedosov +\tauTpolyQ(Q)} (\nablapushed_{\tauTpolyQ(X)} \tauTpolyQ(Y)) \\
&\qquad - \nablapushed_{ (\iL_{\fedosov +\tauTpolyQ(Q)} \tauTpolyQ(X))} \tauTpolyQ(Y) -(-1)^{|X|} \nablapushed_{\tauTpolyQ(X)} ( \iL_{\fedosov +\tauTpolyQ(Q)} \tauTpolyQ(Y)) \\
&= \iL_{\tauTpolyQ(Q)} (\nablapushed_{\tauTpolyQ(X)} \tauTpolyQ(Y)) - \nablapushed_{ (\iL_{\tauTpolyQ(Q)} \tauTpolyQ(X))} \tauTpolyQ(Y) -(-1)^{|X|} \nablapushed_{\tauTpolyQ(X)} ( \iL_{\tauTpolyQ(Q)} \tauTpolyQ(Y)) \\
&= \tauTpolyQ(\atiyahcocycleQ)(\tauTpolyQ(X),\tauTpolyQ(Y)).
\end{align*}

For the second assertion, since Todd cocycles can be expressed in terms of scalar Atiyah cocycles, it suffices to establish that
\begin{equation}\label{eq:maple-1}
\tauTpolyQ\big(\supertrace\big((\atiyahcocycleQ)^k\big)\big) = \supertrace\big((\atiyahcocycle{\nablapushed}{\cFQ})^k\big)
\end{equation}
for each $k \in \NN$. Equation~\eqref{eq:maple-1} follows directly from the first assertion along with the structural properties of $\tauTpolyQ$ detailed in~\cite{paper-1B}, via the chain of equalities
\[
\tauTpolyQ\big(\supertrace\big((\atiyahcocycleQ)^k\big)\big) = \supertrace\big(\tauTpolyQ\big((\atiyahcocycleQ)^k\big)\big) = \supertrace\big(\big(\tauTpolyQ(\atiyahcocycleQ)\big)^k\big) = \supertrace\big((\atiyahcocycle{\nablapushed}{\cFQ})^k\big).
\]
This completes the proof.
\end{proof}

\begin{lemma}\label{lem:Tolbiac}
The two cochain maps
\begin{enumerate}
\item \label{lem:Tolbiac-(1)}
$\big(\todd^{\can}_{\cF_Q}\big)^{-\frac{1}{2}} \circ \big(\todd^{\breve\nabla}_{\cF_Q} \big)^{\frac{1}{2}}: \big( \totTpolyF{\bullet}[1] \big)_{\tauQ} \to \big( \totTpolyF{\bullet}[1] \big)_{\tauQ}$,
\item \label{lem:Tolbiac-(2)}
$\big(\todd^{\breve\nabla}_{\cF_Q}\big)^{\frac{1}{2}} \circ \big(\todd^{\can}_{\cF_Q} \big)^{-\frac{1}{2}}: \Big(\totbApolyF{\bullet}, \iL_{\fedosov + \tauTpolyF(Q)}\Big) \to \Big(\totbApolyF{\bullet}, \iL_{\fedosov + \tauTpolyF(Q)}\Big)$,
\end{enumerate}
(where the Todd cocycles act by contraction in~\ref{lem:Tolbiac-(1)}
and by multiplication in~\ref{lem:Tolbiac-(2)})
are both homotopic to the identity.
\end{lemma}
\begin{proof}
Let $\tauQ = \fedosov + \tauTpolyF(Q)$ as before. Since the Todd class of a dg vector bundle is independent of the choice of connection, it follows that $(\todd^{\nablapushed}_{\cF_Q})^{\frac{1}{2}} = (\todd^{\can}_{\cF_Q})^{\frac{1}{2}} + \iL_{\tauQ} \zeta$ for some $\zeta \in \totbApolyF{-1}$. Because $\iL_{\tauQ}(\todd^{\can}_{\cF_Q}) = 0$, we have
\begin{gather*}
(\todd^{\can}_{\cF_Q})^{-\frac{1}{2}} \wedge (\todd^{\nablapushed}_{\cF_Q})^{\frac{1}{2}} = 1 + (\todd^{\can}_{\cF_Q})^{-\frac{1}{2}} \wedge \iL_{\tauQ} \zeta = 1 + \iL_{\tauQ} \xi
\\ \intertext{and}
(\todd^{\nablapushed}_{\cF_Q})^{\frac{1}{2}} \wedge (\todd^{\can}_{\cF_Q})^{-\frac{1}{2}} = 1 + (\iL_{\tauQ} \zeta) \wedge (\todd^{\can}_{\cF_Q})^{-\frac{1}{2}} = 1 + \iL_{\tauQ} \eta
,\end{gather*}
where $\xi = (\todd^{\can}_{\cF_Q})^{-\frac{1}{2}} \wedge \zeta$ and $\eta = \zeta \wedge (\todd^{\can}_{\cF_Q})^{-\frac{1}{2}}$ are elements in $\totbApolyF{-1}$. Consequently,
\[ \iI_{(\todd^{\can}_{\cF_Q})^{-\sfrac{1}{2}} \wedge (\todd^{\breve\nabla}_{\cF_Q})^{\sfrac{1}{2}}} = \id + \iI_{\iL_{\tauQ} \xi} = \id + \iL_{\tauQ} \circ \iI_\xi + \iI_\xi \circ \iL_{\tauQ}, \]
which demonstrates that the cochain map in Lemma~\ref{lem:Tolbiac}~\ref{lem:Tolbiac-(1)} is homotopic to the identity via the homotopy operator $\iI_\xi$.
Similarly, the cochain map in Lemma~\ref{lem:Tolbiac}~\ref{lem:Tolbiac-(2)} is homotopic to the identity with homotopy operator $\eta \wedge \argument$, since
\[ \big(\todd^{\breve\nabla}_{\cF_Q}\big)^{\frac{1}{2}} \circ \big(\todd^{\can}_{\cF_Q} \big)^{-\frac{1}{2}} = \id + (\iL_{\tauQ} \eta) \wedge \argument = \id + \iL_{\tauQ} \circ (\eta \wedge \argument) + (\eta \wedge \argument) \circ \iL_{\tauQ}. \]
This completes the proof.
\end{proof}

\begin{lemma}\label{lem:Tau-Todd}
The following diagrams
\[ \begin{tikzcd}[column sep=large]
\totTpolyF{\bullet} \ar[r,"(\todd^{\breve\nabla}_{\cF_Q})^{\frac{1}{2}}"] & \totTpolyF{\bullet} \\
\totTpolyM{\bullet} \ar[r,"(\toddcocycleQ)^{\frac{1}{2}}"] \ar[u,"\tauDpolyF"] & \totTpolyM{\bullet} \ar[u,"\tauTpolyF"']
\end{tikzcd} \]
and
\[ \begin{tikzcd}[column sep=large]
\totbApolyF{\bullet} \ar[r,"(\todd^{\breve\nabla}_{\cF_Q})^{\frac{1}{2}}"] & \totbApolyF{\bullet} \\
\totApolyM{\bullet} \ar[r,"(\toddcocycleQ)^{\frac{1}{2}}"] \ar[u,"\tauDpolyF"] & \totApolyM{\bullet} \ar[u,"\tauTpolyF"']
\end{tikzcd} \]
commute.
\end{lemma}
\begin{proof}
By Lemma~\ref{maple}, for any $\xi \in \totTpolyM{\bullet}$, we have
\[ \tauTpolyQ\big(\iI_{(\toddcocycleQ)^{\sfrac{1}{2}}} \xi\big) = \iI_{\tauTpolyQ((\toddcocycleQ)^{\sfrac{1}{2}})} \tauTpolyQ(\xi) = \iI_{(\todd^{\breve\nabla}_{\cF_Q})^{\sfrac{1}{2}}} \tauTpolyQ(\xi) .\]
Similarly, for any $\omega \in \totApolyM{\bullet}$, we have
\[ \tauTpolyF\big((\toddcocycleQ)^{\sfrac{1}{2}} \wedge \omega \big) = \tauTpolyF\big((\toddcocycleQ)^{\sfrac{1}{2}}\big) \wedge \tauTpolyF(\omega) = (\todd^{\breve\nabla}_{\cF_Q})^{\sfrac{1}{2}} \wedge \tauTpolyF(\omega). \]
This completes the proof.
\end{proof}

\subsection{Proof of main theorems}

We are now ready to prove the main theorems of the paper.

The proof of Theorem~\ref{thm:formalitydg}~\ref{thm:formalitydg-K} was sketched in~\cite{MR3754617}.
For completeness, we provide a full proof below.

\begin{proof}[Proof of Theorem~\ref{thm:formalitydg}~\ref{thm:formalitydg-K}]
By
\begin{equation}\label{eq:TauAsLooMor-Tpoly}
\Tau_T: \big( \totTpolyM{\bullet}[1] \big)_Q \inftyto \big( \totTpolyF{\bullet}[1] \big)_{\tauQ},
\end{equation}
we denote the $L_\infty$ quasi-isomorphism having
$\tauTpolyQ$ as the first Taylor coefficient and
all higher Taylor coefficients equal to zero.
Next, since Lemma~\ref{lem:Tolbiac} implies that the cochain map
$\big(\todd^{\can}_{\cF_Q}\big)^{-\frac{1}{2}} \circ \big(\todd^{\breve\nabla}_{\cF_Q} \big)^{\frac{1}{2}} : \big( \totTpolyF{\bullet}[1] \big)_{\tauQ} \to \big( \totTpolyF{\bullet}[1] \big)_{\tauQ}$
is homotopic to the identity map, Lemma~\ref{lem:LiftedLooMor-htp} ensures the existence of an $L_\infty$ quasi-isomorphism
\begin{equation}\label{eq:Loo-ChangeTodd}
\Xi^{\breve\nabla} : \big( \totTpolyF{\bullet}[1] \big)_{\tauQ} \inftyto \big( \totTpolyF{\bullet}[1] \big)_{\tauQ},
\end{equation}
homotopic to the identity, such that $\Xi^{\breve\nabla}_1 = \big(\todd^{\can}_{\cF_Q}\big)^{-\frac{1}{2}} \circ \big(\todd^{\breve\nabla}_{\cF_Q} \big)^{\frac{1}{2}}$.
Recall that according to Proposition~\ref{thm:KontsevichFormalityFedosov},
we have a Kontsevich formality morphism for the Fedosov dg Lie algebroid:
\[ \ukontsevich:
\big( \totTpolyF{\bullet}[1] \big)_{\tauQ}
\inftymorphism \big( \totDpolyF{\bullet} [1] \big)_{\tauQ} \]
Finally, note that
by applying Proposition~\ref{prop:LooTransfer-Inj-Surj}~\ref{prop:LooTransfer-Inj} to the contraction \eqref{eq:Contraction-Dpoly},
there exists an $L_\infty$ quasi-isomorphism
\[ \Sigma_D: \big( \totDpolyF{\bullet} [1] \big)_{\tauQ}
\inftyto \big( \totDpolyM{\bullet} [1] \big)_Q \]
such that $\Sigma_{D,1} = \sigmaDpolyQ$.
Consider the $L_\infty$ morphism
\[ \dgkont : \big( \totTpolyM{\bullet}[1] \big)_Q \inftymorphism \big( \totDpolyM{\bullet} [1] \big)_Q \]
defined as the composition:
\begin{equation}\label{eq:dgkont-construction}
\dgkont = \Sigma_D \circ \ukontsevich \circ \Xi^{\breve\nabla} \circ \Tau_T
.\end{equation}
It remains to prove that $\dgkont_1=\hkr\circ(\toddcocycleQ)^{\frac{1}{2}}$. We have
\begin{align*}
\dgkont_1 & = \Sigma_{D,1} \circ \ukontsevich_1 \circ \Xi^{\breve\nabla}_1 \circ \Tau_{T,1} && \\
& = \sigmaDpolyQ \circ \hkr \circ \big(\todd^{\can}_{\cF_Q}\big)^{\frac{1}{2}} \circ
\big(\todd^{\can}_{\cF_Q}\big)^{-\frac{1}{2}} \circ \big(\todd^{\breve\nabla}_{\cF_Q} \big)^{\frac{1}{2}}
\circ \tauTpolyQ && \text{by Proposition~\ref{thm:KontsevichFormalityFedosov}} \\
& = \sigmaDpolyQ \circ \hkr \circ \tauTpolyQ \circ (\toddcocycleQ)^{\frac{1}{2}}
&& \text{by Lemma~\ref{lem:Tau-Todd}} \\
& = \sigmaDpolyQ \circ \tauTpolyQ \circ \hkr \circ (\toddcocycleQ)^{\frac{1}{2}}
&& \text{by Lemma~\ref{lem:Tau-HKR}} \\
& = \hkr \circ (\toddcocycleQ)^{\frac{1}{2}}. &&
\end{align*}
This completes the proof.
\end{proof}

Similarly to Equation~\eqref{eq:Loo-ChangeTodd}, in order to prove Theorem~\ref{thm:formalitydg}~\ref{thm:formalitydg-S}, we first modify the $L_\infty$ module formality morphism $\ushoikhet$ introduced in Proposition~\ref{thm:TsyganFormalityFedosov}.

\begin{lemma}\label{lem:TsyganFormality}
There exists an $L_\infty$ module quasi-isomorphism
\begin{equation}
\tshoikhet^{\breve\nabla}:
(\ukontsevich \circ \Xi^{\breve\nabla})^* \big( \totcCpolyF{\bullet} \big)_{\tauQ} \inftyto \big( \totbApolyF{\bullet} \big)_{\tauQ}
\end{equation}
over the dgla $\big( \totTpolyF{\bullet}[1] \big)_{\tauQ}$ whose zeroth Taylor coefficient is given by $\tshoikhet^{\breve\nabla}_0 = (\todd^{\breve\nabla}_{\cF_Q})^{\frac{1}{2}} \circ \hhkr$.
\end{lemma}

\begin{proof}
Note that, by Lemma~\ref{lem:Tolbiac}~\ref{lem:Tolbiac-(2)}, the cochain map
\[ \big(\todd^{\breve\nabla}_{\cF_Q}\big)^{\frac{1}{2}} \circ \big(\todd^{\can}_{\cF_Q} \big)^{-\frac{1}{2}}: \Big(\totbApolyF{\bullet}, \iL_{\fedosov + \tauTpolyF(Q)}\Big) \to \Big(\totbApolyF{\bullet}, \iL_{\fedosov + \tauTpolyF(Q)}\Big) \]
is homotopic to the identity.
Since the $L_\infty$ morphism $\Xi^{\breve\nabla}$ in~\eqref{eq:Loo-ChangeTodd}
is homotopic to the identity, it follows from Lemma~\ref{lem:Pyramides}
that there exists an $L_\infty$ module morphism over the dgla $\big( \totTpolyM{\bullet}[1] \big)_Q$,
\begin{equation}\label{eq:LooMod-ChangeTodd}
\mho^{\breve\nabla}: (\ukontsevich\circ \Xi^{\breve\nabla})^*
\big(\totbApolyF{\bullet} \big)_{\tauQ} \inftyto
\ukontsevich^* \big(\totbApolyF{\bullet}\big)_{\tauQ},
\end{equation}
such that $\mho^{\breve\nabla}_0 = \big(\todd^{\breve\nabla}_{\cF_Q}\big)^{\frac{1}{2}} \circ \big(\todd^{\can}_{\cF_Q} \big)^{-\frac{1}{2}}$.
Consequently, the composition $\tshoikhet^{\breve\nabla} = \mho^{\breve\nabla} \circ \ushoikhet$ gives the desired quasi-isomorphism,
where $\ushoikhet$ is the Shoikhet formality morphism for the Fedosov dg Lie algebroid, as in Proposition~\ref{thm:TsyganFormalityFedosov}.
\end{proof}

\begin{proof}[Proof of Theorem~\ref{thm:formalitydg}~\ref{thm:formalitydg-S}]
Let $\dgkont : \big( \totTpolyM{\bullet}[1] \big)_Q \inftymorphism \big( \totDpolyM{\bullet} [1] \big)_Q$ be the formality morphism defined by Equation~\eqref{eq:dgkont-construction}. Denote by
\begin{equation}\label{eq:TsyganFormality-pf-1}
\Tau_C: \dgkont^* \big(\totCpolyM{\bullet}\big)_Q \inftyto
(\tauDpolyF \circ \dgkont)^* \big(\totCpolyF{\bullet}\big)_{\tauQ}
\end{equation}
the $L_\infty$ module quasi-isomorphism over the dgla $\big( \totTpolyM{\bullet}[1] \big)_Q$
having $\tauTpolyQ$ as the zeroth Taylor coefficient and all higher Taylor coefficients equal to zero.

Next, consider the following two $L_\infty$ morphisms from $\big(\totTpolyM{\bullet}[1]\big)_Q$ to $\big(\totDpolyF{\bullet}[1]\big)_{\tauQ}$:
\[ \Tau_D \circ \dgkont \qquad \text{and} \qquad \ukontsevich \circ \Xi^{\breve\nabla} \circ \Tau_T ,\]
where $\Tau_D: \big(\totDpolyM{\bullet}[1]\big)_Q \inftyto \big(\totDpolyF{\bullet}[1]\big)_{\tauQ}$ is the $L_\infty$ morphism whose first Taylor coefficient is $\tauDpolyQ$ and whose higher Taylor coefficients all vanish. By Equation~\eqref{eq:dgkont-construction}, we have the relation $\Tau_D \circ \dgkont = \Tau_D \circ \Sigma_D \circ \ukontsevich \circ \Xi^{\breve\nabla} \circ \Tau_T$. By Lemma~\ref{lem:HtpMor-HtpTransfer}, the composition $\Tau_D \circ \Sigma_D$ is homotopic to the identity morphism. It follows that $\Tau_D \circ \dgkont$ and $\ukontsevich \circ \Xi^{\breve\nabla} \circ \Tau_T$ are homotopic as $L_\infty$ morphisms. Consequently, Lemma~\ref{lem:PullbackViaHtpMor} implies the existence of an $L_\infty$ module isomorphism
\begin{equation}\label{eq:TsyganFormality-pf-2}
\mathfrak{I}: (\tauDpolyQ \circ \dgkont)^* \big( \totcCpolyF{\bullet}\big)_{\tauQ} \inftyto (\ukontsevich \circ \Xi^{\breve\nabla} \circ \Tau_T)^* \big( \totcCpolyF{\bullet} \big)_{\tauQ},
\end{equation}
whose zeroth Taylor coefficient is the identity map, i.e.\ $\mathfrak{I}_0 = \id$.

Note that the $L_\infty$ module morphism $\tshoikhet^{\breve\nabla}$ in Lemma~\ref{lem:TsyganFormality} is over $\big(\totTpolyF{\bullet}[1]\big)_{\tauQ}$.
By pulling back $\tshoikhet^{\breve\nabla}$ along the $L_\infty$ quasi-morphism $\Tau_T$ as in~\eqref{eq:TauAsLooMor-Tpoly}, we obtain an $L_\infty$ module morphism
\begin{equation}\label{eq:TsyganFormality-pf-3}
(\Tau_T)^{\ast}\tshoikhet^{\breve\nabla}: (\ukontsevich \circ \Xi^{\breve\nabla} \circ \Tau_T)^*\big( \totcCpolyF{\bullet} \big)_{\tauQ} \inftyto (\Tau_T)^{\ast}\big( \totbApolyF{\bullet} \big)_{\tauQ}
\end{equation}
over $\big( \totTpolyM{\bullet}[1] \big)_Q$.

Lastly, since the injection $\tauTpolyQ$ in the contraction \eqref{eq:Contraction-Apoly} defines a strict $L_\infty$ module morphism from $\big(\totApolyM{\bullet}\big)_{Q}$ to $(\Tau_T)^\ast\big(\totApolyF{\bullet}\big)_{\tauQ}$ over the dgla $\big( \totTpolyM{\bullet}[1] \big)_Q$ according to Theorem~\ref{thm:mainT}~\ref{thm:mainT-TA}, it follows from Proposition~\ref{prop:LooModTransfer-Inj-Surj}~\ref{prop:LooModTransfer-Inj} that there exists an $L_\infty$ module quasi-isomorphism over the dgla $\big( \totTpolyM{\bullet}[1] \big)_Q$:
\begin{equation}\label{eq:TsyganFormality-pf-4}
\Sigma_A: (\Tau_T)^{\ast} \big(\totbApolyF{\bullet}\big)_{\tauQ} \inftyto \big(\totApolyM{\bullet}\big)_Q
\end{equation}
such that its zeroth Taylor coefficient is given by $\Sigma_{A,0} = \sigmaTpolyQ$, where $\sigmaTpolyQ$ is the surjection in the contraction \eqref{eq:Contraction-Apoly}.

Consider the composition of $L_\infty$ module morphisms \eqref{eq:TsyganFormality-pf-1}-\eqref{eq:TsyganFormality-pf-4}:
\[ \dgshoi := \Sigma_A \circ ((\Tau_T)^\ast \tshoikhet^{\breve\nabla} )\circ \mathfrak{I} \circ \Tau_C:
\dgkont^* \big( \totCpolyM{\bullet} \big)_Q \inftymorphism \big( \totApolyM{\bullet} \big)_Q. \]
It remains to show that $\dgshoi_0 = (\toddcocycleQ)^{\frac{1}{2}} \circ \hhkr$. We have
\begin{align*}
\dgshoi_0 & = \Sigma_{A,0} \circ \tshoikhet^{\breve\nabla}_0 \circ \mathfrak{I}_0 \circ \Tau_{C,0} && \\
& = \sigmaTpolyQ \circ (\todd^{\breve\nabla}_{\cF_Q})^{\frac{1}{2}} \circ \hhkr \circ \tauTpolyQ
&& \text{by Lemma~\ref{lem:TsyganFormality}} \\
& = \sigmaTpolyQ \circ (\todd^{\breve\nabla}_{\cF_Q})^{\frac{1}{2}} \circ \tauTpolyQ \circ \hhkr
&& \text{by Lemma~\ref{lem:Tau-HKR}} \\
& = \sigmaTpolyQ \circ \tauTpolyQ \circ (\toddcocycleQ)^{\frac{1}{2}} \circ \hhkr
&& \text{by Lemma~\ref{lem:Tau-Todd}} \\
& = (\toddcocycleQ)^{\frac{1}{2}} \circ \hhkr. &&
\end{align*}
This completes the proof of Theorem~\ref{thm:formalitydg}.
\end{proof}

\begin{remark}[\cite{MR3754617}]
Given a pair of torsion-free affine connections $\nabla$ and $\nabla'$ on $(\cM,Q)$ with corresponding Todd cocycles $\toddcocycle{\nabla}{T_{(\cM,Q)}}$ and $\toddcocycle{\nabla'}{T_{(\cM,Q)}}$, there exist an $L_\infty$ algebra (auto)morphism $\Xi^{\nabla,\nabla'}$ of $\big(\totTpolyM{\bullet}[1]\big)_Q$
with the cochain complex isomorphism
$\big(\toddcocycle{\nabla}{T_{(\cM,Q)}}\big)^{-\shalf}\circ\big(\toddcocycle{\nabla'}{T_{(\cM,Q)}}\big)^{\shalf}$
as first Taylor coefficient $\Xi^{\nabla,\nabla'}_1$
and an $L_\infty$ module isomorphism
\begin{equation}
\mho^{\nabla,\nabla'}: \big(\Xi^{\nabla,\nabla'}\big)^*
\big(\totApolyM{\bullet}\big)_Q \inftyto
\big(\totApolyM{\bullet}\big)_Q
,\end{equation}
with the cochain complex isomorphism
$\big(\toddcocycle{\nabla}{T_{(\cM,Q)}}\big)^{-\shalf}\circ\big(\toddcocycle{\nabla'}{T_{(\cM,Q)}}\big)^{\shalf}$
as zeroth Taylor coefficient $\mho^{\nabla,\nabla'}_0$.
\end{remark}

Finally, we are ready to prove Theorem~\ref{thm:main}.

\begin{proof}[Proof of Theorem~\ref{thm:main}]
By the construction of $\dgkont$ and $\dgshoi$, we have the following commutative diagrams:
\[ \begin{tikzcd}
\cohomology{}\big(\totTpolyF{\bullet},\iL_{\tauTpolyF(Q) + \fedosov} \big) \ar[r,"\ukontsevich_1"] & \cohomology{}\big(\totDpolyF{\bullet},\iL_{\tauTpolyF(Q) + \fedosov} + \hochschild\big) \\
\HTT \ar[r,"\dgkont_{1}"] \ar[u,"\tauTpolyQ"] & \HDD \ar[u,"\tauDpolyQ"']
\end{tikzcd} \]
and
\[ \begin{tikzcd}
\cohomology{}\big(\totbApolyF{\bullet},\iL_{\tauTpolyF(Q) + \fedosov} \big)
& \ar[l,"\ushoikhet_0"'] \cohomology{}\big(\totcCpolyF{\bullet},\iL_{\tauTpolyF(Q) + \fedosov} + \hochschildb\big) \\
\HOM \ar[u,"\tauTpolyQ"] & \ar[l,"\dgshoi_{0}"'] \HCC . \ar[u,"\tauDpolyQ"']
\end{tikzcd} \]
On the cohomology level, the pair $(\ukontsevich_1, \ushoikhet_0^{-1}) = \big( \hkr \circ (\toddclass{\cF_Q})^{\frac{1}{2}}, \, \big( (\toddclass{\cF_Q})^{\frac{1}{2}} \circ \hhkr \big)^{-1} \big)$ and the vertical maps are isomorphisms of calculi (by Proposition~\ref{thm:DK_Fedosov} and Theorem~\ref{thm:mainT}). It follows immediately that
\[
(\dgkont_1, \dgshoi_0^{-1}) = \big( \hkr \circ (\toddclassQ)^{\frac{1}{2}}, \, \big( (\toddclassQ)^{\frac{1}{2}} \circ \hhkr \big)^{-1} \big) : \calculus_C(\cM, Q) \longrightarrow \calculus_H(\cM, Q)
\]
is likewise an isomorphism of calculi. This completes the proof of Theorem~\ref{thm:main}.
\end{proof}

\appendix

\section{\texorpdfstring{$L_\infty$ algebras, $L_\infty$ modules, and homotopy transfer}{L-infinity algebras, L-infinity modules, and homotopy transfer}}

This appendix summarizes results on $L_\infty$ algebras, $L_\infty$ modules, and homotopy transfer
employed throughout this work.
Further details can be found in~\cite{MR4485797,MR4693987} and the references therein.

\subsection{\texorpdfstring{$L_\infty$ modules over an $L_\infty$ algebra}{L-infinity modules over an L-infinity algebra}}

Recall that an \textbf{$L_\infty$ algebra} is a graded vector space $\frakg = \bigoplus_{n \in \ZZ} \frakg^n$ equipped with a sequence $(\lambda_i)_{i\in\NN}$ of maps $\lambda_i: \Lambda^i \frakg \to \frakg$ of degree $2-i$ satisfying the relations
\begin{equation}\label{eq:LooAlg_def}
\sum_{\substack{p+q=k \\ p\geq 1 \\ q\geq 0}} \sum_{\sigma\in \shuf(p,q)} (\pm 1)\cdot
\lambda_{q+1}\big(\lambda_p(x_{\sigma(1)}\wedge \cdots \wedge x_{\sigma(p)})
\wedge x_{\sigma(p+1)} \wedge \cdots \wedge x_{\sigma(p+q)}\big) =0
\end{equation}
for all $k\geq 1$ and all homogeneous vectors $x_1,\cdots,x_k\in\frakg$.
The notation $\shuf(p,q)$ refers to the set of permutations $\sigma$ of the finite ordered set $\{1,2,\cdots,p+q\}$ such that $\sigma(1) < \sigma(2) < \cdots < \sigma(p)$ and $\sigma(p+1) < \sigma(p+2) < \cdots < \sigma(p+q)$.

If Equation~\eqref{eq:LooAlg_def} is rewritten in terms of the associated sequence
$(\tilde{\lambda}_i)_{i\in\NN}$ of maps of degree one
\[ \tilde{\lambda}_i:= \nshift \circ \lambda_i \circ (\pshift)^{\otimes i}
: S^i(\frakg[1]) \to \frakg[1] \]
where $(\pshift)^{\otimes i}: S^i(\frakg[1]) \to \Lambda^i \frakg$ is the natural décalage map,
it takes the form
\begin{equation}\label{eq:Loo[1]Alg_def}
\sum_{\substack{p+q=k \\ p\geq 1 \\ q\geq 0}} \sum_{\sigma\in \shuf(p,q)} \sgn(\sigma) \cdot
\tilde{\lambda}_{q+1}\big( \tilde{\lambda}_p(\nshift x_{\sigma(1)}\odot \cdots \odot \nshift x_{\sigma(p)}) \odot \nshift x_{\sigma(p+1)} \odot \cdots \odot \nshift x_{\sigma(p+q)}\big) =0,
\end{equation}
where the (now much simpler) sign $\sgn(\sigma)$ is determined by the Koszul sign convention.
Such a sequence $(\tilde{\lambda}_i)_{i=1}^\infty$ of maps is called an \textbf{$L_\infty[1]$ algebra} structure on $\frakg[1]$.

Denoting by $Q_{\tilde{\lambda}}$ the degree-one coderivation of the cofreely cogenerated coalgebra $S(\frakg[1])$ whose corestriction is the degree-one map $\tilde{\lambda} := \sum_{i=1}^\infty \tilde{\lambda}_i: S(\frakg[1]) \to \frakg[1]$, the relations~\eqref{eq:Loo[1]Alg_def} can be rewritten as a single equation:
\[ Q_{\tilde{\lambda}} \circ Q_{\tilde{\lambda}} = 0 .\]
Hence $\big(S(\frakg[1]),Q_{\tilde{\lambda}}\big)$ is a dg coalgebra.
For more explicit formulae, see~\cite{arXiv:1705.02880,MR4735657,MR4693987,MR4485797}.

A \textbf{morphism of $L_\infty$ algebras} from an $L_\infty$ algebra $(\frakg,\lambda)$ to an $L_\infty$ algebra $(\frakh,\mu)$ is a sequence $(\Xi_k)_{k\in\NO}$ of maps $\Xi_k: \Lambda^k \frakg \to \frakh$ of degree $1-k$ satisfying certain relations.
In terms of the morphism of cofreely cogenerated coalgebras
\[ F_{\tilde{\Xi}}:S(\frakg[1]) \to S(\frakh[1]) \]
whose corestriction is the sum
$\tilde{\Xi}:= \sum_{k=1}^\infty \tilde{\Xi}_k: S(\frakg[1]) \to \frakh[1]$
of the associated degree-zero maps
\[ \tilde{\Xi}_k := \nshift \circ \Xi_k \circ (\pshift)^{\otimes k}: S^k(\frakg[1]) \to \frakh[1] ,\]
these relations can be rewritten as a single equation:
\[ F_{\tilde{\Xi}} \circ Q_{\tilde{\lambda}} = Q_{\tilde{\mu}} \circ F_{\tilde{\Xi}} .\]
Hence $F_{\tilde{\Xi}}$ is a morphism of dg coalgebras from $\big(S(\frakg[1]),Q_{\tilde{\lambda}}\big)$
to $\big(S(\frakh[1]),Q_{\tilde{\mu}}\big)$.
Such an $L_\infty$ morphism $\Xi$ is said to be curved if $\Xi_0\neq 0$ and flat if $\Xi_0=0$.

An \textbf{$L_\infty$ module} over an $L_\infty$ algebra $(\frakg,\lambda)$ is a graded vector space $V = \bigoplus_{n \in \ZZ} V^n$ endowed with a sequence $(\alpha_j)_{j\in\NO}$ of maps $\alpha_j: \Lambda^j \frakg \otimes V \to V$ of degree $1-j$ satisfying the relations
\begin{multline}\label{eq:LooMod_def}
\sum_{\substack{p+q=k \\ p\geq 1 \\ q\geq 0}} \sum_{\sigma\in \shuf(p,q)} (\pm 1)\cdot
\alpha_{q+1}\big(\lambda_p(x_{\sigma(1)}\wedge\cdots\wedge x_{\sigma(p)})
\wedge x_{\sigma(p+1)}\wedge\cdots\wedge x_{\sigma(p+q)} \otimes v\big) \\
+\sum_{\substack{p+q=k \\ p\geq 0 \\ q\geq 0}} \sum_{\sigma\in \shuf(p,q)} (\pm 1)\cdot
\alpha_{p}\big(x_{\sigma(1)}\wedge\cdots\wedge x_{\sigma(p)}\otimes
\alpha_q(x_{\sigma(p+1)}\wedge\cdots\wedge x_{\sigma(p+q)}\otimes v)\big)=0
\end{multline}
for all $k\geq 0$ and all homogeneous vectors $x_1,\cdots,x_k\in\frakg$ and $v\in V$.

If Equation~\eqref{eq:LooMod_def} is rewritten in terms of the associated sequence
\( (\tilde{\alpha}_j)_{j\in\NO} \) of maps of degree one
\[ \tilde{\alpha}_j:= \alpha_j \circ \big((\pshift)^{\otimes j} \otimes \id_V \big)
: S^j(\frakg[1]) \otimes V \to V ,\]
it takes the form
\begin{multline}\label{eq:Loo[1]Mod_def}
\sum_{\substack{p+q=k \\ p\geq 1 \\ q\geq 0}} \sum_{\sigma\in \shuf(p,q)} \sgn(\sigma) \cdot
\tilde{\alpha}_{q+1}\big(\tilde{\lambda}_p(\nshift x_{\sigma(1)} \odot\cdots\odot \nshift x_{\sigma(p)})
\odot \nshift x_{\sigma(p+1)} \odot\cdots\odot \nshift x_{\sigma(p+q)} \otimes v\big) \\
+ \sum_{\substack{p+q=k \\ p\geq 0 \\ q\geq 0}} \sum_{\sigma\in \shuf(p,q)} \sgn(\sigma) \cdot
\tilde{\alpha}_{p}\big(\nshift x_{\sigma(1)} \odot\cdots\odot \nshift x_{\sigma(p)} \otimes
\tilde{\alpha}_q( \nshift x_{\sigma(p+1)} \odot\cdots\odot \nshift x_{\sigma(p+q)} \otimes v)\big) =0
,\end{multline}
where the (now much simpler) signs $\sgn(\sigma)$ are determined by the Koszul sign convention.

Denoting by $D_{\tilde{\alpha}}$ the degree-one codifferential on the cofreely cogenerated comodule $S(\frakg[1]) \otimes V$ (over the dg coalgebra $(S(\frakg[1]),Q_{\tilde{\lambda}})$) whose corestriction is the degree-one map $\tilde{\alpha} := \sum_{j=0}^\infty \tilde{\alpha}_j: S(\frakg[1]) \otimes V \to V$, we can rewrite the relations \eqref{eq:Loo[1]Mod_def} as a single equation:
\[ D_{\tilde{\alpha}} \circ D_{\tilde{\alpha}} =0 .\]
Hence $\big(S(\frakg[1]) \otimes V,D_{\tilde{\alpha}}\big)$ is a dg comodule.
See, for example, \cite{MR2199629,MR4693987,MR2004726}.

Let $(V,\alpha)$ and $(W,\beta)$ be two $L_\infty$ modules over an $L_\infty$ algebra $(\frakg,\lambda)$.
A \textbf{morphism of $L_\infty$ modules} from $(V,\alpha)$ to $(W,\beta)$ is a sequence $(\Phi_l)_{l\in\NO}$ of maps $\Phi_l: \Lambda^l \frakg \otimes V \to W$ of degree $-l$ satisfying the relations
\begin{multline}\label{eq:LooModMor_def}
\sum_{\substack{p+q=k \\ p\geq 1 \\ q\geq 0}} \sum_{\sigma\in \shuf(p,q)} (\pm 1)\cdot
\Phi_{q+1}\big(\lambda_p(x_{\sigma(1)} \wedge\cdots\wedge x_{\sigma(p)})
\wedge x_{\sigma(p+1)} \wedge\cdots\wedge x_{\sigma(p+q)} \otimes v\big) \\
+ \sum_{\substack{p+q=k \\ p\geq 0 \\ q\geq 0}} \sum_{\sigma\in \shuf(p,q)} (\pm 1)\cdot
\Phi_{p}\big(x_{\sigma(1)} \wedge\cdots\wedge x_{\sigma(p)}
\otimes \alpha_q(x_{\sigma(p+1)} \wedge\cdots\wedge x_{\sigma(p+q)} \otimes v)\big) \\
= \sum_{\substack{p+q=k \\ p\geq 0 \\ q\geq 0}} \sum_{\sigma\in \shuf(p,q)} (\pm 1)\cdot
\beta_{p}\big(x_{\sigma(1)} \wedge\cdots\wedge x_{\sigma(p)} \otimes
\Phi_q(x_{\sigma(p+1)} \wedge\cdots\wedge x_{\sigma(p+q)} \otimes v)\big)
.\end{multline}

In terms of the morphism of cofreely cogenerated comodules
\[ F_{\tilde{\Phi}}: S(\frakg[1]) \otimes V \to S(\frakg[1]) \otimes W \]
whose corestriction is the sum
$\tilde{\Phi} := \sum_{l=0}^\infty \tilde{\Phi}_l : S(\frakg[1]) \otimes V \to W$
of the associated degree-zero maps
\[ \tilde{\Phi}_l := \Phi_l \circ \big((\pshift)^{\otimes l} \otimes \id_V \big) :
S^l(\frakg[1]) \otimes V \to W ,\]
these relations (which involve rather complicated signs) can be rewritten as a single simple equation:
\[ F_{\tilde{\Phi}} \circ D_{\tilde{\alpha}} = D_{\tilde{\beta}} \circ F_{\tilde{\Phi}} .\]
Hence $F_{\tilde{\Phi}}$ is a morphism of dg comodules
from $\big(S(\frakg[1]) \otimes V,D_{\tilde{\alpha}}\big)$
to $\big(S(\frakg[1]) \otimes W,D_{\tilde{\beta}}\big)$.

The following two remarks are well-known --- see, for example, \cite[Remark~3.2.1]{MR1729368}.

\begin{remark}\label{rmk:LooModAsLooAlg}
The direct sum vector space $\frakg\oplus V$, which is automatically $\ZZ$-graded, can be thought of as a $\ZZ$-graded vector space with an alternate grading, called \emph{weight}, in which the direct summand $\frakg$ is assigned weight $0$ and the direct summand $V$ is assigned weight $1$.
With that convention, an $L_\infty$ module structure $\alpha$ on $V$ over an $L_\infty$ algebra $(\frakg,\lambda)$ is equivalent to an $L_\infty$ algebra structure $(\lambda_k^\alpha)_{k\in\NN}$ on $\frakg \oplus V$ satisfying the following conditions:
\begin{enumerate}[series=conditions,label=\textbf{(C\arabic*)},ref=(C\arabic*)]
\item\label{prima} $(\frakg,\lambda)$ is an $L_\infty$ subalgebra of $(\frakg \oplus V,\lambda^\alpha)$;
\item\label{seconda} $\weight(\lambda^\alpha_k)=0$ for all $k\in\NN$.
\end{enumerate}
Note that $\lambda_k^\alpha$ is a map from $\Lambda^k(\frakg \oplus V)$ to $\frakg \oplus V$ of degree $2-k$.
Condition~\ref{seconda} ensures that $\lambda^\alpha_k(y_1 \wedge \cdots \wedge y_k) = 0$
if two or more vectors amongst $y_1,\cdots, y_k \in \frakg \oplus V$ belong to the direct summand $V$.
For all $k\geq 0$; $x_1,\dots,x_k\in\frakg$; and $v\in V$, we have
\[ \alpha_k(x_1\wedge\cdots\wedge x_k\otimes v)=\lambda^\alpha_{k+1}(x_1\wedge\cdots\wedge x_k\wedge v) .\]
\end{remark}

\begin{remark}\label{rmk:LooModMorphAsLooAlgMorph}
Similarly, given two $L_\infty$ modules $(V,\alpha)$ and $(W,\beta)$ over the same $L_\infty$ algebra $(\frakg,\lambda)$, an $L_\infty$ module morphism $(\Theta_l)_{l\in\NO}$ from $(V,\alpha)$ to $(W,\beta)$ (with $\Theta_l: \Lambda^l \frakg \otimes V \to W$) is equivalent to an $L_\infty$ algebra morphism $(\Phi_k)_{k\in\NN}$ from $(\frakg\oplus V,\lambda^\alpha)$ to $(\frakg\oplus W,\lambda^\beta)$ (with $\Phi_k: \Lambda^k (\frakg \oplus V) \to \frakg \oplus W$)
satisfying the following conditions:
\begin{enumerate}[resume=conditions,label=\textbf{(C\arabic*)},ref=(C\arabic*)]
\item\label{quarta} $\Phi_1(x) = x$ if the vector $x$ belongs to the direct summand $\frakg$;
\item\label{quinta} $\Phi_k(x_1\wedge\cdots\wedge x_k) = 0$ if $x_1,\cdots,x_k\in\frakg$ and $k\geq 2$;
\item\label{terza} $\weight(\Phi_k) =0$ for all $k\in\NN$.
\end{enumerate}
For all $k\geq 0$; $x_1,\dots,x_k\in\frakg$; and $v\in V$, we have
\[ \Theta_k(x_1\wedge\cdots\wedge x_k\otimes v)=\Phi_{k+1}(x_1\wedge\cdots\wedge x_k\wedge v) .\]
\end{remark}

\subsection{Homotopy transfer theorem}

Let $(V,d)$ and $(W,\delta)$ be two dg vector spaces over a field $\KK$.
Recall that a \textbf{contraction} is the data
\begin{equation}\label{eq:contraction}
\begin{tikzcd}
(V,d) \arrow[r, "\tau", shift left,hook] & (W ,\delta)
\arrow[l, " \sigma", shift left,two heads] \arrow[loop, "h",out=12,in= -12,looseness = 3]
\end{tikzcd}
\end{equation}
where $\tau:V\to W$ is an injective cochain map,
$\sigma:W\to V$ is a surjective cochain map,
and $h: W \to W$ is a linear map of degree $-1$, which satisfy the relations
\[ \sigma\tau=\id_V, \qquad \id_W-\tau\sigma=h\delta+\delta h, \qquad
\sigma h=0, \qquad h\tau=0, \qquad h h=0 \]
--- see, for example, \cite[Appendix~A]{MR4665716} and references therein.

We begin by briefly recalling the well-known \emph{homotopy transfer theorem for $L_\infty$ algebras}.
More details can be found in~\cite[Theorem~12.4.1]{MR4485797} and~\cite{arXiv:1705.02880,MR4693987}.

\begin{theorem}\label{thm:LooTransfer}
Suppose we are given a contraction \eqref{eq:contraction} and an $L_\infty$ algebra structure
$\mu=(\mu_i)_{i=1}^\infty$ on $W$ such that $\mu_1=\delta$.
Then there exist an $L_\infty$ algebra structure $\lambda = (\lambda_i)_{i=1}^\infty$ on $V$
such that $\lambda_1 = d$, and two $L_\infty$ quasi-isomorphisms
\[ \Phi=(\Phi_j)_{j=1}^\infty : (V,\lambda) \inftyto (W,\mu)
\qquad \text{and} \qquad
\Psi=(\Psi_j)_{j=1}^\infty : (W,\mu) \inftyto (V,\lambda) \]
such that $\Phi_1=\tau$, $\Psi_1=\sigma$, and $F_{\tilde{\Psi}}\circ F_{\tilde{\Phi}} =\id_{S(V[1])}$.
\end{theorem}

Instead of giving formulae for the $L_\infty$ sequences $\lambda$; $\Phi$; and $\Psi$, we will describe explicitly the associated (and equivalent) $L_\infty[1]$ sequences $\tilde{\lambda}$; $\tilde{\Phi}$; and $\tilde{\Psi}$.

Recall that the coalgebra morphism $F_{\tilde{\Phi}}:S(V[1]) \to S(W[1])$ induced by $\tilde{\Phi}:S(V[1]) \to W[1]$ is given by
\begin{multline*} F_{\tilde{\Phi}}( v_1 \odot \cdots \odot v_n) = \sum_{k=1}^n \dfrac{1}{k!} \sum_{\substack{\sigma \in \shuf(i_1,\cdots, i_k) \\ i_1 + \cdots + i_k =n \\ i_1,\cdots, i_k \geq 1}} \epsilon \cdot \tilde{\Phi}_{i_1}(v_{\sigma(1)} \odot \cdots \odot v_{\sigma(i_1)} ) \\ \odot \tilde{\Phi}_{i_2}(v_{\sigma(i_1 +1)} \odot \cdots \odot v_{\sigma(i_1+i_2)} ) \odot \cdots \odot \tilde{\Phi}_{i_k}(v_{\sigma(i_1+\cdots+i_{k-1} +1)}\odot\cdots \odot v_{\sigma(n)})
,\end{multline*}
where $v_1,\cdots, v_n \in V[1]$; $\shuf(i_1,\cdots, i_k)$ denotes the set of permutations $\sigma$ of $\{1,\cdots, n\}$ such that $\sigma(j) < \sigma(j+1)$ for every $j\notin\{i_1,i_1+i_2,\cdots,i_1+i_2+\cdots+i_{k-1}\}$; and the sign $\epsilon =\pm 1$ is determined by the Koszul convention.
We denote by $\tilde{\Phi}_i^j$ the composition
\[ \begin{tikzcd}
S^i(V[1]) \ar[r,hook] \ar[rrr,bend left = 15,"\tilde{\Phi}_i^j"] & S(V[1]) \ar[r,"F_{\tilde{\Phi}}"'] & S(W[1]) \ar[r, two heads] & S^j(W[1]) .
\end{tikzcd} \]
It is clear that the map $\tilde{\Phi}_i^j$ is determined by $\tilde{\Phi}_1,\cdots,\tilde{\Phi}_{i-j+1}$.
(Here $\tilde{\Phi}_1=\nshift \circ \Phi_1 \circ \pshift=\nshift \circ \tau \circ \pshift=\tilde{\tau}$.)

Also recall that the coderivation $Q_{\tilde{\mu}}:S(W[1]) \to S(W[1])$ induced by $\tilde{\mu}:S(W[1]) \to W[1]$ is given by the formula
\[ Q_{\tilde{\mu}}(w_1\odot \cdots \odot w_n) = \sum_{i=1}^n \sum_{\sigma \in \shuf(i,n-i)} \epsilon\cdot \tilde{\mu}_i(w_{\sigma(1)} \odot \cdots \odot w_{\sigma(i)}) \odot w_{\sigma(i+1)} \odot \cdots \odot w_{\sigma(n)} ,\]
where $w_1,\cdots, w_n \in W[1]$ and the sign $\epsilon =\pm 1$ is given by the Koszul sign convention.
We denote by $\tilde{\mu}_i^j$ the composition
\[ \begin{tikzcd}
S^i(W[1]) \ar[r,hook] \ar[rrr,bend left = 15,"\tilde{\mu}_i^j"] & S(W[1]) \ar[r,"Q_{\tilde{\mu}}"'] & S(W[1]) \ar[r, two heads] & S^j(W[1])
.\end{tikzcd} \]
It is clear that $\tilde{\mu}_i^j$ (with $i\geq j$) is determined by $\tilde{\mu}_{i-j+1}$:
\[ \tilde{\mu}_i^j(w_1\odot \cdots \odot w_i) = \sum_{\sigma\in\shuf(i-j+1,j-1)} \epsilon\cdot \tilde{\mu}_{i-j+1}(w_{\sigma(1)} \odot \cdots \odot w_{\sigma(i-j+1)}) \odot w_{\sigma(i-j+2)} \odot \cdots \odot w_{\sigma(i)} .\]

The ``Taylor coefficients'' $\tilde{\Phi}_i$, $\tilde{\mu}_i$, $\tilde{\Psi}_i$
(which are equivalent to the ones in Theorem~\ref{thm:LooTransfer})
are determined by the following recursive formulae:
\begin{align}
\tilde{\Phi}_i & = \sum_{j=2}^i \tilde{h} \circ \tilde{\mu}_j \circ \tilde{\Phi}_i^j,
& \tilde{\Phi}_1 & =\tilde{\tau},
\label{eq:LooTransfer_Inj} \\
\tilde{\lambda}_i & = \sum_{j=2}^i \tilde{\sigma} \circ \tilde{\mu}_j \circ \tilde{\Phi}_i^j,
& \tilde{\lambda}_1 & = \tilde{d},
\label{eq:LooTransfer_Loo} \\
\tilde{\Psi}_i & = \sum_{j=1}^{i-1} \tilde{\Psi}_j \circ \tilde{\mu}_i^j \circ \tilde{H}_i,
& \tilde{\Psi}_1 & = \tilde{\sigma},
\label{eq:LooTransfer_Suj}
\end{align}
where $\tilde{h} = \nshift \circ h \circ \pshift$;
$\tilde{\sigma} = \nshift \circ \sigma \circ \pshift$;
$\tilde{\tau} = \nshift \circ \tau \circ \pshift$;
$\tilde{d} = \nshift \circ d \circ \pshift$
and
\begin{multline}\label{eq:LooTransfer_Htp}
\tilde H_i(w_1 \odot \cdots \odot w_i) \\
= \dfrac{1}{i !} \sum_{\sigma \in S_i} \sum_{j=1}^i \epsilon \cdot \tilde{\tau}\tilde{\sigma}(w_{\sigma(1)} ) \odot \cdots \odot \tilde{\tau}\tilde{\sigma}(w_{\sigma(j-1)} ) \odot \tilde h(w_{\sigma(j)}) \odot w_{\sigma(j+1)} \odot \cdots \odot w_{\sigma(i)},
\end{multline}
for all $w_1,\cdots, w_i \in W[1]$.
The sign $\epsilon = \pm 1$ in \eqref{eq:LooTransfer_Htp} is determined by the Koszul sign convention.

\begin{proposition}[{\cite[Lemma~1.11]{arXiv:1705.02880}}]
\label{prop:LooTransfer-Inj-Surj}
Suppose we are given a contraction \eqref{eq:contraction},
an $L_\infty$ algebra structure $\mu=(\mu_i)_{i=1}^\infty$ on $W$ with $\mu_1=\delta$,
and an $L_\infty$ algebra structure $\lambda'=(\lambda'_i)_{i=1}^\infty$ on $V$ with $\lambda'_1=d$.
\begin{enumerate}
\item\label{prop:LooTransfer-Inj}
If $\tau$ is a strict $L_\infty$ algebra morphism from $(V,\lambda')$ to $(W,\mu)$,
i.e. $\tau\circ\lambda'_i=\mu_i\circ\tau^{\otimes i}$ for all $i\in\NN$,
then $\lambda'$ is precisely the homotopy transferred $L_\infty$ algebra structure $\lambda$ on $V$
(whose existence is asserted by Theorem~\ref{thm:LooTransfer}),
and the $L_\infty$ quasi-isomorphism $\Phi:(V,\lambda)\inftyto(W,\mu)$
(with $\Phi_1=\tau$) of Theorem~\ref{thm:LooTransfer} satisfies $\Phi_j=0$ for all $j\geq 2$.
\item\label{prop:LooTransfer-Surj}
If $\sigma$ is a strict $L_\infty$ algebra morphism from $(W,\mu)$ to $(V,\lambda')$,
i.e. $\sigma\circ\mu_i=\lambda'_i\circ\sigma^{\otimes i}$ for all $i\in\NN$,
then $\lambda'$ is precisely the homotopy transferred $L_\infty$ algebra structure $\lambda$ on $V$
(whose existence is asserted by Theorem~\ref{thm:LooTransfer}),
and the $L_\infty$ quasi-isomorphism $\Psi:(W,\mu)\inftyto(V,\lambda)$ (with $\Psi_1=\sigma$)
of Theorem~\ref{thm:LooTransfer} satisfies $\Psi_j=0$ for all $j\geq 2$.
\end{enumerate}
\end{proposition}

\begin{proof}
The proof of the second assertion can be found in \cite{arXiv:1705.02880} and is therefore omitted.

Assuming $\tau$ is a strict morphism, we begin by proving that $\Phi_j=0$ for all $j\geq 2$.
We will reason by induction on $j$.
From Equation~\eqref{eq:LooTransfer_Inj}, it follows that
\[ \Phi_2=\pm h\circ\mu_2\circ(\Phi_1\otimes\Phi_1)=\pm h\circ\mu_2\circ(\tau\otimes\tau)
=\pm h\circ\tau\circ\lambda'_2=0 \] since $\Phi_1=\tau$, $\tau$ is a strict $L_\infty$ morphism, and $h\circ\tau=0$.
Assuming that $\Phi_3=\Phi_4=\cdots=\Phi_{j-1}=0$, it follows from Equation~\eqref{eq:LooTransfer_Inj} again that
\begin{multline*}
\Phi_j(v_1,\cdots,v_j) = \dfrac{1}{j!}\sum_{\sigma\in S_j} (\pm 1)\cdot h
\circ\mu_j \big(\Phi_1(v_{\sigma(1)}),\cdots,\Phi_1(v_{\sigma(j)})\big) \\
= h\circ\mu_j \big(\Phi_1(v_1),\cdots,\Phi_1(v_j)\big)
= h\circ\mu_j \big(\tau(v_1),\cdots,\tau(v_j)\big) \\
= h\circ\tau\circ\lambda'_j (v_1,\cdots,v_j) =0
.\end{multline*}
This completes the induction.

Next, we show that $\lambda'=\lambda$. By assumption $\lambda'_1=d=\lambda_1$.
Since $\Phi_k=0$ for all $k\geq 2$; since $\Phi_1=\tau$;
$\tau$ is a strict $L_\infty$ morphism; and $\sigma\circ\tau=\id_V$,
it follows from Equation~\eqref{eq:LooTransfer_Loo} that, for all $i\geq 2$ and all $v_1,\cdots,v_i\in V$,
\begin{multline*}
\lambda_i(v_1,\cdots,v_i) =\sigma\circ\mu_i \big(\Phi_1(v_1),\cdots,\Phi_1(v_i)\big)
=\sigma\circ\mu_i \big(\tau(v_1),\cdots,\tau(v_i)\big) \\
=\sigma\circ\tau\circ\lambda'_i (v_1,\cdots,v_i) =\lambda'_i (v_1,\cdots,v_i)
.\end{multline*}
The first assertion is proven.
\end{proof}

Next, we derive a \emph{homotopy transfer theorem for $L_\infty$ modules}
from the homotopy transfer theorem for $L_\infty$ algebras.

\begin{theorem}\label{thm:LooModTransfer}
Suppose we are given a contraction
\begin{equation}\label{gemelli}
\begin{tikzcd}
(V,d) \arrow[r, "\tau", shift left,hook] & (W,\delta)
\arrow[l, " \sigma", shift left,two heads] \arrow[loop, "h",out=12,in= -12,looseness = 3]
\end{tikzcd}
\end{equation}
and an $L_\infty$ module structure $\beta = (\beta_j)_{j=0}^\infty$ on $W$ over an $L_\infty$ algebra $(\frakg,\lambda)$ such that $\beta_0 = \delta$.
Then there exist an $L_\infty$ module structure $\alpha = (\alpha_j)_{j=0}^\infty$ on $V$ over the same $L_\infty$ algebra $(\frakg,\lambda)$ with $\alpha_0=d$, and two $L_\infty$ module quasi-isomorphisms
\[ \Theta = (\Theta_k)_{k=0}^\infty : (V,\alpha) \inftyto (W,\beta)
\quad \text{and} \quad
\Xi = (\Xi_k)_{k=0}^\infty : (W,\beta) \inftyto (V,\alpha) \]
such that $\Theta_0 = \tau$, $\Xi_0 = \sigma$,
and $F_{\tilde{\Xi}} \circ F_{\tilde{\Theta}} = \id_{S(\frakg[1])\otimes V}$.
\end{theorem}

\begin{proof}
Following Remark~\ref{rmk:LooModAsLooAlg}, we denote by $\lambda^\beta = (\lambda_j^\beta)_{j=1}^\infty$ the $L_\infty$ algebra structure on $\frakg \oplus W$ induced by the $L_\infty$ module structure $\beta$ on $W$.
The contraction \eqref{gemelli} induces a contraction
\begin{equation}\label{eq:contraction_Mod}
\begin{tikzcd}
(\frakg \oplus V,\lambda_1\oplus d) \arrow[r, "\id_{\frakg}\oplus\tau", shift left,hook]
& (\frakg \oplus W,\lambda_1\oplus\delta) \arrow[l, "\id_{\frakg}\oplus\sigma", shift left,two heads]
\arrow[loop, "0_{\frakg}\oplus h", out=12, in=-12, looseness=3]
\end{tikzcd}.
\end{equation}

It follows from the assumptions that $\lambda_1^\beta= \lambda_1\oplus\beta_0 =\lambda_1\oplus\delta$.
Thus, according to the homotopy transfer theorem for $L_\infty$ algebras (Theorem~\ref{thm:LooTransfer}),
there exist an $L_\infty$ algebra structure $\kappa$ on $\frakg\oplus V$ and two $L_\infty$ quasi-isomorphisms
$\Phi: (\frakg\oplus V,\kappa) \inftyto (\frakg\oplus W, \lambda^\beta)$
and $\Psi: (\frakg\oplus W, \lambda^\beta) \inftyto (\frakg\oplus V,\kappa)$,
which are respectively determined by the recursive equations
\begin{align}
\Phi_i & = \sum_{j=2}^i (0_{\frakg}\oplus h) \circ \lambda^\beta_j \circ \Phi_i^j,
& \Phi_1 & = \id_{\frakg}\oplus\tau, \label{recursive-into} \\
\kappa_i & = \sum_{j=2}^i (\id_{\frakg}\oplus\sigma) \circ \lambda^\beta_j \circ \Phi_i^j,
& \kappa_1 & =\lambda_1\oplus d, \label{recursive-Loo} \\
\Psi_i & = \sum_{j=1}^{i-1} \Psi_j \circ \big(\lambda^\beta\big)_i^j \circ H_i,
& \Psi_1 & = \id_{\frakg}\oplus\sigma, \label{recursive-onto}
\end{align}
with
\begin{multline}\label{tensor-trick-homotopy}
H_i(y_1\wedge\cdots\wedge y_i) \\
= \dfrac{1}{i !} \sum_{\sigma\in S_i} \sum_{j=1}^i (\pm 1)\cdot
(\id_{\frakg}\oplus\tau\sigma)(y_{\sigma(1)}) \wedge\cdots\wedge
(\id_{\frakg}\oplus\tau\sigma)(y_{\sigma(j-1)}) \\
\wedge (0_{\frakg}\oplus h)(y_{\sigma(j)}) \wedge y_{\sigma(j+1)}
\wedge\cdots\wedge y_{\sigma(i)},
\end{multline}
for all $y_1,\cdots,y_i\in\frakg\oplus W$.

We proceed with showing that the transferred $L_\infty$ structure $\kappa$ satisfies the conditions in Remark~\ref{rmk:LooModAsLooAlg} and that the $L_\infty$ algebra morphisms $\Phi$ and $\Psi$ satisfy the conditions in Remark~\ref{rmk:LooModMorphAsLooAlgMorph}.
Once that is done, setting
\begin{gather*}
\alpha_k(x_1,\cdots,x_k,v)=\kappa_{k+1}(x_1,\cdots,x_k,v) ,\\
\Theta_k(x_1,\cdots,x_k,v)=\Phi_{k+1}(x_1,\cdots,x_k,v) ,\\
\intertext{and}
\Xi_k(x_1,\cdots,x_k,w)=\Psi_{k+1}(x_1,\cdots,x_k,w)
\end{gather*}
for all $k\in\NO$, $x_1,\dots,x_k\in\frakg$, $v\in V$ and $w\in W$,
we obtain the desired $L_\infty$ module structure $\alpha$ on $V$
and $L_\infty$ module morphisms $\Theta:(V,\alpha)\inftyto(W,\beta)$
and $\Xi:(W,\beta)\inftyto(V,\alpha)$.
It follows immediately from $\kappa_1=\lambda_1\oplus d$; $\Phi_1=\id_{\frakg}\oplus\tau$;
and $\Psi_1=\id_{\frakg}\oplus\sigma$ that $\alpha_0=d$; $\Theta_0=\tau$, and $\Xi_0=\sigma$.

\textsc{Step~1.} We start with verifying, by induction on $i$, that $\Phi = (\Phi_i)_{i=1}^\infty$ satisfies the conditions \ref{quarta}, \ref{quinta} and \ref{terza} in Remark~\ref{rmk:LooModMorphAsLooAlgMorph}.

Since $\Phi_1=\id_{\frakg}\oplus\tau$, it satisfies \ref{quarta} and \ref{terza}.

According to Equation~\eqref{recursive-into}, the map $\Phi_i$ is given by
\begin{multline}\label{batterie}
\Phi_i(y_1,\cdots,y_i) = \sum_{j=2}^i \dfrac{1}{j!}
\sum_{\substack{\sigma\in\shuf(k_1,\cdots,k_j) \\ k_1+\cdots+k_j=i \\ k_1,\cdots,k_j\geq 1}}
(\pm 1)\cdot (0_{\frakg}\oplus h)\circ\lambda^\beta_j\big(
\Phi_{k_1}(y_{\sigma(1)},\cdots,y_{\sigma(k_1)}), \\ \Phi_{k_2}(y_{\sigma(k_1+1)},\cdots,y_{\sigma(k_1+k_2)}),
\cdots,\Phi_{k_j}(y_{\sigma(k_1+\cdots+k_{j-1}+1)},\cdots,y_{\sigma(i)})\big)
\end{multline}
for all $y_1,\cdots,y_i\in\frakg\oplus V$.
Assuming that $\Phi_2,\cdots,\Phi_{i-1}$ satisfy \ref{terza},
it follows that $\Phi_i$ does satisfy \ref{terza} as well because $\lambda^\beta_j$ has weight zero.

Furthemore, assuming $\Phi_2,\cdots,\Phi_{i-1}$ satisfy \ref{quinta},
it follows that $\Phi_i$ does satisfy \ref{quinta} as well.
Indeed, if $x_1,\cdots,x_i$ do all belong to the direct summand $\frakg$, then
\[ \lambda^\beta_j\big(
\Phi_{k_1}(x_{\sigma(1)},\cdots,x_{\sigma(k_1)}),\Phi_{k_2}(x_{\sigma(k_1+1)},\cdots,x_{\sigma(k_1+k_2)}),
\cdots,\Phi_{k_j}(x_{\sigma(k_1+\cdots+k_{j-1}+1)},\cdots,x_{\sigma(i)})\big) \]
also belongs to the direct summand $\frakg$ and Equation~\eqref{batterie} yields
$\Phi_i(x_1,\cdots,x_i)=0$.

This completes the induction.

\textsc{Step~2.} Next, we verify that $\kappa = (\kappa_i)_{i=1}^\infty$ satisfies the conditions \ref{prima} and \ref{seconda} in Remark~\ref{rmk:LooModAsLooAlg}.

According to Equation~\eqref{recursive-Loo}, we have
\begin{multline*}
\kappa_i(y_1,\cdots,y_i)=\sum_{j=2}^i \dfrac{1}{j!}
\sum_{\substack{\sigma \in \shuf(k_1,\cdots, k_j) \\ k_1 + \cdots + k_j = i \\ k_1,\cdots, k_j \geq 1}}
(\pm 1)\cdot (\id_{\frakg}\oplus\sigma)\circ\lambda^\beta_j
\big( \Phi_{k_1}(y_{\sigma(1)},\cdots,y_{\sigma(k_1)}), \\
\Phi_{k_2}(y_{\sigma(k_1+1)},\cdots,y_{\sigma(k_1+k_2)}), \cdots,
\Phi_{k_j}(y_{\sigma(k_1+\cdots+k_{j-1}+1)},\cdots,y_{\sigma(i)}) \big)
\end{multline*}
for all $y_1,\cdots,y_i\in\frakg\oplus V$.
In particular, since $\Phi$ satisfies \ref{quarta} and \ref{quinta}, we have
\begin{multline*}
\kappa_i(x_1,\cdots,x_i)=\dfrac{1}{i!}\sum_{\sigma\in S_i} (\pm 1)\cdot
(\id_{\frakg}\oplus\sigma)\circ\lambda^\beta_i(x_{\sigma(1)},\cdots,x_{\sigma(i)})
\\
=\dfrac{1}{i!}\sum_{\sigma\in S_i} (\pm 1)\cdot
(\id_{\frakg}\oplus\sigma)\circ\lambda_i(x_{\sigma(1)},\cdots,x_{\sigma(i)})
=\lambda_i(x_1,\cdots,x_i)
\end{multline*}
for all $x_1,\cdots,x_i\in\frakg$.
Condition \ref{prima} is thus satisfied by $\kappa$.
Furthermore, since the $\Phi_k$'s and $\lambda^\beta_j$'s all have weight $0$,
every $\kappa_i$ also has weight zero.
Condition \ref{seconda} is thus satisfied by $\kappa$.

\textsc{Step~3.} Finally, we verify, by induction on $i$, that $\Psi = (\Psi_i)_{i=1}^\infty$ satisfies conditions \ref{quarta}, \ref{quinta} and \ref{terza} in Remark~\ref{rmk:LooModMorphAsLooAlgMorph}.

Since $\Psi_1=\id_{\frakg}\oplus\sigma$, it satisfies conditions \ref{quarta} and \ref{terza}.

It is immediate that $H_i(x_1,\cdots,x_i)=0$ if $x_1,\cdots,x_i$ do all belong to the direct summand $\frakg$.
It follows from Equation~\eqref{recursive-onto} that $\Psi_i(x_1,\cdots,x_i)=0$
if $i\geq 2$ and $x_1,\cdots,x_i\in\frakg$.
Condition \ref{quinta} is thus satisfied by $\Psi_i$

Assuming that $\Psi_2,\cdots,\Psi_{i-1}$ satisfy condition \ref{terza},
it follows from Equation~\eqref{recursive-onto} that $\Psi_i$ does satisfy \ref{terza} as well
because $H_i$ and $\lambda^\beta$ have weight zero.

This completes the proof.
\end{proof}

The next proposition follows from Proposition~\ref{prop:LooTransfer-Inj-Surj}
and Remarks~\ref{rmk:LooModAsLooAlg} and~\ref{rmk:LooModMorphAsLooAlgMorph}.

\begin{proposition}\label{prop:LooModTransfer-Inj-Surj}
Suppose we are given a contraction \eqref{eq:contraction},
an $L_\infty$ module structure $\beta=(\beta_j)_{j=0}^\infty$ on $W$
over an $L_\infty$ algebra $(\frakg,\lambda)$ with $\beta_0=\delta$,
and an $L_\infty$ module structure $\alpha'=(\alpha'_j)_{j=0}^\infty$ on $V$
over the same $L_\infty$ algebra $(\frakg,\lambda)$ with $\alpha'_0=d$.
\begin{enumerate}
\item\label{prop:LooModTransfer-Inj}
If $\tau$ is a strict $L_\infty$ module morphism from $(V,\alpha')$ to $(W,\beta)$,
i.e. $\tau\circ\alpha'_j=\beta_j\circ(\id_{S^j(\frakg)}\otimes\tau)$ for all $j\in\NO$,
then $\alpha'$ is precisely the homotopy transferred $L_\infty$ module structure $\alpha$ on $V$
(whose existence is asserted by Theorem~\ref{thm:LooModTransfer})
and the $L_\infty$ module quasi-isomorphism $\Theta:(V,\alpha)\inftyto(W,\beta)$
(with $\Theta_0=\tau$) of Theorem~\ref{thm:LooModTransfer} satisfies
$\Theta_j=0$ for all $j\geq 1$.
\item\label{prop:LooModTransfer-Surj}
If $\sigma$ is a strict $L_\infty$ module morphism from $(W,\beta)$ to $(V,\alpha')$,
i.e. $\sigma\circ\beta_j=\alpha'_j\circ(\id_{S^j(\frakg)}\otimes\sigma)$ for all $j\in\NO$,
then $\alpha'$ is precisely the homotopy transferred $L_\infty$ module structure $\alpha$ on $V$
(whose existence is asserted by Theorem~\ref{thm:LooModTransfer})
and the $L_\infty$ module quasi-isomorphism $\Xi:(W,\beta)\inftyto(V,\alpha)$
(with $\Xi_0=\sigma$) of Theorem~\ref{thm:LooModTransfer} satisfies
$\Xi_j=0$ for all $j\geq 1$.
\end{enumerate}
\end{proposition}

\subsection{\texorpdfstring{Pullbacks of $L_\infty$ modules}{Pullbacks of L-infinity modules}}
\label{sec:PullBack}

We summarize the necessary definitions and properties of pullbacks of $L_\infty$ modules in this section. See~\cite{MR4693987} for further details.

Let $\frakg$ be an $L_\infty$ algebra and let $V$ be a graded vector space.
It is well-known that an $L_\infty$ module structure $\alpha$ on $V$ over $\frakg$
is equivalent to a curved morphism of $L_\infty$ algebras
\[ \hat{\alpha}: \frakg \inftyto \End(V), \]
where $\End(V)$ denotes the dgla of endomorphisms on $V$, equipped with the zero differential and the bracket of graded commutators. See \cite[Lemma~7.7]{MR4693987}.

Let $\frakh$ be another $L_\infty$ algebra, and let $\Xi:\frakh \inftyto \frakg$ be an $L_\infty$ morphism. The composition
\[ \hat{\alpha} \circ \Xi : \frakh \inftyto \End(V) \]
defines an $L_\infty$ module structure on $V$ over $\frakh$, denoted by $\Xi^\ast\alpha$. Given an $L_\infty$ module $M=(V,\alpha)$ over $\frakg$, the notation $\Xi^\ast M$ refers to the {\em pullback $L_\infty$ module} $(V,\Xi^\ast\alpha)$.
Furthermore, if $\Theta:\frakh' \inftyto \frakh$ is another $L_\infty$ morphism, then $(\Xi \circ \Theta)^\ast M = \Theta^\ast(\Xi^\ast M)$.

Let $N = (W,\beta)$ be another $L_\infty$ module over $\frakg$, and let $\Phi: M \inftyto N$ be a morphism of $L_\infty$ modules over $\frakg$. Such a morphism can be equivalently encoded as a curved morphism of $L_\infty$ algebras
\[ \hat{\Phi}: \frakg \inftyto \End(V[1] \oplus W) \]
such that
\[ \hat{\Phi}_k(x_1, \cdots, x_k) = \begin{pmatrix} \bar\alpha_k(x_1, \cdots, x_k) & 0 \\ \bar\Phi_k(x_1, \cdots, x_k) & \beta_k(x_1, \cdots, x_k) \end{pmatrix}, \]
where
\[ \bar\alpha_k(x_1, \cdots, x_k)(\nshift v)
= (-1)^{|x_1|+ \cdots + |x_k|+ 1 -k}\nshift(\alpha_k(x_1, \cdots, x_k)(v)) \]
and
\[ \bar\Phi_k(x_1, \cdots, x_k)(\nshift v)
= (-1)^{|x_1|+ \cdots + |x_k| - k}\nshift(\Phi_k(x_1, \cdots, x_k)(v)) .\]
The composition $\hat{\Phi} \circ \Xi: \frakh \inftyto \End(V[1]\oplus W)$
therefore defines an $L_\infty$ module morphism
\[ \Xi^\ast \Phi: \Xi^\ast M \inftyto \Xi^\ast N \]
over $\frakh$, satisfying $(\Xi \circ \Theta)^\ast \Phi = \Theta^\ast(\Xi^\ast \Phi)$.

According to~\cite[Section~6]{MR4693987},
a morphism $\Xi:\frakh \inftyto \frakg$ of $L_\infty$ algebras can be identified
with a Maurer--Cartan element in the convolution $L_\infty$ algebra $\Hom(S^{\geq 1}(\frakh[1]),\frakg)$
defined in~\cite[Proposition~6.1]{MR4693987}.
Two $L_\infty$ morphisms $\Xi,\Theta$ from $\frakh$ to $\frakg$ are said to be \emph{homotopic},
which is denoted by $\Xi\sim\Theta$,
if they are homotopy equivalent as Maurer--Cartan elements in $\Hom(S^{\geq 1}(\frakh[1]),\frakg)$
in the sense of \cite[Definition~5.17]{MR4693987} --- see \cite[Definition~6.3]{MR4693987}.

The next two lemmas are used in Section~\ref{sec:FormalityDGmfd}.

\begin{lemma}[{\cite[Corollary~7.19]{MR4693987}}]\label{lem:PullbackViaHtpMor}
Let $\Xi$ and $\Theta$ be two homotopic morphisms of $L_\infty$ algebras from $\frakh$ to $\frakg$,
and let $M$ be an $L_\infty$ module over $\frakg$. Then there exists an isomorphism
\[ \Phi:\Xi^\ast M \inftyto \Theta^\ast M \]
of $L_\infty$ modules such that $\Phi_0=\id_{M}$.
\end{lemma}

\begin{lemma}[{\cite[Corollary~6.14]{MR4693987}}]\label{lem:HtpMor-HtpTransfer}
In Theorem~\ref{thm:LooTransfer}, the morphisms of $L_\infty$ algebras
$\Phi:(V,\lambda)\inftyto(W,\mu)$ and $\Psi:(W,\mu)\inftyto(V,\lambda)$
satisfy $\Phi\circ\Psi\sim\id_W$.
\end{lemma}

Since we could not find a proof of the following lemma in the literature,
we provide one here for completeness.

\begin{lemma}\label{lem:LiftedLooMor-htp}
Let $\frakg$ be an $L_\infty$ algebra with structure maps $\lambda=(\lambda_i)_{i=1}^\infty$,
and let $f$ be a cochain endomorphism of the underlying cochain complex $(\frakg,\lambda_1)$.
If $f$ is cochain homotopic to the identity endomorphism $\id_{\frakg}$,
then there exists an $L_\infty$ morphism $F:\frakg \inftyto \frakg$ such that $F\sim\id_{\frakg}$ and $F_1=f$.
\end{lemma}

\begin{proof}
First, observe that the identity $L_\infty$ morphism
$\id_{\frakg}:\frakg\inftyto\frakg$ is characterized by
the sequence $(\tilde{q}_k)_{k=1}^\infty$ of degree-zero maps
$\tilde{q}_k:S^k(\frakg[1])\to\frakg[1]$ in which
$\tilde{q}_1=\id_{\frakg[1]}$ and $\tilde{q}_k=0$ for all $k\geq 2$.
Note that $\tilde{q}:=\sum_{k=1}^\infty \tilde{q}_k$ is the canonical projection
of $S^{\geq 1}(\frakg[1])$ onto $\frakg[1]$.

According to \cite[Definitions~5.17 and~6.3]{MR4693987},
an $L_\infty$ morphism $F:\frakg\inftyto\frakg$
(with associated sequence $(\tilde{F}_k)_{k=1}^\infty$ of degree-zero maps
$\tilde{F}_k:S^k(\frakg[1])\to\frakg[1]$) is homotopic to
the identity $L_\infty$ morphism $\id_{\frakg}$
if there exists a solution
\[ \begin{cases} \pi(t)=\sum_{i=0}^\infty \pi_i \cdot t^i
& \text{with}\ \pi_i\in\Hom^1\big(S^{\geq 1}(\frakg[1]),\frakg[1]\big) \\
\varpi(t)=\sum_{i=0}^\infty \varpi_i \cdot t^i
& \text{with}\ \varpi_i\in\Hom^0\big(S^{\geq 1}(\frakg[1]),\frakg[1]\big)
\end{cases} \]
to the boundary value problem
\[ \pi'(t) = \sum_{j=0}^\infty \frac{1}{j!} \tilde{\nu}_{1+j}\big(\varpi(t),\pi(t),\dots,\pi(t)\big),
\qquad \pi(0)=\sum_{k=1}^\infty\tilde{q}_k, \qquad \pi(1)=\sum_{k=1}^\infty \tilde{F}_k .\]
Here, the sequence $(\tilde{\nu}_j)_{j=1}^\infty$ of degree-zero maps
\[ \tilde{\nu}_j:S^j\big(\Hom(S^{\geq 1}(\frakg[1]),\frakg[1])\big) \to \Hom(S^{\geq 1}(\frakg[1]),\frakg[1]) \]
is the convolution $L_\infty$ algebra structure on $\Hom(S^{\geq 1}(\frakg[1]),\frakg[1])$.

Let $\kappa:\frakg^\bullet \to \frakg^{\bullet-1}$ be a cochain homotopy between $\id_{\frakg}$ and $f$,
so that $f - \id_{\frakg} = \lambda_1 \kappa + \kappa \lambda_1$.
Suspending the coboundary operator $\lambda_1$, the homotopy $\kappa$,
and the chain map $f$, we obtain the endomorphisms $\tilde{\lambda}_1=\nshift\circ\lambda_1\circ\pshift$, $\tilde{\kappa}=\nshift\circ\kappa\circ\pshift$, and $\tilde{f}=\nshift\circ f\circ\pshift$ of $\frakg[1]$.

Set $\varpi_0=\tilde{\kappa}\circ\tilde{q}$ and $\varpi_i=0$ for all $i\geq 1$.
Then define $\pi(t)$ as the solution of the integral equation
\begin{equation}\label{integral_equation}
\pi(t) = \pi(0) + \int_0^t \sum_{j=0}^\infty \frac{1}{j!}
\tilde{\nu}_{1+j}\big(\varpi_0,\pi(s),\dots,\pi(s)\big) \ ds
\end{equation}
with initial condition $\pi(0)=\pi_0=\tilde{q}$.
The coefficients $\pi_i$ of the power series $\pi(t)$ can be computed recursively.

According to \cite[Proposition~5.19]{MR4693987},
evaluation at any number $t\in[0,1]$ of the power series $\pi(t)$ defined in this way
yields an $L_\infty$ endomorphism of $\frakg$.
In particular, evaluation of $\pi(t)$ at $t=1$ yields an $L_\infty$ morphism
$F:\frakg\inftyto\frakg$ whose $k$-th Taylor coefficient $\tilde{F}_k=\pi(1)\circ\tilde{\jmath}_k$
is the composition of the canonical inclusion $\tilde{\jmath}_k:S^k(\frakg[1])\big)\into S^{\geq 1}(\frakg[1])$
and $\pi(1):S^{\geq 1}(\frakg[1])\to\frakg[1]$.
By its very construction, this $L_\infty$ morphism $F$ satisfies $F\sim\id_{\frakg}$.

It follows from Equation~\eqref{integral_equation} that
\[ \pi(1)=\pi_0+\nu_1(\varpi_0)+\int_0^1 \sum_{j=1}^\infty\frac{1}{j!}
\nu_{1+j}\big(\varpi_0,\pi(s),\cdots,\pi(s)\big) \ ds .\]
According to \cite[Equation~(6.1)]{MR4693987}, we have
\[ \nu_1(\varpi_0)=\nu_1(\tilde{\kappa}\circ\tilde{q}\big)
=\tilde{\lambda}_1\circ\tilde{\kappa}\circ\tilde{q}
-(-1)^1\tilde{\kappa}\circ\tilde{q}\circ Q_{\tilde{\lambda}} .\]
Furthermore, it is straightforward to check that
\[ \nu_{1+j}(\varpi_0,\pi_{i_1},\cdots,\pi_{i_j})\in\Hom\big(S^{\geq 1}(\frakg[1]),\frakg[1]\big) \]
vanishes on $S^1(\frakg[1])$ for all $j\geq 1$ and all $i_1,\dots,i_j$.
Therefore, we obtain
\begin{multline*}
\tilde{F}_1=\pi(1)\circ\tilde{\jmath}_1=\big(\pi_0+\nu_1(\varpi_0)\big)\circ\tilde{\jmath}_1
=\tilde{q}\circ\tilde{\jmath}_1
+\tilde{\lambda}_1\circ\tilde{\kappa}\circ\tilde{q}\circ\tilde{\jmath}_1
+\tilde{\kappa}\circ\tilde{q}\circ Q_{\tilde{\lambda}}\circ\tilde{\jmath}_1
\\
=\id_{\frakg[1]}+\tilde{\lambda}_1\circ\tilde{\kappa}+\tilde{\kappa}\circ\tilde{\lambda}_1=\tilde{f}
.\end{multline*}
The proof is complete.
\end{proof}

The following lemma is also used in Section~\ref{sec:FormalityDGmfd}.

\begin{lemma}\label{lem:Pyramides}
Let $M$ be an $L_\infty$ module (with $L_\infty$ module structure maps $\alpha=(\alpha_i)_{i=0}^\infty$) over an $L_\infty$ algebra $\frakg$, and let $F: \frakg \inftyto \frakg$ be an $L_\infty$ morphism that is homotopic to the identity morphism $\id_{\frakg}$. If a cochain map $\phi:(M, \alpha_0) \to (M, \alpha_0)$ is homotopic to the identity map $\id_M$, then there exists an $L_\infty$ module quasi-isomorphism $\Phi: F^\ast M \inftyto M$ whose zeroth Taylor coefficient is $\phi$.
\end{lemma}

\begin{proof}
The lemma follows by an argument similar to the proof of \cite[Proposition~6.8]{MR3754617}; we include the proof here for completeness.

\textsc{Step~1.} We first consider the case where $F=\id_{\frakg}$. Let $\kappa : M^\bullet \to M^{\bullet -1}$ be a chain homotopy between $\id_M$ and $\phi$, so that $\phi - \id_M = \alpha_0 \kappa + \kappa \alpha_0$. Consider the $L_\infty$ module $\KK[t,dt] \otimes M$ over the $L_\infty$ algebra $\frakg$ obtained by tensoring the cdga $\KK[t,dt]$ (where the scalars have degree $0$, the variable $t$ has degree $0$, and its differential image $dt$ has degree $1$) with the $L_\infty$ module $M$. It is straightforward to see that the evaluation maps $\ev_0 : \KK[t,dt] \otimes M \to M$ and $\ev_1 : \KK[t,dt] \otimes M \to M$ defined by
\[ \ev_0\Big(\sum_{p\geq 0} t^p v_p + dt \sum_{q\geq 0} t^q w_q \Big) = v_0 \quad\text{and}\quad \ev_1\Big(\sum_{p\geq 0} t^p v_p + dt \sum_{q\geq 0} t^q w_q \Big) = \sum_{p\geq 0} v_p \]
are strict $L_\infty$ module morphisms. The chain homotopy $\kappa : M^\bullet \to M^{\bullet -1}$ induces a contraction
\begin{equation}\label{eq:lem:Pyramides-contraction}
\begin{tikzcd}
(M,\alpha_0) \arrow[r, "j", shift left, hook] & \big(\KK[t,dt] \otimes M, d\otimes\id_M+\id_{\KK[t,dt]}\otimes \alpha_0\big) \arrow[l, "\ev_0", shift left, two heads] \arrow[loop, "K", out=3, in=357, looseness=5, distance=2em]
\end{tikzcd}
\end{equation}
with surjection $\ev_0$, injection $j(v) = (1-t) \cdot v + t \cdot \phi(v) + dt \cdot \kappa(v)$, and homotopy
\[ K\Big(\sum_{p\geq 0} t^p v_p + dt \sum_{q\geq 0} t^q w_q \Big) = t \cdot \kappa(v_0) - \sum_{q\geq 0} \frac{1}{q+1} t^{q+1} w_q. \]
According to Proposition~\ref{prop:LooModTransfer-Inj-Surj}~\ref{prop:LooModTransfer-Surj}, there exists an $L_\infty$ module morphism $J : M \inftyto \KK[t,dt] \otimes M$ whose zeroth Taylor coefficient is $j$. The desired $L_\infty$ module morphism $\Phi : M \inftyto M$ is then the composition $\Phi=\ev_1\circ J$. Its zeroth Taylor coefficient is $\ev_1\circ j=\phi$.

\textsc{Step~2.} Assume now that $F:\frakg \inftyto \frakg$ is any $L_\infty$ morphism homotopic to $\id_{\frakg}$. According to Lemma~\ref{lem:PullbackViaHtpMor}, there exists an isomorphism $\Psi: F^\ast M \inftyto M$ of $L_\infty$ modules such that $\Psi_0 = \id_{M}$. Let $\Phi: M \inftyto M$ be the $L_\infty$ module morphism constructed in Step~1. The composition $\Phi \circ \Psi: F^\ast M \inftyto M$ then yields an $L_\infty$ module morphism satisfying $(\Phi \circ \Psi)_0 = \phi$. This completes the proof.
\end{proof}

\section{Formality theorems for trivialized \texorpdfstring{$\ZZ$}{Z}-graded manifolds}
\label{BlackBoxes}

In this section we recall formality theorems for trivialized $\ZZ$-graded manifolds
needed for the proofs of our main theorems.
The material is well-known; see~\cite{MR2062626,MR2004726,MR3522653,MR2304327} for details.

\subsection{Formality theorems for trivialized \texorpdfstring{$\ZZ$}{Z}-graded manifolds}

By a trivialized $\ZZ$-graded manifold we mean a graded manifold $\cV$,
with support $\RR^d$ and algebra of functions
\[ C^\infty(\cV) = C^\infty(\RR^d) \cotimes \widehat{S} (V\dual) ,\]
arising from the graded vector bundle $\RR^d \times V\to \RR^d$.
Here $V$ is a $\ZZ$-graded vector space of finite dimension $r$.
We briefly recall extensions of the well-known formality theorems of Kontsevich and Shoikhet
to the setting of such trivialized $\ZZ$-graded manifolds.
We will often use the shorthand $\cA$ for the algebra of functions $C^\infty(\cV)$.
Then $\Gamma(T_{\cV})$ is the space of derivations of the algebra $\cA$
and $\DD(\cV)$ is the space of differential operators on the algebra $\cA$.
Their respective duals are $\Gamma(T^\vee_{\cV})=\Hom_{\cA}\big(\Der(\cA),\cA\big)$
and $\JJ(\cV)=\Hom_{\cA}\big(\DD(\cV),\cA\big)$.

\subsubsection{Kontsevich's formality morphism}

Essentially, we will follow the presentation of Cattaneo--Felder \cite{MR2304327}.
The spaces of polyvector fields and polydifferential operators on $\cV$ are the direct sums
$\Tpoly{}(\cV)=\bigoplus_{k\geq 0} \Tpoly{k}(\cV)$ and $\Dpoly{}(\cV)=\bigoplus_{k\geq 0} \Dpoly{k}(\cV)$, where $\Tpoly{k}(\cV)$ and $\Dpoly{k}(\cV)$ are defined as in Section~\ref{corniche}.

Given $k$ differential operators $D_1,D_2,\cdots,D_k\in\DD(\cV)$,
we think of the $k$-differential operator
\[ \pshift D_1\otimes\pshift D_2\otimes\cdots\otimes\pshift D_k\in \big(\DD(\cV)[-1]\big)^{\otimes_{\cA} k} \]
as a Hochschild $k$-cochain of the algebra $\cA$.
Indeed, there is a canonical inclusion
\[ \Psi : \Dpoly{k}(\cV)= \big(\DD(\cV)[-1]\big)^{\otimes_{\cA} k} \to
\Hom(\underset{\text{$k$ factors}}{\underbrace{\cA[1]\otimes\cdots\otimes\cA[1]}},\cA) :\]
for all homogeneous $f_1,\cdots,f_k\in\cA$, we have
\begin{multline*}
\Psi\big( \pshift D_1\otimes\pshift D_2\otimes\cdots\otimes D_k \big)
(\nshift f_1\otimes\nshift f_2\otimes\cdots\nshift f_k)
\\
= (-1)^{\sum_{i=2}^k\sum_{j=1}^{i-1}\degree{\pshift D_i}\degree{\nshift f_j}}
\cdot (-1)^{\sum_{i=1}^k\degree{D_i}} \cdot
D_1(f_1)\cdot D_2(f_2)\cdot\cdots\cdot D_k(f_k)
.\end{multline*}

We choose a homogeneous basis for $\RR^d\oplus V$
and denote the induced system of linear coordinate functions on $\cV$ by $(x_1,x_2,\cdots,x_{d+r})$.
Consequently, the elements of the algebra $\cA=C^\infty(\cV)$ are identified
with formal power series in the variables $x_{d+1},\cdots,x_{d+r}$
whose coefficients are smooth functions of the variables $x_1,\cdots,x_d$.
It will be convenient to write $k$-vector fields on $\cV$,
i.e. elements of $S^k_{\cA}\big(\Gamma(T_{\cV})[-1]\big)$
of the form \[ \sum_{i_1,\cdots,i_k}\gamma^{i_1,\cdots,i_k}(x_1,\cdots,x_{d+r})\cdot
\pshift\big(\tfrac{\partial}{\partial x_{a_1}}\big)\odot\cdots
\odot\pshift\big(\tfrac{\partial}{\partial x_{a_k}}\big) \]
as polynomials \[ \sum_{i_1,\cdots,i_k}\gamma^{i_1,\cdots,i_k}(x_1,\cdots,x_{d+r})\cdot p_{i_1}\cdots p_{i_k} \]
of order $k$ in auxiliary variables $p_1,\cdots,p_{d+r}$ with coefficients $\gamma^{i_1,\cdots,i_k}$ in $\cA$.
We declare that the auxiliary variable $p_i$ replacing $\pshift\big(\frac{\partial}{\partial x_i}\big)$
has degree $\degree{p_i}=1-\degree{x_i}$, where $\degree{x_i}$ denotes the degree of $x_i$ in $\cA$.

While the multiplication $\odot_{\cA}$ turns $\Tpoly{}(\cV)$ into a graded associative algebra,
the Schouten bracket (and the zero differential) turn $\frakg=\big(\Tpoly{}(\cV)\big)[1]$ into a dgla.
Likewise, while the cup product $\otimes_{\cA}$ turns
$\Dpoly{}(\cV)$ into a graded associative algebra,
the Gerstenhaber bracket and the Hochschild coboundary turn $\frakh=\big(\Dpoly{}(\cV)\big)[1]$ into a dgla.

\begin{theorem}[Kontsevich \cite{MR2062626}]\label{KontsevichFormality}
There exists a $\GL(\RR^d)\times\GL(V)$-equivariant quasi-isomorphism of $L_\infty$ algebras
\begin{equation}\label{eq:KontsevichFormality}
\kontsevich : \Tpoly{}(\cV)[1] \inftyto \Dpoly{}(\cV)[1]
\end{equation}
from the dgla $\frakg=\big(\Tpoly{}(\cV)[1],0,\schouten{\argument}{\argument}\big)$ of polyvector fields on $\cV$
to the dgla $\frakh=\big(\Dpoly{}(\cV)[1],\hochschild,\gerstenhaber{\argument}{\argument}\big)$
of polydifferential operators on $\cV$,
whose sequence $(\kontsevich_n)_{n=1}^\infty$ of Taylor coefficient maps
\[ \kontsevich_n: \overbrace{\Tpoly{}(\cV)[1]\wedge\cdots\wedge\Tpoly{}(\cV)[1]}^{\text{$n$ factors}}
\to \Dpoly{}(\cV)[1] \]
satisfy the properties listed below:
\begin{enumerate}
\item \label{gingko} The first Taylor coefficient $\kontsevich_1$ is (the suspension of) the Hochschild--Kostant--Rosenberg map
\begin{equation}\label{eq:HKR_vecsp}
\Tpoly{}(\cV) \ni \pshift X_1\odot\cdots\odot\pshift X_n \longmapsto
\frac{1}{n!} \sum_{\sigma\in S_n} \pm\ \pshift X_{\sigma(1)}\otimes\cdots\otimes\pshift X_{\sigma(n)}
\in \Dpoly{}(\cV)
.\end{equation}
\item \label{biloba} For $n\geq 2$, given any collection of vector fields $X_1,X_2,\cdots,X_n$ on $\cV$,
we have \[ \kontsevich_n(X_1,X_2,\cdots,X_n)=0 .\]
\item \label{LinVF-kontsevich}
For $n\geq 2$, given any \emph{linear}\footnote{Written as elements of $\cA[p_1,\cdots,p_{d+r}]$, vector fields are linear w.r.t. the auxiliary variables $p_1,\cdots,p_{d+r}$. A vector field is said to be \emph{linear} if it is linear w.r.t. the variables $x_1,\cdots,x_{d+r}$ as well.}
vector field $X$ on $\cV$
and any collection of polyvector fields $\gamma_2,\gamma_3,\cdots,\gamma_n$ on $\cV$,
we have \[ \kontsevich_n(X,\gamma_2,\gamma_3,\cdots,\gamma_n)=0 .\]
\item \label{iris} Given any finite collection of polyvector fields $\gamma_1,\gamma_2,\cdots,\gamma_n$ on $\cV$,
if $\kontsevich_n(\gamma_1,\gamma_2,\cdots,\gamma_n)$ is a $k$-differential operator on $\cV$,
then $\kontsevich_{n+1}(X,\gamma_1,\gamma_2,\cdots,\gamma_n)$ is a $(k-1)$-differential operator on $\cV$
for every vector field $X$ on $\cV$.
\end{enumerate}
\end{theorem}

We can think of Kontsevich's formality $L_\infty$ quasi-isomorphism
from the dgla of polyvector fields $\frakg$ to the dgla of polydifferential operators $\frakh$ as
\begin{itemize}
\item either a sequence $(\kontsevich_n)_{n\in\NN}$ of maps
$\kontsevich_n:\Lambda^n(\frakg)\to\frakh$ of degree $1-n$,\\
i.e. $\kontsevich_n:\Lambda^n(\Tpoly{}(\cV)[1])\to\Dpoly{}(\cV)[1]$;
\item or a sequence $(\tilde{\kontsevich}_n)_{n\in\NN}$ of maps
$\tilde{\kontsevich}_n:S^n(\frakg[1])\to\frakh[1]$ of degree $0$,\\
i.e. $\tilde{\kontsevich}_n:S^n(\Tpoly{}(\cV)[2])\to\Dpoly{}(\cV)[2]$;
\item or a morphism (of degree $0$) of dg coalgebras
$F_{\tilde{\kontsevich}}:S(\frakg[1])\to S(\frakh[1])$,\\
i.e. $F_{\tilde{\kontsevich}}:S(\Tpoly{}(\cV)[2])\to S(\Dpoly{}(\cV)[2])$.
\end{itemize}

Adopting the second point of view, the $n$-th Taylor coefficient
of Kontsevich's formality morphism is a map of degree $0$
\[ S^n\Big(\big(S_{\cA}(\Der(\cA)[-1])\big)[2]\Big) \to \Big(\bigoplus_{m\geq 0}
\Hom\big(\underset{\text{$m$ factors}}{\underbrace{\cA[1]\otimes\cdots\otimes\cA[1]}},\cA\big)\Big)[2] \]
or, equivalently, a map of degree $0$
\[ S^n\big(S_{\cA}(\Der(\cA)[-1])\big) \to \bigoplus_{m\geq 0}
\Big(\Hom\big(\underset{\text{$m$ factors}}{\underbrace{\cA\otimes\cdots\otimes\cA}},\cA\big)[2-m-2n]\Big) .\]
In other words, Kontsevich's formality is really a collection
(indexed by the pair of numbers $n$ and $m$) of maps
\begin{equation}\label{nmMAPk}
S^n\big(S_{\cA}(\Der(\cA)[-1])\big)\to
\Hom(\underset{\text{$m$ factors}}{\underbrace{\cA\otimes\cdots\otimes\cA}},\cA)
\end{equation}
of degree $2-m-2n$ swallowing $n$ polyvector fields on $\cV$
(written as elements of the graded algebra $\cA[p_1,\cdots,p_{d+r}]$)
and returning a polydifferential operator on $\cV$ (written as a Hochschild $m$-cochain of the algebra $\cA$ of functions on $\cV$).

More explicitly, the ``Taylor coefficients'' of Kontsevich's formality $L_\infty$ quasi-isomorphism are a sequence of maps of the form
\begin{equation}\label{def:Kontsevich}
\kontsevich_n = \sum_{m \geqslant 0} \sum_{\Gamma \in \graphs_{n,m}} w_\Gamma \kontsevich_\Gamma
,\end{equation}
where $\graphs_{n,m}$ denotes the set of admissible graphs of type $(n,m)$,
$w_\Gamma$ is a real number called the Kontsevich weight of the graph $\Gamma$,
and $\kontsevich_\Gamma$ is a map that assembles $n$ polyvector fields
into an $m$-differential operator in a way determined by the graph $\Gamma$.

A directed graph $\Gamma$ is a pair of (finite) sets $V_\Gamma$ and $E_\Gamma$ together with two maps $s,t: E_\Gamma\to V_\Gamma$.
The elements of $V_\Gamma$ are called vertices.
The elements of $E_\Gamma$ are called edges.
Each edge $e \in E_\Gamma$ starts at its source $s(e)\in V_\Gamma$ and ends at its target $t(e)\in V_\Gamma$.

An \emph{admissible graph} of type $(n,m)$ is a directed graph $\Gamma=(V_\Gamma, E_\Gamma)$
with ordered labels on its vertices and edges satisfying the following requirements:
\begin{enumerate}
\item The set of vertices is partitioned into two subsets: $V_\Gamma = V_\Gamma^1 \sqcup V_\Gamma^2$.
The elements of $V_\Gamma^1$ are labeled $1,2,\cdots,n$ and called vertices of the first type or internal vertices.
The elements of $V_\Gamma^2$ are labeled $\bar{1},\bar{2},\cdots,\bar{m}$ and called vertices of the second type or external vertices.
\item For all $e\in E_\Gamma$, $s(e)\in V_\Gamma^1$.
\item For all $e\in E_\Gamma$, $s(e)\neq t(e)$.
\item No two edges have the same source and the same target.
\item For every vertex $k \in V_\Gamma^1$ of the first type,
there is a specified ordering on the set $s^{-1}(k)$ of all edges with source $k$.
As a result, there is an ordering on the set $E_\Gamma$ of all edges compatible with
the labeling of the vertices: the edges with source $1$ are listed in their prescribed order first,
then the edges with source $2$ are listed in their prescribed order, and so on.
\end{enumerate}

We recall the construction of $\kontsevich_\Gamma$ in the graded setting \cite{MR2304327,MR3522653}.

Fix an admissible graph $\Gamma\in\graphs_{n,m}$.
Identifying $S_{\cA}\big(\Der(\cA)[-1]\big)$ with $\cA[p_1,\cdots,p_{d+r}]$,
i.e. writing polyvector fields on $\cV$ as polynomials in the auxiliary variables $p_1,\cdots,p_{d+r}$
with coefficients in $\cA$, we get an inclusion
\[ \big(S_{\cA}(\Der(\cA)[-1])\big)^{\otimes n} \otimes
\underset{\text{$m$ factors}}{\underbrace{\cA\otimes\cdots\otimes\cA}}
\into\big(\cA[p_1,\cdots,p_{d+r}]\big)^{\otimes n+m} .\]
Each edge $e\in E_\Gamma$ induces an operator
\[ H_e = \sum_{k=1}^{d+r} \frac{\partial\argument_{s(e)}}{\partial p_k}
\circ \frac{\partial\argument_{t(e)}}{\partial x_k} \]
of degree $-1$ on $\big(\cA[p_1,\cdots,p_{d+r}]\big)^{\otimes n+m}$,
whose building blocks $\frac{\partial\argument_{s(e)}}{\partial p_k}$
and $\frac{\partial\argument_{t(e)}}{\partial x_k}$ are defined as follows.
Given $n$ polyvector fields $\gamma_1,\cdots,\gamma_n\in\Tpoly{}(\cV)=\cA[p_1,\cdots,p_{d+r}]$
(which we think of as functions of the variables $x_1,\cdots,x_d,p_1,\cdots,p_{d+r}$)
and $m$ functions $f_{\bar{1}},\cdots,f_{\bar{m}}\in C^\infty(\cV)=\cA$
(of the variables $x_1,\cdots,x_d$ exclusively), we set
\[ \frac{\partial\argument_{s(e)}}{\partial p_k}(\gamma_1 \otimes \cdots \otimes \gamma_n \otimes f_{\bar{1}} \otimes \cdots \otimes f_{\bar{m}}) = \pm\, \gamma_1 \otimes \cdots \otimes \frac{\partial}{\partial p_k}(\gamma_{s(e)}) \otimes \cdots \otimes \gamma_n \otimes f_{\bar{1}} \otimes \cdots \otimes f_{\bar{m}} \]
and
\begin{multline*}
\frac{\partial\argument_{t(e)}}{\partial x_k}(\gamma_1 \otimes \cdots \otimes \gamma_n \otimes f_{\bar{1}} \otimes \cdots \otimes f_{\bar{m}}) \\
= \begin{cases}
\pm\, \gamma_1 \otimes \cdots \otimes \frac{\partial}{\partial x_k}(\gamma_{t(e)}) \otimes \cdots \otimes \gamma_n \otimes f_{\bar{1}} \otimes \cdots \otimes f_{\bar{m}} & \text{if}\ t(e) \in V^1_\Gamma, \\
\pm\, \gamma_1 \otimes \cdots \otimes \gamma_n \otimes f_{\bar{1}} \otimes \cdots \otimes \frac{\partial}{\partial x_k}(f_{t(e)}) \otimes\cdots \otimes f_{\bar{m}} & \text{if}\ t(e) \in V^2_\Gamma,
\end{cases}
\end{multline*}
where the signs $\pm$ are determined by the usual Koszul convention.

The map
\[ \kontsevich_\Gamma: \big(S_{\cA}(\Der(\cA)[-1])\big)^{\otimes n} \to
\Hom(\underset{\text{$m$ factors}}{\underbrace{\cA\otimes\cdots\otimes\cA}},\cA) \]
is then defined by
\[ \big(\kontsevich_\Gamma(\gamma_1\otimes\cdots\otimes\gamma_n)\big)
(f_{\bar{1}}\otimes\cdots\otimes f_{\bar{m}})
= \mu \circ (\ev_0)^{\otimes m+n} \circ \Big(\prod_{e\in E_\Gamma} H_e\Big)
(\gamma_1\otimes\cdots\otimes\gamma_n\otimes f_{\bar{1}}\otimes\cdots\otimes f_{\bar{m}}) ,\]
where the operators $H_e$ are composed in the order dictated by the ordering of the set of edges $E_\Gamma$;
the map $\ev_0:\cA[p_1,\cdots,p_{d+r}]\to\cA$ is the evaluation at $p_1=p_2=\cdots=p_{d+r}=0$;
and $\mu:\cA^{\otimes m+n}\to\cA$ is the multiplication of $m+n$ elements in $\cA$.

Since each edge operator has degree $-1$, the map $\tilde{\kontsevich}_\Gamma$ lowers the degree
by the number $\# E_\Gamma$ of edges of the graph $\Gamma$.
As explained by Kontsevich in~\cite{MR2062626},
the weight $w_\Gamma$ is the integral of a differential form $\omega_\Gamma$
over a compact manifold with corners of dimension $2n+m-2$,
the compactification of a certain configuration space.
Since the form $\omega_\Gamma$ arises as the wedge product of $1$-forms
each associated to an edge of the graph $\Gamma$,
the weight $w_\Gamma$ is zero unless the number $\# E_\Gamma$ of edges in the graph $\Gamma$
(i.e. the degree of the differential form $\omega_\Gamma$) is exactly
equal to the dimension $2n+m-2$ of the manifold over which $\omega_\Gamma$ is integrated.
Therefore, the sum $\sum_{\Gamma\in\graphs_{n,m}} w_\Gamma \kontsevich_\Gamma$ is a map of degree $2-m-2n$.

Because the angular $1$-forms corresponding to the edges of the graph $\Gamma$ are multiplied
in the order prescribed by the ordering of the edges in order to obtain $\omega_\Gamma$,
the map $w_\Gamma \kontsevich_\Gamma$ does actually not dependent
on the ordering of the edges of the graph $\Gamma$.
Finally, we note that the sum $\sum_{\Gamma\in\graphs_{n,m}} w_\Gamma \kontsevich_\Gamma$ is,
by construction, invariant under the action of the symmetric group $S_n$ on $\big(S_{\cA}(\Der(\cA)[-1])\big)^{\otimes n}$.
The resulting map
\[ \sum_{\Gamma\in\graphs_{n,m}} w_\Gamma \kontsevich_\Gamma :
S^n\big(S_{\cA}(\Der(\cA)[-1])\big) \to
\Hom\big(\underset{\text{$m$ factors}}{\underbrace{\cA\otimes\cdots\otimes\cA}},\cA\big) \]
of degree $2-m-2n$ is the map \eqref{nmMAPk} we alluded to earlier.

\subsubsection{Shoikhet's formality morphism}

The spaces of polyjets and differential forms on $\cV$ are the direct products
$\cCpoly{}(\cV)=\prod_{p\geq 0}\Cpoly{-p}(\cV)$ and $\cApoly{}(\cV)=\prod_{p\geq 0}\Apoly{-p}(\cV)$, where $\Apoly{-p}(\cV)$ and $\Cpoly{-p}(\cV)$ are defined as in Section~\ref{corniche}.

The duality pairing $\duality{\argument}{\argument}:\JJ(\cV)\times\DD(\cV)\to\cA$
of the space $\DD(\cV)$ of differential operators
and its dual $\JJ(\cV)=\Hom_{\cA}(\DD(\cV),\cA)$, the space of jets, induces a pairing $\pduality{\argument}{\argument}:\Cpoly{-p}(\cV)\times\Dpoly{p}(\cV)\to\cA$ of $p$-jets with $p$-differential operators: for all homogeneous $\xi_1,\cdots,\xi_p\in\JJ(\cV)$
and $u_1,\cdots,u_p\in\DD(\cV)$,
\begin{multline*}
\pduality{\nshift\xi_1\otimes\nshift\xi_2\otimes\cdots\otimes\nshift\xi_p}
{\pshift u_1\otimes\pshift u_2\otimes\cdots\otimes\pshift u_p} \\
=(-1)^{\sum_{i=2}^p \sum_{j=1}^{i-1} \degree{\nshift\xi_i}\degree{\pshift u_j}}
\cdot (-1)^{\sum_{i=1}^p \degree{\xi_i}} \cdot \duality{\xi_1}{u_1} \cdot \duality{\xi_2}{u_2}
\cdot \cdots \cdot \duality{\xi_p}{u_p}
.\end{multline*}

There is a canonical surjection $\Phi: \cA\otimes\underset{\text{$p$ factors}}{\underbrace{\cA[1]\otimes\cdots\otimes\cA[1]}}
\to \Cpoly{-p}(\cV)$ defined by
\begin{multline*}
\pduality{\Phi(f_0\otimes\nshift f_1\otimes\cdots\otimes\nshift p)}{\pshift u_1\otimes\cdots\otimes\pshift u_p}
\\ = (-1)^{\sum_{i=1}^p \sum_{j=1}^i \degree{\nshift f_i}\degree{\pshift u_j}}
\cdot (-1)^{\sum_{i=1}^p \degree{u_i}} \cdot f_0 \cdot u_1(f_1) \cdot u_2(f_2) \cdot \cdots \cdot u_p(f_p)
.\end{multline*}

The Lie derivative and the trivial differential turn $\cApoly{}(\cV)$
into a dgla module over the dgla $\frakg=\Tpoly{}(\cV)[1]$ of polyvector fields.
Likewise, the generalized Lie derivative of Tamarkin--Tsygan and the Hochschild boundary operator
turn $\cCpoly{}(\cV)$ into a dgla module over the dgla $\frakh=\Dpoly{}(\cV)[1]$ of polydifferential operators.

Pulling back the $\Dpoly{}(\cV)[1]$-module structure on $\cCpoly{}(\cV)$
via Kontsevich's $L_\infty$ quasi-isomorphism $\kontsevich: \Tpoly{}(\cV)[1] \inftyto \Dpoly{}(\cV)[1]$,
one obtains an $L_\infty$ module structure $(\alpha_j)_{j\in\NO}$ on $\cCpoly{}(\cV)$
over the dgla $\Tpoly{}(\cV)[1]$. Explicitly, the maps $\alpha_n:\Lambda^n\big(\Tpoly{}(\cV)[1]\big)\otimes\cCpoly{}(\cV)\to \cCpoly{}(\cV)[1-n]$ are given by $\alpha_0=\hochschildb=\liederivative{m_2}$ and $\alpha_n(\gamma_1\wedge\cdots\wedge\gamma_n\otimes\xi)
=\liederivative{\kontsevich_n(\gamma_1\wedge\cdots\wedge\gamma_n)}\xi$ for all $n\geq 1$, $\gamma_1,\cdots,\gamma_n\in\Tpoly{}(\cV)[1]$, and $\xi\in\cCpoly{}(\cV)$. As in Section~\ref{sec:PullBack}, the $L_\infty$ module $(\cCpoly{}(\cV),\alpha_n)$ will be denoted by $\kontsevich^\ast \cCpoly{}(\cV)$.

\begin{theorem}[{Shoikhet \cite{MR2004726} and Dolgushev \cite[Theorem~3]{MR2199629}}]
\label{ShoikhetFormality}
There exists a $\GL(\RR^d)\times\GL(V)$-equivariant quasi-isomorphism
of $L_\infty$ modules over the dgla $(\Tpoly{}(\cV)[1],0,\schouten{\argument}{\argument})$
\begin{equation}\label{eq:ShoikhetFormality}
\shoikhet: \kontsevich^\ast \cCpoly{}(\cV) \inftyto \cApoly{}(\cV)
,\end{equation}
whose sequence $(\shoikhet_n)_{n=0}^\infty$ of Taylor coefficient maps
\[ \shoikhet_n: \overbrace{\Tpoly{}(\cV)[1] \wedge\cdots\wedge \Tpoly{}(\cV)[1] }^{\text{$n$ factors}}
\otimes \cCpoly{}(\cV) \to \cApoly{}(\cV) \]
satisfy the properties listed below:
\begin{enumerate}
\item \label{acer} The zeroth Taylor coefficient map $\shoikhet_0=\hhkr$ is the Hochschild--Kostant--Rosenberg map
\begin{equation}\label{eq:hHKR_vecsp}
\cCpoly{}(\cV) \ni \Phi(a_0\otimes \nshift a_1 \otimes \cdots \otimes \nshift a_p)
\longmapsto a_0\cdot \nshift(da_1) \odot \cdots \odot \nshift(da_p) \in \cApoly{}(\cV)
.\end{equation}
\item \label{palmatum} For $n \geq 1$, given any \emph{linear} vector field $X$;
any collection of polyvector fields $\gamma_2,\gamma_3,\cdots,\gamma_n\in\Tpoly{}(\cV)[1]$;
and any polyjet $\alpha\in\cCpoly{}(\cV)$ on $\cV$, we have
\[ \shoikhet_n(X, \gamma_2, \cdots, \gamma_n; \alpha) = 0 .\]
\item \label{soja} Given any finite collection of polyvector fields $\gamma_1,\cdots,\gamma_n\in\Tpoly{}(\cV)[1]$
and any polyjet $\alpha\in\cCpoly{}(\cV)$,
if $\shoikhet_n(\gamma_1,\cdots,\gamma_n;\alpha)$ is a $p$-form on $\cV$, then
\[ \shoikhet_{n+1}(X,\gamma_1,\cdots,\gamma_n;\alpha) \]
is a $(p+1)$-form on $\cV$ for every vector field $X$.
\end{enumerate}
\end{theorem}

We can think of Shoikhet's quasi-isomorphism of $L_\infty$ modules
over the dgla $\frakg=\Tpoly{}(\cV)[1]$ from the polyjets $\cCpoly{}(\cV)$ to the
differential forms $\cApoly{}(\cV)$ as
\begin{itemize}
\item either a sequence $(\shoikhet_n)_{n\in\NO}$ of maps
$\shoikhet_n:\Lambda^n(\frakg)\otimes\cCpoly{}(\cV)\to\cApoly{}(\cV)$ of degree $-n$,\\
i.e. $\shoikhet_n:\Lambda^n(\Tpoly{}(\cV)[1])\otimes\cCpoly{}(\cV)\to\cApoly{}(\cV)$;
\item or a sequence $(\tilde{\shoikhet}_n)_{n\in\NO}$ of maps
$\tilde{\shoikhet}_n:S^n(\frakg[1])\otimes\cCpoly{}(\cV)\to\cApoly{}(\cV)$ of degree $0$,\\
i.e. $\tilde{\shoikhet}_n:S^n(\Tpoly{}(\cV)[2])\otimes\cCpoly{}(\cV)\to\cApoly{}(\cV)$;
\item or a morphism $F_{\tilde{\shoikhet}}:S(\frakg[1])\otimes\cCpoly{}(\cV)\to S(\frakg[1])\otimes\cApoly{}(\cV)$
of degree $0$ of dg comodules over the dg coalgebra $S(\frakg[1])$, \\
i.e. $F_{\tilde{\shoikhet}}:S(\Tpoly{}(\cV)[2])\otimes\cCpoly{}(\cV)
\to S(\Tpoly{}(\cV)[2])\otimes\cApoly{}(\cV)$.
\end{itemize}

Adopting the second point of view,
the $n$-th Taylor coefficient of Shoikhet's formality morphism is a map
\begin{multline*}
S^n\Big(\big(S_{\cA}(\Der(\cA)[-1])\big)[2]\Big)\otimes\prod_{m\geq 0}\Phi\big(\cA\otimes
\underset{\text{$m$ factors}}{\underbrace{\cA[1]\otimes\cdots\otimes\cA[1]}}\big) \\
\to \prod_{l\geq 0} \Hom_{\cA}\Big(S^l_{\cA}\big(\Der(\cA)[-1]\big),\cA\Big) = \Hom_{\cA}\Big(\bigoplus_{l\geq 0} S^l_{\cA}\big(\Der(\cA)[-1]\big),\cA\Big)
\end{multline*}
of degree $0$ or, equivalently, a collection (indexed by the pair of nonnegative integers $n$ and $m$) of maps
\begin{equation}\label{nmMAPs}
S^n\Big(S_{\cA}\big(\Der(\cA)[-1]\big)\Big)\otimes \cA^{\otimes (1+m)}
\to \Hom_{\cA}\Big(\bigoplus_{l\geq 0} S^l_{\cA}\big(\Der(\cA)[-1]\big),\cA\Big)
\end{equation}
of degree $-m-2n$.

More explicitly, the ``Taylor coefficients'' of Shoikhet's formality $L_\infty$ quasi-isomorphism
are a sequence of maps of the form
\begin{equation}
\shoikhet_n = \sum_{m\geqslant 0} \sum_{\Gamma\in\graphs^\circ_{n,m}}
w^\circ_\Gamma \shoikhet_\Gamma
,\end{equation}
where $\graphs^\circ_{n,m}$ is a certain subset of admissible directed graphs
with $1+n$ internal and $1+m$ external vertices,
$w^\circ_\Gamma$ is a real number called the Shoikhet weight \cite{MR2004726} of the graph $\Gamma$,
and $\shoikhet_\Gamma$ is a map that assembles $n$ polyvector fields
and a Hochschild $m$-chain into a single differential form in a way determined by the graph $\Gamma$.

More precisely, an element of $\graphs^\circ_{n,m}$ is a directed graph $\Gamma=(V_\Gamma,E_\Gamma)$
with ordered labels on its vertices and edges satisfying the following requirements:
\begin{enumerate}
\item The set of vertices is partitioned into two subsets: $V_\Gamma = V_\Gamma^1 \sqcup V_\Gamma^2$.
The elements of $V_\Gamma^1$ are labeled $0,1,2,\cdots,n$ and called vertices of the first type or internal vertices.
The elements of $V_\Gamma^2$ are labeled $\bar{0},\bar{1},\bar{2},\cdots,\bar{m}$ and called vertices of the second type or external vertices.
\item \emph{The internal vertex $0$ is not the target of any edge: $t(e)\neq 0$ for all $e\in E_\Gamma$.}
\item For all $e\in E_\Gamma$, $s(e)\in V_\Gamma^1$
\item For all $e\in E_\Gamma$, $s(e)\neq t(e)$.
\item No two edges have the same source and the same target.
\item For every vertex $k \in V_\Gamma^1$ of the first type,
there is a specified ordering on the set $s^{-1}(k)$ of all edges with source $k$.
As a result, there is an ordering on the set $E_\Gamma$ of all edges compatible with
the labeling of the vertices.
\end{enumerate}

The construction of $\shoikhet_\Gamma$ described in \cite{MR2004726,MR2836399}
can be adapted to the graded setting.

More explicitly, fix a graph $\Gamma\in\graphs^\circ_{n,m}\subset\graphs_{1+n,1+m}$. The map
\[ \shoikhet_\Gamma: \Big(S_{\cA}\big(\Der(\cA)[-1]\big)\Big)^{\otimes n} \otimes \cA^{\otimes (1+m)}
\to \Hom_{\cA}\Big(S_{\cA}\big(\Der(\cA)[-1]\big),\cA\Big) \]
is defined by
\[ \big(\shoikhet_\Gamma(\gamma_1\otimes\cdots\otimes\gamma_n;f_0\otimes f_1\otimes\cdots\otimes f_m)\big)
(\beta) \\ =\pm \big(\kontsevich_\Gamma(\beta \otimes
\gamma_1\otimes\cdots\otimes\gamma_n)\big)(f_0\otimes f_1\otimes\cdots\otimes f_m) \]
for all polyvector fields $\beta,\gamma_1,\cdots,\gamma_n\in\Tpoly{}(\cV)$,
and all functions $f_0,f_1,\cdots,f_m\in C^\infty(\cM)=\cA$.
The Koszul sign $\pm$ in the r.h.s. is due to the polyvector field $\beta$ leapfrogging over the $\gamma$'s and the $f$'s.

Since the map $\kontsevich_\Gamma$ lowers the degree
by the number $\# E_\Gamma$ of edges of the graph $\Gamma$,
so does the map $\shoikhet_\Gamma$.
As explained by Shoikhet in~\cite{MR2004726},
the weight $w^\circ_\Gamma$ is the integral of a differential form $\omega^\circ_\Gamma$
over a compact manifold of dimension $2n+m$ with corners. This manifold is the compactification
of the space of embeddings of $V_\Gamma$ into $D$, the closed disk of radius $1$ centered at $0$ in $\CC$,
that map internal vertices of $\Gamma$ to points in the interior of $D$, external vertices of $\Gamma$ to points
on the boundary circle of $D$, the internal vertex $0\in V^1_\Gamma$ to $0\in\mathring{D}$,
and the external vertex $\bar{0}\in V^2_\Gamma$ to $1\in \partial D$.
Since the form $\omega^\circ_\Gamma$ arises as the wedge product of $1$-forms
each corresponding to an edge of the graph $\Gamma$,
the weight $w^\circ_\Gamma$ is zero unless the number $\# E_\Gamma$ of edges in the graph $\Gamma$
(i.e. the degree of the differential form $\omega^\circ_\Gamma$)
is exactly equal to the dimension $2n+m$ of the manifold over which $\omega^\circ_\Gamma$ is integrated.
Therefore, the sum $\sum_{\Gamma\in\graphs^\circ_{n,m}} w^\circ_\Gamma \shoikhet_\Gamma$
is a map of degree $-m-2n$.

Because the angular $1$-forms corresponding to the edges of the graph $\Gamma$ are multiplied
in the order prescribed by the ordering of the edges in order to obtain $\omega^\circ_\Gamma$,
the map $w^\circ_\Gamma \shoikhet_\Gamma$ does not actually depend on the edge ordering.
Finally, we note that the sum $\sum_{\Gamma\in\graphs^\circ_{n,m}} w^\circ_\Gamma \shoikhet_\Gamma$ is,
by construction, invariant under the action of the symmetric group $S_n$ on $\big(S_{\cA}(\Der(\cA)[-1])\big)^{\otimes n}$.
The resulting map
\[ \sum_{\Gamma\in\graphs^\circ_{n,m}} w^\circ_\Gamma \shoikhet_\Gamma :
S^n\Big(S_{\cA}\big(\Der(\cA)[-1]\big)\Big) \otimes T^{1+m}(\cA) \to
\Hom_{\cA}\Big(S_{\cA}\big(\Der(\cA)[-1]\big),\cA\Big) \]
of degree $-m-2n$ is the map \eqref{nmMAPs} we alluded to earlier.

\subsection{Twisting by a homological vector field}

In this section, we will essentially follow the presentations of Dolgushev \cite{MR2199629,MR2102846}
and Shoikhet \cite{arXiv:math/9812009}.

Let $Q$ be a homological vector field on $\cV$.
Since $Q$ is a Maurer--Cartan element in the dgla $\frakg=\Tpoly{}(\cV)[1]$,
Kontsevich's $L_\infty$ algebra morphism $\kontsevich$
and Shoikhet's $L_\infty$ module morphism $\shoikhet$ can both be twisted by $Q$.
For every $n\in\NN$, the $n$-th ``Taylor coefficient'' $\big(\kontsevich_Q\big)_n$
of the twisted $L_\infty$ morphism $\kontsevich_Q$
(from the dgla $\frakg=\Tpoly{}(\cV)[1]$ twisted by $Q$ to the dgla $\frakh=\Dpoly{}(\cV)[1]$ twisted by
$\kontsevich(Q)$) is defined by the series
\[ \big(\kontsevich_Q\big)_n(\gamma_1, \cdots, \gamma_n) = \sum_{j=0}^\infty \frac{1}{j!} \kontsevich_{n+j}
(\overbrace{Q, \cdots, Q}^{\text{$j$ copies}}, \gamma_1, \cdots, \gamma_n) \]
for all $\gamma_1,\gamma_2,\cdots,\gamma_n\in\Tpoly{}(\cV)[1]$
--- the convergence of the series is guaranteed by Theorem~\ref{KontsevichFormality}~\ref{iris}.
It follows immediately from properties \ref{biloba} and \ref{gingko} of Theorem~\ref{KontsevichFormality} that
\begin{equation}\label{eq:MC-Loo}
\kontsevich(Q) = \sum_{j=1}^\infty \frac{1}{j!}
\kontsevich_{j}(\overbrace{Q,\cdots,Q}^{\text{$j$ copies}})
= \kontsevich_1(Q) = \hkr(Q) = Q
.\end{equation}
Thus $\kontsevich_Q$ is an $L_\infty$ morphism from the twisted dgla
\[ \big(\Tpoly{}(\cV)[1]\big)_Q =
\big(\Tpoly{}(\cV)[1],\schouten{Q}{\argument},\schouten{\argument}{\argument}\big) \]
to the twisted dgla
\[ \big(\Dpoly{}(\cV)[1]\big)_Q =
\big(\Dpoly{}(\cV)[1],\hochschild+\gerstenhaber{Q}{\argument},\gerstenhaber{\argument}{\argument}\big) .\]
For every $n\in\NO$, the $n$-th ``Taylor coefficient'' $\big(\shoikhet_Q\big)_n$
of the twisted morphism $\shoikhet_Q$
of $L_\infty$ modules over the $L_\infty$ algebra $\frakg=\Tpoly{}(\cV)[1]$
from $\cCpoly{}(\cV)$ to $\cApoly{}(\cV)$ is defined by the series
\[ \big(\shoikhet_Q\big)_n (\gamma_1,\cdots,\gamma_n;\alpha) = \sum_{j=0}^\infty \frac{1}{j!}
\shoikhet_{n+j}(\overset{\text{$j$ copies}}{\overbrace{Q,\cdots,Q}}, \gamma_1,\cdots,\gamma_n; \alpha) \]
for all $\gamma_1,\gamma_2,\cdots,\gamma_n\in\Tpoly{}(\cV)[1]$ and $\alpha\in\cCpoly{}(\cV)$.
Since $\cApoly{}(\cV)$ was defined as the direct product
rather than the direct sum of the $\Apoly{-p}(\cV)$'s,
it follows immediately from Theorem~\ref{ShoikhetFormality}~\ref{soja}
that each Taylor coefficient of $\shoikhet_Q$ is well-defined.
Thus we obtain a morphism
\[ \shoikhet_Q: \big(\cCpoly{}(\cV)\big)_Q \inftyto \big(\cApoly{}(\cV)\big)_Q \]
of $L_\infty$ modules over the dgla $\big(\Tpoly{}(\cV)[1]\big)_Q$.
As cochain complexes, we have
\[ \big(\cCpoly{}(\cV)\big)_Q = \big(\cCpoly{}(\cV), \hochschildb + L_Q \big)
\qquad\text{and}\qquad
\big(\cApoly{}(\cV)\big)_Q = \big(\cApoly{}(\cV), L_Q \big) .\]

Consider the standard flat connection $\nabla^\std: \Gamma(T_{\cV}) \times \Gamma(T_{\cV}) \to \Gamma(T_{\cV})$ on $\cV$, defined by setting $\nabla^\std_{\partial_{x_i}} \partial_{x_j} = 0$ for all $i,j \in \{1, 2, \cdots, d+r\}$, and the induced Atiyah and \Aroof\ cocycles $\atiyahcocycle{\std}{(\cV,Q)}$ and $\Ahatcocycle{\std}{(\cV,Q)}$ introduced in Section~\ref{sec:Atiyah}.

The following two propositions follow from
\cite{MR2501196,MR2355492}.
The first proposition was actually conjectured by Shoikhet \cite{MR1854132}. For more details, the reader might want to consult \cite{MR2950766}.

\begin{proposition}\label{prop:DK_vs}
The first Taylor coefficient of the twisted Kontsevich morphism $\kontsevich_Q$ is the composition
\[ (\kontsevich_Q)_1 = \hkr \circ \big(\Ahatcocycle{\std}{(\cV,Q)}\big)^{\sfrac{1}{2}}: \Tpoly{}(\cV) \to \Dpoly{}(\cV) \]
of the action (by contraction) of the square root of the \Aroof\ cocycle on $\Tpoly{}(\cV)$
followed by the Hochschild-Kostant-Rosenberg map (for cochains).
\end{proposition}

\begin{proposition}\label{prop:DKS_vs}
The zeroth Taylor coefficient of the twisted Shoikhet morphism $\shoikhet_Q$ is the composition
\[ (\shoikhet_Q)_0 = \big(\Ahatcocycle{\std}{(\cV,Q)}\big)^{\sfrac{1}{2}} \circ \hhkr: \cCpoly{}(\cV) \to \cApoly{}(\cV) \]
of the Hochschild-Kostant-Rosenberg map (for chains) followed by the action (by multiplication)
of the square root of the \Aroof\ cocycle on $\cApoly{}(\cV)$.
\end{proposition}

\subsection{Compatibility with the calculus structures}

Let $(C,d_C)$ be a dgca.
The calculus operations on $\Tpoly{}(\cV)$, $\cApoly{}(\cV)$, $\Dpoly{}(\cV)$, and $\cCpoly{}(\cV)$
admit straightforward extensions to $C\otimes\Tpoly{}(\cV)$,
$C\cotimes\cApoly{}(\cV)$, $C\otimes\Dpoly{}(\cV)$, and $C\cotimes\cCpoly{}(\cV)$.
Kontsevich's morphism $\kontsevich$ and Shoikhet's morphism $\shoikhet$ also admit natural extensions,
which, given any Maurer--Cartan element $\omega$ in $C\otimes\Tpoly{}(\cV)$, can be twisted:
\begin{gather*}
\kontsevich_\omega: \big(C \otimes \Tpoly{}(\cV)[1]\big)_\omega
\inftyto \big( C \otimes \Dpoly{}(\cV)[1]\big)_{\kontsevich(\omega)} \\
\shoikhet_\omega: \big( C \cotimes \cCpoly{}(\cV)\big)_\omega
\inftyto \big( C \cotimes \cApoly{}(\cV)\big)_{\kontsevich(\omega)}
.\end{gather*}

\subsubsection{Compatibility with $d_{\DR}$ and $B$}

Willwacher proved that Shoikhet's morphism intertwines the Connes--Rinehart operator $B$ on Hochschild chains with the de~Rham differential $d_{\DR}$ on differential forms:

\begin{theorem}[\cite{MR2836399}]\label{prop:CompatibleWithB_local}
For all $n\in\NO$; $\alpha\in C\cotimes\cCpoly{}(\cV)$;
and $\gamma_1,\gamma_2,\cdots,\gamma_n\in C\otimes\Tpoly{}(\cV)[1]$, we have
\[ d_{\DR} \shoikhet_n (\gamma_1,\cdots,\gamma_n;\alpha) = \shoikhet_n (\gamma_1,\cdots,\gamma_n;B\alpha) .\]
\end{theorem}

As a consequence, we have
\begin{equation}\label{eq:cyclic}
d_{\DR} \circ (\shoikhet_\omega)_0 = (\shoikhet_\omega)_0 \circ B .
\end{equation}

\subsubsection{Compatibility with the cup products}

It is well-known that the first Taylor coefficient of the twisted Kontsevich morphism preserves,
up to homotopy, the natural associative multiplications on polyvector fields and polydifferential operators:
\begin{theorem}[\cite{MR1990011,MR2077241,MR1872382,MR2950766,MR3522653,MR3964152}]
\label{prop:HmtOpKont}
Given any Maurer--Cartan element $\omega$ in the dgla $C\otimes\Tpoly{}(\cV)[1]$,
there exists a ($\id_C\otimes\GL(\RR^d)\times\GL(V)$-equivariant) homotopy operator
\[ H_\omega^{\kontsevich}: \big(C\otimes\Tpoly{}(\cV)\big) \times \big(C\otimes\Tpoly{}(\cV)\big)
\to \big(C\otimes\Dpoly{}(\cV)\big) \]
such that
\begin{multline}\label{eq:HtpEqKont}
\pshift\big((\kontsevich_\omega)_1 \big(\nshift(\gamma_1\odot\gamma_2)\big)\big)
- \pshift\big((\kontsevich_\omega)_1(\nshift\gamma_1)\big) \mathbin{\smile} \pshift\big((\kontsevich_\omega)_1(\nshift\gamma_2)\big)
\\ = (d_C + \hochschild + \gerstenhaber{\kontsevich(\omega)}{\argument})
\big(H_\omega^{\kontsevich}(\gamma_1, \gamma_2) \big)
- H_\omega^{\kontsevich}\big( (d_C + \schouten{\omega}{\argument})(\gamma_1), \gamma_2 \big)
\\ - (-1)^{|\gamma_1|} H^{\kontsevich}_\omega\big(\gamma_1, (d_C + \schouten{\omega}{\argument})(\gamma_2) \big),
\end{multline}
for all $\gamma_1,\gamma_2\in C\otimes\Tpoly{}(\cV)$.

Furthermore, given two Maurer-Cartan elements $\omega$ and $\omega'$ in the dgla $C\otimes\Tpoly{}(\cV)[1]$,
whose difference is a \emph{linear} vector field (i.e. $\omega'-\omega\in C\otimes\XX_{\mathrm{linear}}(\cV)$),
the associated homotopy operators are necessarily equal: $H_\omega^{\kontsevich} = H_{\omega'}^{\kontsevich}$.
\end{theorem}

\subsubsection{Compatibility with the cap products}

The cohomologies
\[ H^\bullet\big(C \cotimes \cCpoly{}(\cV), d_C + \hochschildb + \iL_\omega \big) \qquad \text{and} \qquad H^\bullet\big(C \cotimes \cApoly{}(\cV), d_C + \iL_\omega \big) \]
are modules over the associative algebras
\[ H^\bullet\big(C\otimes\Dpoly{}(\cV),
d_C+\hochschild+\gerstenhaber{\omega}{\argument} \big)
\quad\text{and}\quad
H^\bullet\big(C\otimes\Tpoly{}(\cV),
d_C+\schouten{\omega}{\argument}\big) ,\]
respectively.
Making use of the map $(\kontsevich_\omega)_1$,
one can also construe
$H^\bullet\big(C \cotimes \cCpoly{}(\cV), d_C + \hochschildb + \iL_\omega \big)$
as a module over $H^\bullet\big(C\otimes\Tpoly{}(\cV),d_C+\schouten{\omega}{\argument}\big)$.

It is well-known that the zeroth Taylor coefficient of the twisted Shoikhet morphism preserves, up to homotopy, the natural module structures on polyjets and differential forms:

\begin{theorem}[\cite{MR2764337,MR2836111,MR3522653}]\label{prop:HmtOpShoi}
Given any Maurer--Cartan element $\omega$ in $C\otimes\Tpoly{}(\cV)[1]$,
there exists a ($\id_C\otimes\GL(\RR^d)\times\GL(V)$-equivariant) chain homotopy
(lowering the degree by $1$)
\[ H_\omega^{\shoikhet}: \big(C \otimes \Tpoly{}(\cV)\big) \times \big(C \cotimes \cCpoly{}(\cV) \big) \to \big( C \cotimes \cApoly{}(\cV) \big) \]
such that
\begin{multline}\label{eq:HtpEqShoi}
(\shoikhet_\omega)_0 \big(\iI_{\pshift (\kontsevich_\omega)_1(\nshift\gamma)} (\alpha) \big) - \iI_\gamma \big( (\shoikhet_\omega)_0(\alpha) \big) = (d_C + \iL_\omega)\big(H_\omega^{\shoikhet}(\gamma; \alpha) \big) \\
- H_\omega^{\shoikhet}\big( (d_C+\schouten{\omega}{\argument}) (\gamma); \alpha\big)
- (-1)^{|\gamma|} H^{\shoikhet}_\omega\big(\gamma; (d_C + \hochschildb + \iL_{\kontsevich(\omega)})(\alpha) \big),
\end{multline}
for all $\gamma\in C\otimes\Tpoly{}(\cV)$ and $\alpha\in C\cotimes\cCpoly{}(\cV)$.

Furthermore, given two Maurer-Cartan elements $\omega$ and $\omega'$
in the dgla $C\otimes\Tpoly{}(\cV)[1]$,
whose difference is a \emph{linear} vector field
(i.e. $\omega'-\omega\in C\otimes\XX_{\mathrm{linear}}(\cV)$),
the associated homotopy operators are necessarily equal:
$H_\omega^{\shoikhet} = H_{\omega'}^{\shoikhet}$.
\end{theorem}

\section*{Acknowledgments}

We are especially grateful to Boris Shoikhet, who sparked our interest in formality theorems for dg manifolds by drawing our attention to \cite{arXiv:math/9812009}, and to Thomas Willwacher for patiently answering numerous questions throughout various stages of this work.

We also wish to thank many colleagues for insightful discussions and valuable comments: Ruggero Bandiera, Alberto Cattaneo, Sophie Chemla, Zhuo Chen, Vasily Dolgushev, Giovanni Felder, Niels Kowalzig, Dominique Manchon, Marco Manetti, Michael Pevzner, Kyoji Saito, Jonas Schnitzer, Boris Tsygan, Maosong Xiang --- and anyone we may have inadvertently forgotten.

Finally, we would like to thank several institutions for their hospitality during the preparation of this work: the Research Institute for Mathematical Sciences at Kyoto University (Liao, Stiénon, Xu), the Institut Mittag-Leffler (Stiénon, Xu), Pennsylvania State University (Liao), Taiwan's National Center for Theoretical Sciences (Liao, Stiénon, Xu), National Tsing Hua University (Stiénon, Xu), the Institut des Hautes Études Scientifiques (Liao), and the Institut Henri Poincaré (Liao, Stiénon, Xu).

\printbibliography

\end{document}